\documentclass[11pt]{amsart}

\usepackage{epigamath}

\usepackage[english]{babel}

\numberwithin{equation}{section}

\usepackage[shortlabels,inline]{enumitem}
\setlist[enumerate,1]{label={\rm(\arabic*)}, ref={\rm\arabic*}} 

\usepackage{tikz}
\usepackage{tabularx}
\usepackage{ltablex}
\usepackage{booktabs}
\usepackage{frcursive}
\usepackage{xcolor}
\usepackage[matrix,arrow]{xy}
\usepackage{amsmath,amsthm,amscd,amsfonts,amssymb,euscript,latexsym,mathrsfs}
\usepackage{graphics}
\usepackage{tikz-cd}
\usepackage{mathtools}
\usepackage{pdflscape}

\input{xy}
\xyoption{all}

\input{amssym.def}

\theoremstyle{plain}
\newtheorem*{theorem*}{Theorem}
\newtheorem{guess}{Theorem}[section]
\newtheorem{thm}[guess]{Theorem}
\newtheorem{lem}[guess]{Lemma}
\newtheorem{claim}[guess]{Claim}
\newtheorem*{claim*}{Claim}
\newtheorem{prop}[guess]{Proposition}
\newtheorem{Cor}[guess]{Corollary}

\theoremstyle{definition}
\newtheorem{defi}[guess]{Definition}
\newtheorem{notat}[guess]{Notation}

\theoremstyle{remark}
\newtheorem{ex}[guess]{Example}
\newtheorem{rem}[guess]{Remark}

\newcommand{\brem}{\begin{rem}}
\newcommand{\erem}{\end{rem}}

\newcommand{\bth}{\begin{thm}}
\newcommand{\eeth}{\end{thm}}

\newcommand{\gfr}{\mathfrak{g}}

\newcommand{\cU}{\mathcal{U}}

\newcommand{\cO}{\mathcal{O}}

\newcommand{\cE}{\mathcal{E}}
\newcommand{\cR}{\mathcal{R}}
\newcommand{\cH}{\mathcal{H}}

\newcommand{\cF}{\mathcal{F}}

\newcommand{\cA}{\mathcal{A}}

\newcommand{\gG}{\mathfrak{G}}
\newcommand{\cT}{\mathcal{T}}

\newcommand{\cP}{\tt{P}}

\newcommand{\lra}{\longrightarrow}

\newcommand{\hra}{\hookrightarrow}

\newcommand{\ra}{\rightarrow}
\newcommand{\ol}{\overline}

\newcommand{\ms}{\mapsto}

\newcommand{\bvt}{\boldsymbol\Theta}
\newcommand{\bwt}{\boldsymbol\Omega}

\newcommand{\bT}{\boldsymbol\theta}
\newcommand{\bA}{\mathbb{A}}
\newcommand{\RR}{\mathbb{R}}

\newcommand{\ZZ}{\mathbb{Z}}
\newcommand{\GG}{\mathbb{G}}
\newcommand{\QQ}{\mathbb{Q}}

\DeclareMathOperator{\spec}{Spec}
\DeclareMathOperator{\Spec}{Spec}
\newcommand{\beqa}{\begin{eqnarray}}
\newcommand{\eeqa}{\end{eqnarray}}

\newcommand{\eT}{{\mathscr T}}

\newcommand{\qh}{\tiny\text{\cursive H}}

\DeclareMathOperator{\Ad}{Ad}
\DeclareMathOperator{\diag}{diag}
\DeclareMathOperator{\Lie}{Lie}
\DeclareMathOperator{\Res}{Res}
\DeclareMathOperator{\id}{id}
\DeclareMathOperator{\Aut}{Aut}
\DeclareMathOperator{\Hom}{Hom}
\DeclareMathOperator{\charr}{char}
\DeclareMathOperator{\Fract}{Fract}
\DeclareMathOperator{\rank}{rank}
\DeclareMathOperator{\SL}{SL}
\DeclareMathOperator{\Sym}{Sym}
\DeclareMathOperator{\codim}{codim}
\DeclareMathOperator{\Inf}{Inf}
\DeclareMathOperator{\Pic}{Pic}
\DeclareMathOperator{\GL}{GL}
\DeclareMathOperator{\Gal}{Gal}
\DeclareMathOperator{\Irr}{Irr}
\DeclareMathOperator{\ev}{ev}

\DeclareMathOperator{\Coker}{Coker}
\DeclareMathOperator{\Ker}{Ker}
\DeclareMathOperator{\Supp}{Supp}
\DeclareMathOperator{\LG}{LG}

\newcommand{\cm}{\text{\cursive m}}

\newcommand{\ext}{\mathrm{ext}}
\newcommand{\ad}{\mathrm{ad}}
\newcommand{\red}{\mathrm{red}}
\newcommand{\aff}{\mathrm{aff}}

\newcommand{\longhookrightarrow}{\xhookrightarrow{\makebox[10pt]{}}}

\renewcommand{\Im}{\operatorname{Im}}
\renewcommand{\exp}{\operatorname{exp}}

\newcommand{\supth}[1]{\ensuremath{#1^{\mathrm{th}}}}
\newcommand{\suprd}[1]{\ensuremath{#1^{\mathrm{rd}}}}

\EpigaVolumeYear{8}{2024} \EpigaArticleNr{14} \ReceivedOn{July 5, 2022}
\InFinalFormOn{March 8, 2024}
\AcceptedOn{April 5, 2024}

\title{On Bruhat--Tits theory over a higher-dimensional base}
\titlemark{On Bruhat--Tits theory over a higher-dimensional base}

\author{Vikraman Balaji}
\address{Chennai Mathematical Institute, Plot number H1, Sipcot IT Park, Siruseri, Chennai, 603103 India}
\email{balaji@cmi.ac.in}
\author{Yashonidhi Pandey}
\address{Indian Institute of Science Education and Research, Mohali Knowledge city, Sector 81, SAS Nagar, Manauli PO 140306, India}
\email{ypandey@iisermohali.ac.in, yashonidhipandey@yahoo.co.uk, pandeyyashonidhi@gmail.com}

\authormark{V.~Balaji and Y.~Pandey}

\AbstractInEnglish{Let $k$ be a perfect field. Assume that the characteristic of $k$ satisfies certain tameness assumptions. Let $\mathcal O_{n} := k\llbracket z_{1}, \ldots, z_{n}\rrbracket$  and set $K_{n} := \Fract(\cO_{n})$. Let $G$ be an almost simple, simply connected affine Chevalley  group scheme with a maximal torus $T$ and a Borel subgroup $B$.  Given a $n$-tuple ${\bf f} = (f_{1}, \ldots, f_{n})$ of concave functions on the root system of $G$ as defined by Bruhat--Tits, we define {\it {\tt n}-bounded subgroups ${\tt P}_{{\bf f}}\subset G(K_{n})$} as a direct generalization of Bruhat--Tits groups for the case $n=1$. We show that these groups are {\it schematic}; \textit{i.e.} they are valued points of smooth {\em quasi-affine} (resp.\ {\em affine}\,) group schemes with connected fibres and {\it adapted to the divisor with normal crossing $z_1 \cdots z_n =0$} in the sense that the restriction to the generic point of the divisor $z_i=0$ is given by $f_i$ (resp.\ sums of concave functions given by points of the apartment). This provides a higher-dimensional analogue of the Bruhat--Tits group schemes with natural specialization properties. Under suitable assumptions on $k$, we extend all these results for an $(n+1)$-tuple ${\bf f} = (f_{0}, \ldots, f_{n})$ of concave functions on the root system of $G$, replacing $\mathcal O_{n}$ by ${\cO} \llbracket x_{1},\ldots,x_{n} \rrbracket$, 
where $\cO$ is a complete discrete valuation ring with a perfect residue field $k$ of characteristic $p$. In particular, if $x_0$ is the uniformizer of $\mathcal{O}$, then the group scheme is adapted to the divisor $x_0 \cdots x_n=0$. In the last part of the paper, we give applications in characteristic zero to constructing  certain natural group schemes on wonderful embeddings of groups  and also certain families of {\tt 2-parahoric} group schemes on minimal resolutions of surface singularities that arose in an earlier paper by the first-named author.}

\MSCclass{14L15, 14M27, 14D20}

\KeyWords{Bruhat--Tits group scheme, parahoric groups, concave functions, wonderful compactifications, degenerations}

\acknowledgement{The research for this paper was partially funded by the SERB Core Research Grant CRG/2022/000051.}

\dedication{To our parents and Apurv Pandey}

\begin{document}


\removeabove{20pt}
\removebetween{20pt}
\removebelow{20pt}

\maketitle

\begin{prelims}

\DisplayAbstractInEnglish

\bigskip

\DisplayKeyWords

\medskip

\DisplayMSCclass

\end{prelims}


\newpage

\setcounter{tocdepth}{2}

\tableofcontents


\section{Introduction}
Let $G$ be an almost simple, simply connected  affine Chevalley group scheme. Let $\cO$ be a complete discrete valuation ring (DVR) with a perfect residue field $k$. Let  $\mathcal O_{n}:= k\llbracket z_{1}, \ldots, z_{n}\rrbracket$ and let $K_{n} := \Fract~\cO_{n}$; more generally, we consider ${\cO} \llbracket x_{1},\ldots,x_{n} \rrbracket$ when the residue characteristic $p$ satisfies some tameness assumptions; see Section~\ref{tameness}. The object of this paper is to define and study an optimal class of subgroups ${\cP} \subset G(K_{n})$ which we call  {\it $\tt n$-bounded subgroups}. These are natural generalizations of the classical bounded groups and parahoric groups due to F.~Bruhat and J.~Tits,
 who dealt with the case  $n = 1$ (\textit{cf.} \cite{bruhattits, bruhattits1}). 

We show that these {\tt n}-bounded groups are {\it schematic} in the sense of Bruhat and Tits; \textit{i.e.}
there exist smooth, finite-type group schemes over $\spec(\cO_{n})$
 with connected fibres, such that the group of sections gives back the {\tt n}-bounded groups. We will call these group schemes {\tt nBT}-group schemes; see Definition~\ref{nparahoric}. The {\tt nBT}-group schemes we construct are always {\em quasi-affine}, with a large, and perhaps most significant, class being {\em affine}; see Theorem~\ref{multipargrpsch}. This provides natural analogues of the Bruhat--Tits group schemes over $\cO_{n}$ with good specialization properties.  

In \cite[Section~3.9.4]{bruhattits} one finds general remarks on how one could possibly generalize \cite[Theorem 3.8.1]{bruhattits} to a base of dimension $2$.  Somewhat later in the mid-nineties, there was  an approach to Bruhat--Tits buildings over higher-dimensional local fields made by A.~Parshin \cite{parshin} with possible motivation from questions in arithmetic. Our approach has its origins in geometry stemming from the study of degenerations of the stack $\text{Bun}_{G}$ of $G$-bundles over smooth projective curves; see \cite{balaproc}. Unlike Parshin, our aim is not towards a generalization of the notion of a Bruhat--Tits building. Rather, we follow the general spirit of the Bruhat--Tits approach, \textit{cf.} \cite{bruhattits}, namely, to define subgroups of $G(K_{n})$ characterizable  by group schemes over $\spec(\cO_{n})$. 

Recall that a {\sl parabolic vector bundle} in the sense of C.\,S.~Seshadri is a vector bundle with additional data defined on pairs $(X,D)$, where $X$ is a {\sl smooth quasi-projective scheme} and  $D$ a {\sl simple normal crossing divisor}  (see the appendix).
Hitherto, these  have been {\it objects} of study primarily occurring as points in certain moduli spaces.  In the present paper, the notion of a parabolic bundle helps us get a technical tool to address the issue of extending affine group schemes across subsets of codimension bigger than $1$. These connections first came up  in the paper \cite{base}, where under the assumptions of characteristic zero and for group schemes which are generically split (see Definition~\ref{schematicsubgroups}), it was shown, \textit{cf.} \cite[Theorem 5.2.7]{base}, that {\em any {\sf parahoric} Bruhat--Tits group scheme over a complete discrete valuation ring $\cO$ can be realized through ``invariant direct images'', i.e. obtained by taking Galois invariants of Weil restriction of scalars  of a reductive group scheme, on a ramified cover of $\cO$ \textup{(}see Section~\ref{grstuff}\,\textup{)}. It is easily checked \textup{(}see Section~\ref{mixedstuff}\,\textup{)} that this can be made to work under some mild tameness assumptions on residue field characteristics for arbitrary complete DVRs \textup{(}see Section~\ref{tameness}\;\textup{)}. This result and its underlying philosophy form the cornerstone of some of the basic results of this paper}.

Our first step is to define {\tt n-parahoric} subgroups of $G(K_{n})$. Let $T$ be a maximal torus contained in a Borel $B$ of $G$ over $k$. Let $X(T)\,=\,\Hom(T, \GG_{m})$ be the group of characters of $T$ and  $Y(T)\,=\, \Hom(\GG_m, T)$ be the group of all one--parameter subgroups of $T$. 
Let  $\gfr$ denote the Lie algebra. We denote by $\Phi^{+},\Phi^{-} \subset \Phi$  the sets of positive and negative roots with respect to $B$. Let $S= \{\alpha_{1}, \ldots, \alpha_{\ell}\}$ denote the set of simple roots of $G$, where $\ell$ is the rank of $G$.  Let $\alpha^{\vee}$ denote the  coroot  corresponding to $\alpha \in S$. We also choose an {\sl \'epinglage}  or a {\sl pinning} of $G$; \textit{i.e.} in addition to $S$, we also choose for every $r \in S$, a $k$-isomorphism $u_{r}$ from the additive group $\GG_{a}$ to the unipotent group $U_{r}$ of $G$ associated to the root $r$. Let $\{u_{r}\}_{r \in \Phi}$ be their canonical extensions satisfying the Chevalley relations.

Let  $\mathbb S := S \cup \{ \alpha_0 \}$ denote the set of affine simple roots. Let $\cA_{T}$ denote the affine apartment corresponding to $T$. It can be identified with the  affine space ${\mathbb E} :=Y(T) \otimes_{\ZZ} \RR$ together with its origin $0$.

For a point $\theta \in \cA_{T}$ and a root $r \in \Phi$, let $r(\theta) := (r,\theta)$. Set
\begin{equation} \label{mrtheta} 
m_{r}(\theta) := -\lfloor{r(\theta)}\rfloor.
\end{equation} 
Let $\cO$ be a complete discrete valuation ring with residue field $k$, with uniformizer $z$ and field of fractions $K$. Classically (see \cite{bruhattits1}), in terms of generators a parahoric subgroup ${\cP}_{\theta}$ of $G(K)$ is defined  as
\begin{equation}\label{parahoricwithgen}
{\cP}_{\theta}:= \left\langle T(\cO), U_r\left( z^{m_r(\theta)} \cO\right), r \in \Phi \right\rangle.
\end{equation}
More generally, for   an {\em enclosed} bounded subset ${\Omega}$ in the affine apartment $\cA_{T}$ (\textit{cf.} \cite[Section~2.4.6]{bruhattits1} or \cite[Section 3]{courtes}), set 
\begin{equation}\label{mromega}
  m_r(\Omega): = -\left\lfloor \inf_{\theta \in \Omega} r(\theta) \right\rfloor = \left\lceil \sup_{\theta \in \Omega} -r(\theta) \right\rceil. 
\end{equation}
Note that there is a notion of a {\em closure} of a bounded subset, and the function $m_r(\Omega)$ does not change by taking closures. So unless otherwise mentioned, in this article we always work with enclosed bounded subsets. A bounded subgroup ${\cP}_{\Omega}$ of $G(K)$ associated to a bounded subset $\Omega$ of $\cA_T$ is defined as follows:
\begin{equation}
{\cP}_{\Omega}:= \left\langle T(\cO),U_r\left(  z^{m_r(\Omega)} \cO \right), r \in \Phi \right\rangle.
\end{equation}

Following \cite{bruhattits1,bruhattits}, we define a larger class of bounded groups as follows.

\begin{defi}\label{schematicsubgroups}
A subgroup ${\tt Q} \subset G(K)$ will be called  {\sl schematic} if there exists a smooth $\cO$-group scheme $\gG$ with connected fibres and of finite type with $\gG_{K} \simeq G \times K$ such that ${\tt Q} = \gG(\cO)$.
\end{defi}

{\em Throughout this paper we work under the assumption that all of the {{\tt n}-Bruhat--Tits} group schemes we study are {\em connected} and the generic fibre $G_{K}$ is a product $G \times _{\spec(\mathbb{Z})} \spec(K)$; we call this the {\em ``generically split case''}}.
As we have mentioned above, parahoric subgroups are all {\sl schematic}. However, Bruhat and Tits study a larger class of schematic subgroups of $G(K)$ which are defined by {\em concave functions}  (in fact, they work with the class of quasi-concave functions).
 
Let $\tilde{\Phi}:=\Phi \cup \{0 \}$.
Recall, see \cite[Section~6.4.3, p.~133]{bruhattits1},
that a function $f\colon  \tilde{\Phi} \rightarrow \mathbb{R}$ is said to be {\em concave} if whenever $r_{i} \in \tilde{\Phi}$ are such that $\sum_i r_{i} \in \tilde{\Phi}$, then
\begin{equation}\label{concave}
f\left(\sum_{i} r_{i}\right) \leq \sum_{i} f(r_{i}).
\end{equation}

In this paper we  generalize the theory for all concave functions over higher-dimensional bases.  We hasten to add that Bruhat--Tits theory works for general connected reductive groups, while we make the rather simplifying assumptions on $G$ in the spirit of \cite[Section 3.2, p.~52]{bruhattits} that it is almost simple, simply connected and {\it generically split}. Even under these assumptions, the problem seems sufficiently complex.

A natural class of {\tt 2}-parahoric group schemes (see Definition~\ref{nparahoric}) came up naturally  in the work of the senior author while constructing flat degenerations of the stack of $G$-torsors on smooth projective curves, when we allow the curve to degenerate to stable curves (\textit{cf}.~\cite{balaproc}). A special {\tt 2}-parahoric group scheme comes up in the paper \cite{pz} (see Section~\ref{pappas} below and also \cite{lou}). We note that {\tt BT}-group schemes on discrete valuations rings arising from general concave functions have appeared in the papers of J.-K.~Yu \cite{supercusp,yu} and Reeder--Yu \cite{epipelagic}. Further, a very special class has found applications in the papers of Moy--Prasad (see \cite[Section~2.6]{moyprasad})  and Schneider--Stuhler \cite{schstu}.

Even in the most geometric contexts, \textit{i.e.} without the complications of rationality questions, the larger class of concave functions are needed to express \cite{balaproc} or the present paper (see Theorem~\ref{nsdescrip} and Remark~\ref{balajiproctype3}). Indeed, the description of the {\em closed fibres} of the {\tt n}-parahoric group schemes (see Definition~\ref{nparahoric}) over $\cO_{n}$,  at points of depth bigger than $1$, needs perforce the Bruhat--Tits theory of bounded groups associated to concave functions; see Theorem~\ref{thedescription}. The interesting feature is that the concave functions which appear in this situation {\em do not in general arise from bounded subsets of the apartment} but are from the most general class considered by Bruhat and Tits (see Examples~\ref{B2type3}, \ref{G2type3} and \ref{gratifying}, Equation~\eqref{verygratifying} and Remark~\ref{balajiproctype3}). 

 In view of \cite[Section 3.2, p.~52, second paragraph and Section 3.2.15]{bruhattits}, we believe that Bruhat and Tits may have already envisioned the theory over higher-dimensional bases (\textit{cf.}~\cite[Section~3.9.4]{bruhattits}) as well. We revisit this in Section~\ref{revisit}
 
\subsection{Assumptions on the residue field \texorpdfstring{$\boldsymbol{k}$}{k}} \label{tameness}  In this introduction we assume that $k$ is {\em algebraically closed}. We however note that if we assume only perfectness of $k$ together with a few assumptions on the existence of roots of unity (see Section~\ref{charassum}), the results of the present paper continue to hold. 

We will assume throughout that the characteristic $p$ of $k$ is coprime to the order of the centre of $G$ and the coefficients of the highest root.
\begin{itemize}
\item  This assumption suffices when the concave function is of type I, with no further restrictions on the type of the group $G$, except when $G$ is of type $A_{n}$, where  $p$ does not divide $n + 1$.
\item For groups of type $A_{n}, G_{2}, F_{4}, E_{6}$, this suffices for concave functions of type II as well.
\item Let $h_{G}$ be the Coxeter number of $G$.  Let $m(G)$ be the dimension of the minimal faithful representation of $G$. When the concave function is strictly of type III, we further need $p$ to be greater than $h_{G}$ as well as be coprime to $m(G)$. The second condition more precisely means the following: $p$ is   coprime to $7$ for $G_{2}$,  coprime to $13$ for $F_{4}$,  coprime to $3$ for $E_{6}$,  coprime to $14$ for $E_{7}$, coprime to $31$ for $E_{8}$ and finally coprime to $2n+1$  for other classical groups of rank $n$.
\item For the schematization of {\tt n-Moy--Prasad} groups, we refer to Section~\ref{tamp}.
\end{itemize}

\subsection{Higher dimensions and statement of main results}
 We begin with a few key definitions in terms of which the present paper is organized.

\subsubsection*{Types of Concave functions}

\begin{defi}\label{types} \leavevmode
\begin{enumerate}
\item For $j = 1, \ldots, n$, let points $\theta_{j}$ in the apartment  $\cA_{T}$ be given. Let ${\bf f}_{\bT} = \{f_{\theta_{j} } \}_{j = 1}^{n}$ be the set of concave functions associated to them by the function $r \mapsto m_r(\theta)$ (see \eqref{mrtheta}). These will be called {\tt n-concave} functions of {\tt type I}.
\item For $j = 1, \ldots, n$, let {\em bounded subsets} $\Omega_{j} \subset \cA_{T}$ be given.  Let ${\bf f}_{\bwt} = \{f_{\Omega_{j} } \}_{j = 1}^{n}$ be the set of concave functions associated to them by $r \mapsto m_r(\Omega)$ (see \eqref{mromega}). These will be called {\tt n-concave} functions of {\tt type II}. 
\item Let ${\bf f}:=(\ldots, f_{j},\ldots)\colon \Phi \rightarrow \mathbb{R}^n$ be an {\tt n-concave} function on $\Phi$ defined by concave functions $\{f_{j}\}$ (see \eqref{concave}). These will be called {\tt n-concave} functions of {\tt type III}.
\end{enumerate}
\end{defi}

As mentioned earlier, there are examples of concave functions of {\tt type III} which are not of  {\tt type II} (this is stated in \cite[Section~6.4.4]{bruhattits1}; however, see the claim in Example~\ref{B2type3} below for explicit examples) and similarly of {\tt type II} not of {\tt type I}.

\subsubsection*{Bounded groups in higher dimensions}

 We begin by stating our results in the equicharacteristic case.  Recall that $\mathcal O_{n}:= k\llbracket z_{1}, \ldots, z_{n}\rrbracket$, and let $K_{n} := \Fract(\cO_{n})$. For each set of $n$ rational points $\bT = (\theta_{1}, \ldots, \theta_{n}) \in \cA^{n}$, in terms of generators  we define an {\tt n}-parahoric subgroup of $G(K_{n} )$ as  
\begin{equation}\label{moregen0}
{\cP}_{\bT} := \left\langle T\left(\mathcal O_{n}\right), U_r\left( \prod_{1 \leq i \leq n} z_{i}^{m_r(\theta_{i})} \mathcal O_{n} \right), r \in \Phi \right\rangle.
\end{equation}

More generally, for  $\boldsymbol\Omega = (\Omega_{1}, \ldots, \Omega_{n})$, where the $\Omega_{i} \subset \mathbb E$ are {\em enclosed} bounded subsets (\textit{cf.} \cite[Section~2.4.6]{bruhattits1}), we can analogously define {\tt n}-bounded subgroups of $G(K_{n})$:
\begin{equation}\label{moregen}
{\cP}_{\boldsymbol\Omega} := \left\langle T\left(\cO_{n}\right), U_r\left( \prod_{1 \leq i \leq n} z_{i}^{m_r(\Omega_{i})} \cO_{n}\right), r \in \Phi \right\rangle.
\end{equation}
Even more generally, let $f_{i}\colon \Phi \rightarrow \mathbb{R}$ be concave functions, see~\eqref{concave}, and let ${\bf f}:=(\ldots, f_{i},\ldots)\colon \Phi \rightarrow \mathbb{R}^n$ be the {\tt n-concave} function defined by them. To the {\tt n-concave} function ${\bf f}$ we can associate an {\tt n}-bounded subgroup of $G(K_{n})$ as follows:
\begin{equation}\label{evenmoregen}
{\cP}_{\bf f} := \left\langle T\left(\cO_{n}\right), U_r\left( \prod_{1 \leq i \leq n} z_{i}^{f_{i}(r)} \cO_{n}\right), r \in \Phi \right\rangle.
\end{equation}
Similarly, one may define the {\tt n}-bounded Lie subalgebra $\Lie(\cP_{\bf f})$  of $\gfr(K_{n})$.

As in Definition~\ref{schematicsubgroups}, we can call a subgroup ${\tt Q} \subset G(K_{n})$ a {\sl schematic} subgroup if there exists a connected, smooth  $\cO_{n}$-group scheme $\gG$ of finite type with $\gG_{K_{n}} \simeq G \times_{\spec(\mathbb{Z})} K_{n}$ such that ${\tt Q} = \gG(\cO_{n})$.  

We remark that an important aspect of Bruhat--Tits group schemes over discrete valuation rings is the structure of a {\em big cell}   which extends the big cell of the generic fibre (\textit{cf.} \cite[Section~3.1.3]{bruhattits} and  Theorem~\ref{multipargrpsch}\eqref{mt5} below). This aspect also generalizes to the {\tt n}-Bruhat--Tits group schemes.  

We now state our main theorems in the equicharacteristic case. This corresponds to the layout of the paper. 

\begin{thm}\label{multipargrpsch}  Let $p$ satisfy the hypothesis of Section~\ref{tameness} \textup{(}or more generally those of Section~\ref{charassum}\,\textup{)}. 
For an {\tt $n$-concave} function ${\bf f}$ of  each type in Definition~\ref{types}, there exists a smooth group scheme $\gG_{\bf f}$ of finite type on $\mathbb A^{n}_{k}$ with {\em connected fibres}, whose restriction to $\spec(\cO_{n})$ comes with an isomorphism $h_{K_{n}}$ of the generic fibre $\gG_{K_{n}}$ with $G_{K_{n}}$ and such that the following hold:
\begin{enumerate}
\item The group scheme $\gG_{\bf f}$ is {\tt affine} when ${\bf f}$ is of\, {\tt type I}, and in general it is {\tt quasi-affine} in {\tt types II and III}.
\item \label{mt2} The {\tt n}-bounded subgroup $\cP_{\bf f}$ is {\sl schematic}. More precisely, we have an identification $\gG_{\bf f}(\cO_{n}) = \cP_{\bf f}$. 
\item  We have a natural isomorphism of\, $\Lie(\gG_{\bf f})(\cO_{n})$ with $\Lie(\cP_{\bf f})$.
\item \label{mt4} Let us denote the canonical images under $h_{K_{n}}$ of the maximal torus $T_{K_{n}}$ and the root groups $
\{U_{r, K_{n}}\}_{r \in \Phi}$ in the generic fibre $\gG_{K_{n}}$ by the same notation.

There exist closed subgroup schemes $\{\mathfrak U_{r, \bf f} \subset \gG_{\bf f}\}_{r \in \Phi}$ \textup{(}resp.\ $\cT_{\bf f} \subset \gG_{\bf f}$\textup{)} on $\spec(\cO_{n})$, with generic fibre isomorphic to the root groups $\{U_{r, K_{n}}\}_{r \in \Phi}$ \textup{(}resp.\ $T_{K_{n}}$\textup{)}. 
\item \label{mt5} There exists a {\em big cell}   $\mathfrak B_{\Phi}$ of\, $\gG_{\bf f}$ such that for any choice of total ordering on $\Phi^{+}$ (resp.\ $\Phi^{-}$), the morphism induced by  multiplication
\begin{equation}
\left(\prod_{r \in \Phi^{+}} \mathfrak U_{r, \bf f}\right) \times \cT_{\bf f} \times \left(\prod_{r \in \Phi^{-}} \mathfrak U_{r, \bf f}\right) \lra \gG_{\bf f}
\end{equation} 
is an open immersion with image $\mathfrak B_{\Phi}$. Moreover, $\mathfrak B_{\Phi}$ restricts to  the {\em big cell} $\mathfrak B_{K_{n}}$ of the generic fibre $\gG_{K_{n}}$ over $\spec(K_{n})$.
\item \label{mt6} More generally, on a base which is a smooth quasi-projective variety with a divisor with simple normal crossings, given a data of type I concave functions at the height $1$ primes associated to the components of the divisor, in Theorem~\ref{Artin-Weil-Kawamata}, it is shown that there exist a smooth affine group scheme with connected fibres on $X$ together with a big cell structure which interpolates this datum on the divisors.

\item \label{mt7} Let ${\bf f}$ be a special $n$-concave function of type III such that each term $f_{i}$ is a sum of concave functions of {\tt type~I}. Then the associated group scheme $\mathfrak{G}_{\bf f}$ is in fact {\tt affine}; see Corollary~\ref{moreaffineness}.
\end{enumerate}
The pair $(\gG_{\bf f},h_{K_{n}})$ is uniquely determined up to
a unique  isomorphism, by  properties \eqref{mt4} and \eqref{mt5}.  
\end{thm}

\begin{defi}\label{nparahoric} The group scheme $\gG_{\bf f}$ is called the {\tt nBT}-group scheme associated to an {\tt n}-concave function ${\bf f}$.  An {\tt n}-parahoric group scheme is an {\tt nBT}-group scheme associated to an  {\tt n}-concave function of {\tt type I}. \end{defi}

\begin{thm}\label{thedescription}
The group scheme $\gG_{\bf f}$ associated to the {\tt n-concave} function ${\bf f}$ has the following structure at points of depth bigger than $1$:
\begin{enumerate}[label={\rm(\alph*)}, ref={\rm\alph*}] 
\item\label{grat1-a} Let ${\tt I} \subset \{1, \ldots, n\}$ be a non-empty subset. For a general point $\xi \in \cap_{i \in {\tt I}} H_{i}$, where the $H_{i}$ are the coordinate  hyperplanes in $\spec(\cO_{n})$, the fibre 
$\gG_{{\bf f},\xi}$ is isomorphic to the closed fibre of the Bruhat--Tits group scheme $\gG_{{\tt f}_{\tt I}}$ on $\spec(\cO)$ associated to the {\tt 1-concave} function ${\tt f}_{\tt I}\colon \Phi \to \mathbb R$ given by $r \mapsto \sum_{i \in {\tt I}} f_{i}$.
\item\label{grat1-b} More precisely, let $\bA_{\tt I} \subset \bA^{n}$ be the ``subdiagonal'' obtained by setting the coordinates $x_{i} = t$ for $i \in \tt I$. Let $\spec(A_{\tt I})$ be the local ring at the generic point of the divisor $t=0$ in $\bA_{\tt I}$. Then the restriction of $\gG_{\bf f}$ to  $\spec(A_{\tt I})$ is isomorphic to the Bruhat--Tits group scheme $\gG_{{\tt f}_{\tt I}}$ associated to the concave function  ${\tt f}_{\tt I}$.
\end{enumerate}
\end{thm}

Let $\mathcal{O}$ be a complete DVR in mixed characteristics.
Extending the methods for group schemes over $\mathbb A^{n}_{k}$ to the case $\mathbb{A}^{n}_{\cO}$, in Section~\ref{mixedstuff} we prove the following. 

\bth\label{multipargrpschmixed} For an {\tt $(n+1)$-concave} function ${\bf f}=(f_0,\ldots,f_n)$, where $f_0$ is prescribed at  the uniformizer of\, $\mathcal{O}$, the main theorems, Theorems~\ref{multipargrpsch}  and~\ref{thedescription},  hold over $\mathbf{A}_{\cO} = {\mathbb A}^{n}_{\cO}$ under the tameness assumption of Section~\ref{tameness} \textup{(}in fact under the assumptions in Section~\ref{charassum}\;\textup{)}. 
\eeth
 
\subsection{An outline of the strategy of proof}
We now outline briefly the strategy of the proof in this paper. The general strategy is an {\em induction on the order of complexity} of concave functions on root systems. The simplest one is a function $f_{\theta}$, see \eqref{mrtheta}, which is a concave function of type I; the next in the order is the concave function $f_{\Omega}$ (see  \eqref{mromega}),
associated to a bounded subset $\Omega$ in the apartment, \textit{i.e.} of type II. The third in the order of complexity is the general concave function $f$ of type III which need not arise as an $f_{\Omega}$.

In other words,  we build the family of {\tt nBT}-group schemes $\{\gG_{\bf f}\}$ on $\spec(\cO_{n})$ associated to an {\tt n-concave} function ${\bf f}\colon\Phi \to {\mathbb R}^{n}$ in three tiers. This is carried out by  a bootstrapping process. The first step, in Section~\ref{schematization1}, is to build the {\tt n}-parahoric group scheme (see Definition~\ref{nparahoric}) associated to a collection of $n$  points in the affine apartment, \textit{i.e.} an {\tt n-concave} function $ {\bf f} = \{f_{\theta_{j}}\}$ of {\tt type I}. This case falls into a general method of proof where the ideas arising from the notion of parabolic vector bundles play a key role. It has its origins in the paper \cite{base}, and the end result is somewhat surprising in that, unlike what Bruhat--Tits expect, namely only a {\em quasi-affine} group scheme (see \cite[Section~3.9.4]{bruhattits} and Section~\ref{revisit} below), one in fact gets an {\em affine} {\tt n}-parahoric group scheme on $\spec(\cO_{n})$ along with its  cell structure. The pair of the {\tt n}-parahoric group scheme and its cell structure become the base for the subsequent constructions. We then naturally move on to the case when instead of points in the apartment, we are given a collection of $n$ bounded subsets in the affine apartment and so work with an $n$-tuple of {\em concave} functions arising from these bounded subsets. The construction (in Section~\ref{schematization2}) of the corresponding {\tt nBT}-group scheme  on $\spec(\cO_{n})$ builds on the construction of an {\tt n}-parahoric group scheme. We assume in Section~\ref{schematization2}  that $G$ is one of $A_{n}, G_{2}, F_{4}$ or $E_{6}$, where results from \cite{bruhattits3} and \cite{ganyu1, ganyu2} play the important role of reducing  schematic constructions over DVRs for enclosed bounded subsets to their extremal vertices. The proof here is somewhat involved, especially in extending  the cell structure and the schematic group law across subsets of codimension $2$. 
The end result is that, much as expected, we get a {\em quasi-affine} {\tt nBT}-group scheme associated to the {\tt n-concave} functions of {\tt type II}. Finally, in Section~\ref{schematization3}  we turn our attention to the case of  $G$ with no constraints and also a general {\tt n-concave} function. This case is then reduced to the previous case, \textit{i.e.} of {\tt n-concave} functions of {\tt type~II}. This reduction follows from the {\em not so obvious} observation, see Proposition~\ref{btremarque}, that an {\em optimal concave function} on groups of type $A_{n}$ always arises as one coming from a bounded subset of the apartment. This remark together with the power of the main results of \cite{bruhattits} help us to realize the {\tt nBT}-group scheme  associated to an $n$-tuple of concave functions; these group schemes are {\em a fortiori} quasi-affine as expected.  A surprise, and a possible opening, is offered by a special class of concave functions of type III, which occur as sums of concave functions of type I. Here, our methods yield the {\em affineness} of the group scheme; see Corollary~\ref{moreaffineness}. This leads us to a natural question, namely ``Characterise $n$-tuples of optimal concave functions for which the associated {\tt nBT}-group schemes are  {\em affine}.''

Recall that in the equicharacteristic case when the residue field characteristic satisfies mild tameness assumptions, by \cite[Theorem 5.2.7]{base}, all parahoric subgroup schemes are recovered by the invariant direct image functor (\textit{i.e.} obtained by taking Galois invariants of Weil restriction of scalars). Following \cite{base} closely, in Section~\ref{schematizationmp} we show that Moy--Prasad group schemes over DVRs which generalize parahorics can be obtained by the invariant direct image together with the simplest types of dilatations. Over higher dimensions, this generalizes suitably  using dilations results of \cite{mayeux}. We use results of Bruhat--Tits \cite{bruhattits} to show in Section~\ref{mixedstuff} that the generalization of \cite[Theorem 5.2.7]{base} to mixed characteristics reduces to a verification on the big cell which was done in Section~\ref{schematizationmp}, following \cite{base} closely.

\subsection{Other related results and proof strategies} In \cite{land}  Bruhat--Tits group schemes are constructed by using extension of {\em birational group laws} and the theorem of Artin--Weil; see \cite[Expos\'e XVIII]{sga3}. The possibility of such an approach is briefly mentioned in \cite[Section 3.1.7]{bruhattits}. In the very recent paper \cite{lou}, smooth and separated group schemes over arbitrary Noetherian bases has been constructed; see \cite[Theorem 3.2.5]{lou}. This last result assumes that the {\it schematic root datum} extends to the base ring $A$. However, in general proving the key properties such as {\em quasi-affineness} of the group schemes does not seem immediate by these approaches. In the present paper, for $k\llbracket z_{1}, \ldots, z_{n}\rrbracket$ or $\mathcal{O}\llbracket z_{1}, \ldots, z_{n}\rrbracket$, where $\mathcal{O}$ is a discrete valuation ring, in the course of our proofs for each of the three types of concave functions, we actually show that the schematic root datum extends. 

In the context of the present paper on higher-dimensional bases, the first construction goes back to \cite[Section~3.9.4]{bruhattits}, where for $A=k[x,y]$ or $A=A_1[[z]]$ for $A_1$ a Dedekind ring, a {\tt 2BT}-group scheme is constructed by generalising \cite[Theorem~3.8.1]{bruhattits}.  

In the special setting of a $2$-concave function ${\bf f} = (0,f)$ over $A=A_1[[z]]$, for $A_1$ a Dedekind ring, we refer to \cite{pz} and Section~\ref{pappas} below. The very recent paper \cite{lou} also gives a completely different  proof of the key {\em affineness result} of these group schemes over $A=A_1[[z]]$; further, it goes on to construct and study affine Grassmannians for these group schemes. 

\subsection{Other directions} In a slightly different direction, recall that Moy--Prasad groups  (see \cite[Section~2.6]{moyprasad}) can be realized as schematic subgroups arising from {\tt BT}-group schemes arising from a small variant of the concave functions in \cite{bruhattits1}. As a natural extension of our approach, in Section~\ref{schematizationmp} we define {\tt n}-Moy--Prasad groups and the corresponding schematization.  Their schematization over discrete valuation rings was derived as a consequence of a general approach by Yu in \cite{yu}. To adapt these groups to our theory, we need to realize the Moy--Prasad groups as {\em invariant direct images}, see Notation~\ref{weilresaspushforward}, of certain group schemes from ramified covers (see \cite{base}); see Section~\ref{schematizationmp}. This involves generalising the notion of a {\em unit group} to one corresponding to the variant of concave function mentioned above. As a consequence, we show that {\tt n}-Moy--Prasad group schemes are affine  just like the group scheme
$\gG_{\bf f}$ when ${\bf f}$ is of {\tt type I}.

As further applications of this approach, we construct higher {\tt BT} {\sl  group schemes}  on certain basic spaces. More precisely, we construct and describe
\begin{itemize}
\item an $\ell${\tt{BT}}-group scheme on the De Concini--Procesi wonderful compactification  ${\bf X}$ (see \cite{decp}) of the adjoint group $G_{\ad}$, where $\ell = \rank(G)$; 
\item an $(\ell+1)${\tt{BT}}-group scheme on the loop ``wonderful embedding'' ${\bf X}^{\aff}$ of the adjoint affine Kac-Moody group $G^{\aff}_{\ad}$, constructed by P.~Solis; see \cite{solis}~-- the novelty is that the base here is an {\em ind-scheme}; 
\item a family of {\tt 2BT}-group schemes on the minimal resolution of singularities of normal surface singularities in the context of \cite{balaproc}.
\end{itemize} 

\subsection*{Frequently occurring notation and conventions}

Throughout this paper we follow the classical notation and conventions as in the foundational papers \cite{bruhattits, bruhattits1} of Bruhat--Tits and as in  the  treatise  \emph{N\'eron Models} \cite{blr} by  S. Bosch, W. Lutkebohmert and M. Raynaud.

\renewcommand{\arraystretch}{1.1}%
\begin{tabularx}{\textwidth}{>{\hsize=.8\hsize}X>{\hsize=1.2\hsize}X}
  \toprule
  $\Phi$ & Root system relative to $T$  \\
  $\cA$ & The affine apartment of $T$ \\
  ${\bf f} = (f_{1}, \ldots, f_{n})$ & {\tt n-concave}-function on $\Phi \cup \{0\}$, see Definition~\ref{types} \\
  $\bwt = (\Omega_{1}, \ldots, \Omega_{n})$ & {\tt n-tuple} of bounded subsets of $\cA$ \\
  $\bvt=(\theta_{1},\ldots, \theta_{n})$ & {\tt n}-tuple of points in the affine apartment \\
  $\bf A$ & The affine space ${\mathbb A}^{n}_{k}$ up to Section~\ref{schematizationmp} or $\mathbb{A}^{n}_{\cO}$ as in Section~\ref{extension1} \\
  $\bf A_{0}$ & The complement of the coordinate hyperplanes in $\bf A$  \\
  $(X,D)$ & Smooth scheme with a simple normal crossing divisor \\
  $X', U$   & Open subschemes of $X$ whose codimension is at least $2$, also called big open subsets \\
  ${\cP}_{\bf f}$ (see \eqref{evenmoregen}) & {\tt n}-bounded groups associated to an {\tt n}-concave function $\bf f$, see Definition~\ref{types} \\
  ${\cP}_{\boldsymbol\Omega}$ (see \eqref{moregen}) & {\tt n}-bounded groups associated to an {\tt n}-bounded subset $\bwt$ \\ 
  ${\cP}_{\bf f}^{\diag}$ (see Definition~\ref{allunifequaldef} and \eqref{allunifequal}) & The specialization to the diagonal \\
  $\gG_{\bf f}$ (resp.\ $\gG_{\bwt}$, $\gG_{\bT}$)  & {\tt n}-Bruhat--Tits group scheme associated to an {\tt n-concave} function $\bf f$ (resp.\ {\tt n-bounded} subset $\bwt$, {\tt n}-points $\bT$) \\
  $\Gamma$ & Galois group of a Kawamata cover \\
  $p^{\Gamma}_{*}$ (see Notation~\ref{weilresaspushforward}) & The {\em invariant direct image}, \textit{i.e.} Weil restriction and $\Gamma$-invariants \\
  $\cR$ &  Lie algebra bundle \\
 \bottomrule
\end{tabularx}

\subsection*{Acknowledgments}This paper had its origins in \cite{arXiv-version},  which however had an inaccuracy in the description of the group scheme at points of codimension higher than $1$. This was pointed out to us by Jochen Heinloth. We thank him for this. The content of the former manuscript is now subsumed in Section~\ref{app1} of this paper. We also thank  Michel Brion and Patrick Polo for their numerous questions and comments on the paper. Finally, we thank the anonymous referee for their comments and questions, which have led to a great improvement of the manuscript.

\part{The bounded groups}

\section{Towards the  {\tt n}-parahoric Lie algebra bundle}
The first two subsections are in vigour throughout this article. The remaining ones recall a ``loop'' approach to parahoric groups for application in Sections~\ref{constructionR} and~\ref{app1}.

\subsection{Lie data of \texorpdfstring{$\boldsymbol{G/k}$}{G/k}}\label{liedata}  In this subsection we assume the group-theoretical notions of $G$ already introduced in the introduction and give only those that we need further in this article. Let $\mathbf{a}_0$ be the unique Weyl alcove of $G$ whose closure contains $0$ and which is contained in the dominant Weyl chamber corresponding to $B$.    Under the natural pairing between $Y(T) \otimes_{\ZZ} \QQ$ and $X(T) \otimes_\ZZ \QQ$, the integral basis elements dual to $S$ are called the fundamental co-weights $\{\omega^{\vee}_{\alpha}\mid \alpha \in S\}$. Let $c_{\alpha}$ be the coefficient of $\alpha$ in the highest root. The vertices of the Weyl alcove $\mathbf{a}_0$ are exactly $0$ and 
\beqa\label{alcovevertices}
\theta_{\alpha} := {{\omega^{\vee}_{\alpha}} \over {c_{\alpha}}},\quad \alpha \in S.
\eeqa
We will call the lattice generated by the $\theta_{\alpha}$ the Tits lattice.

Let $\SL(m(G))$ be the minimal faithful representation of $G$. We fix a maximal split torus $T_{\SL}$ of $\SL(m(G))$. Let $\cA_{T_{\SL}}$ denote its apartment. We fix an alcove $\mathbf{a}$ in $\cA_{T_{\SL}}$.

\begin{defi} \label{dtheta} For any $\theta \in Y(T) \otimes \mathbb Q$,                let $d_{\theta}$ be the least positive integer such that $d_{\theta}\cdot\theta$ lies in $Y(T)$. For $\theta \in Y(T_{\SL}) \otimes \mathbb{Q}$, let $d_{\theta}$ be defined similarly. \end{defi}

Thus, if $e_\alpha$ is the order of $\omega^\vee_\alpha$ in the quotient of the co-weight lattice by $Y(T)$, it follows that for $\theta_{\alpha}$, the number $d_{\alpha}:= d_{\theta_{\alpha}}$ is 
\begin{equation} \label{dalpha}
  d_{\alpha} = e_\alpha \cdot c_{\alpha}.
\end{equation}

An affine simple root $\alpha \in \mathbb S$ may be viewed as an affine functional on $\cA_T$. Any non-empty subset $\mathbb I \subset \mathbb S$ defines the facet $\Sigma_\mathbb I \subset \ol{\mathbf{a}_0}$ where exactly the $\alpha$ not lying in the subset $\mathbb I$ vanish. So the set $\mathbb S$ corresponds to the interior of the alcove, and the vertex $\theta_{\alpha}$ of the alcove corresponds to $\alpha \in \mathbb S$. Conversely, any facet $\Sigma \subset \ol{\mathbf{a}_0}$ defines a non-empty subset $\mathbb{I}_{\Sigma} \subset \mathbb S$. For $\emptyset \neq \mathbb I \subset \mathbb S$, the barycenter of $\Sigma_\mathbb I$ is given by  
 \begin{equation} \label{Ibarycenter}
 \theta_\mathbb I:= \frac{1}{|\mathbb I|} \sum_{\alpha \in \mathbb I} \theta_{\alpha}.
 \end{equation}
 
 \subsection{Assumptions on the residue field \texorpdfstring{$\boldsymbol{k}$}{k} and points in the apartments compatible with \texorpdfstring{$\boldsymbol{k}$}{k}} \label{charassum}
 The following assumptions will be in vigour throughout the article. We assume $k$ is perfect to be able to apply \cite{bruhattits}. Let $p:=\charr(k)$.
 
{\it Characteristic Assumptions.} These are needed for the existence of a Kawamata cover (see Section~\ref{kawa}). We will assume that  $p$ is coprime to $d_{\alpha}$; see Equation~\eqref{dalpha}. Notice that since $G$ is simply connected, this assumption is implied by Section~\ref{tameness}. Indeed, since $Y(T)$ equals the coroot lattice, $Z_G$ identifies with the quotient of the co-weight lattice by $Y(T)$. Further, whenever we work with a finite set of rational points $\{\theta_1,\ldots, \theta_n\}$ of $\mathcal{A}_T$ or $\mathcal{A}_{T_{\SL}}$, we assume that for $1 \leq i \leq n$, $p$ is coprime to the integers $d_{\theta_i}$; see Definition~\ref{dtheta}. Let $h_{G}$ denote the Coxeter number of $G$.  Let $m(G)$ denote the dimension of the minimal faithful representation of $G$. When the concave function is strictly of type III, we further need $p > h_{G}$ as well as $p$ to be coprime to $m(G)$. 

{\it Assumptions on the existence of primitive roots of unity.} These are needed to make Section~\ref{grstuff} work. The residue field $k$ is such that for every facet $\Sigma$ of $\mathbf{a}_0$  (see Section~\ref{liedata}), there exists a rational point $\theta \in \Sigma$ such that $d_{\theta}$ (see Definition~\ref{dtheta}) is coprime to $\charr(k)$ and $k$ contains the primitive $\supth{d_{\theta}}$ roots of unity.

\begin{itemize}
\item  For a general $G$, this suffices when the concave function is of type I.
\item For $A_{n}, G_{2}, F_{4}, E_{6}$, this suffices for type II as well.
\item When the concave function is strictly of type III, in addition to  the condition for $\mathbf{a}_0$, we need to assume the analogous condition for $\mathbf{a}$ (see Section~\ref{liedata}). 
\end{itemize}

Whenever we work with a finite set of rational points $\{\theta_1,\ldots, \theta_n\}$ of $\mathcal{A}_T$ or $\mathcal{A}_{T_{\SL}}$, we assume that for $1 \leq i \leq n$, $k$ contains the primitive $\supth{d_{\theta_i}}$ roots of unity.

We observe that any facet $\Sigma \subset \cA_T$ 
contains a rational one-parameter subgroup $\theta$ of $T$ which is a convex combination of the vertices of $\Sigma$ in a way that for $d \in \mathbb{N}$ coprime to $p$, the multiple $d \theta$ becomes a one-parameter subgroup of $T$. We see this as follows. Since the alcove $\mathbf{a}_0$ is a fundamental domain for the action of the affine Weyl group, we may suppose that $\Sigma$ is a facet $\Sigma_{\mathbb{I}}$ of $\mathbf{a}_0$. Say $\mathbb{I}=\{\theta_{0},\ldots, \theta_{j} \}$ are the vertices of $\Sigma_{\mathbb{I}}$ and $\theta_{0}=\theta_{\alpha}$. Then since the $d_{\alpha}$ defined in \eqref{dalpha} are always strictly greater than $1$, it follows that we may take 
\begin{equation}
\theta= \frac{\theta_{0} + \sum_{i=1}^j \left(d_{\alpha}^i - d_{\alpha}^{i-1}\right) \theta_{i}}  {d_{\alpha}^j}.
\end{equation}

We remark that if $L$ denotes the LCM of the coefficients of the highest root and $\ell$ is the rank of $T$, then for type I, it suffices that $k$ has $\big( |Z_G| L \big)^{\ell}$ primitive roots of unity. For type III, $k$ must further have primitive $\supth{({m(G)^{m(G)}})}$ roots of unity.

\subsection{Loop groups and their parahoric subgroups}  \label{loopgpsec}
{\sl We have included this section, where we introduce  the language of loop groups to describe parahoric groups. We however add that this terminology is restricted to Sections~\ref{constructionR} and~\ref{app1} of this paper.}

  The loop group $\LG$ is the group functor which associates to  a $k$-algebra $R$ the group $G(R(\!(z)\!))$; this is representable by an ind-scheme over $k$. Similarly, the loop Lie algebra functor $L \gfr$ is given by $L \gfr(R)=\gfr(R(\!(z)\!))$. We can similarly define the positive loops (also called {\it jet groups} in the literature) 
$L^+(G)$ to be the subfunctor of $\LG$ defined by $L^+(G)(R) := G(R\llbracket z \rrbracket)$. The positive-loop construction extends more generally to any group scheme $\mathcal{G} \rightarrow \spec(A)$ and any vector bundle $\mathfrak{P} \rightarrow \spec(A)$ whose sheaf of sections carries a Lie bracket.

Let $L^{\ltimes}G= \GG_m \ltimes \LG$,  where the {\it rotational torus} $\GG_m$ acts on $\LG$ by acting on the uniformizer via the {\it loop  rotation action} as follows: $u \in \GG_m(R)$ acts on $\gamma(z) \in \LG(R)=G(R(\!(t)\!))$ by $u\gamma(z)u^{-1}=\gamma(uz)$. A maximal torus of $L^{\ltimes}G$ is $T^{\ltimes}=\GG_m \times T$.  An $\eta \in \Hom(\GG_m,T^{\ltimes}) \otimes_{\ZZ} \QQ$ over $R(s)$ can be viewed as a rational one-parameter subgroup ($1$-PS), \textit{i.e.} a $1$-PS $\GG_m \ra T^{\ltimes}$ over $R(w)$, where $w^n=s$ for some $n \geq 1$. In this case, for $\gamma(z) \in L^{\ltimes}G(R)$, we will view $\eta(s) \gamma(z) \eta(s)^{-1}$ as an element in $L^{\ltimes} G(R(w))$. Hence, by the condition  
\begin{equation}
  \lim_{s \ra 0} \eta(s) \gamma(z) \eta(s)^{-1} \text{ exists in } L^{\ltimes}G(R)
\end{equation} 
on $\gamma(z) \in L^{\ltimes}G(R)$, we mean that  there exists an $n \geq 1$ such that for $w^n=s$, we have
\begin{equation}
  \eta(s) \gamma(z) \eta(s)^{-1} \in L^{\ltimes}G(R \llbracket w \rrbracket ).
\end{equation}
This is the underlying principle in \cite{base} in a more global setting.
Let $\eta=(a,\theta)$ for $a$ a strictly positive rational number and $\theta$ a rational $1$-PS of $T$. Then, we have
\begin{equation} \label{loopconjugation}
\eta(s) \gamma(z) \eta(s)^{-1}=\theta(s) \gamma\left(s^{a}z\right) \theta(s)^{-1}.
\end{equation}
We note further that for any $0<d \in \mathbb{N}$, setting $\eta=(\frac{a}{d}, \frac{\theta}{d})$ we have
\begin{equation}
\eta(s) \gamma(z) \eta(s)^{-1}=\theta\left(s^{\frac{1}{d}}\right) \gamma\left(s^{\frac{a}{d}}z\right) \theta\left(s^{\frac{1}{d}}\right)^{-1}.
\end{equation}
So for $\eta=\frac{1}{d}(1,\theta)$, observe that the statement ``$\lim_{s \rightarrow 0} \eta(s) \gamma(z) \eta(s)^{-1}$ exists'' is a condition which is equivalent to  
\begin{equation}\label{observation}
  \theta(s) \gamma(s) \theta(s)^{-1} \in L^{\ltimes}G(R \llbracket w \rrbracket ).
\end{equation} 
In other words, the condition of the existence of limits is independent of $d$,  and we may further set $z=s$ in $\gamma(z)$. More generally, if $a>0$, the conditions for $(a,\theta)$ and $(1,\frac{\theta}{a})$ are equivalent. We may write this condition on $\gamma(z) \in L^{\ltimes}G(R)$ or $ \LG(R)$ as 
\begin{equation} \label{meaning}
 \text{``}\lim_{s \rightarrow 0} \theta(s) \gamma(s) \theta(s)^{-1}  \text{ exists in }  L^{\ltimes}G(R) \text{ or }  \LG(R) \text{ if }  \gamma(z) \in \LG(R)\text{''}.
\end{equation}
For $r \in \Phi$, let $u_r\colon \mathbb{G}_a \rightarrow G$ denote the root subgroup. If for some $b \in \mathbb{Z}$ and $t(z) \in R\llbracket z \rrbracket$, we take $\gamma(z):=u_r(z^{b} t(z)) \in Lu_r(R)$ and $\eta:=(1,\theta)$, then 
\begin{equation}
\eta(s) u_r(z^{b} t(z)) \eta(s)^{-1}=\theta(s) u_r((sz)^{b} t(sz)) \theta(s)^{-1}=u_r( s^{r(\theta)} (sz)^{b} t(sz)). 
\end{equation}
So for $\eta=(1,\theta)$, the condition that the limit exists is equivalent to 
\begin{equation} \label{floorcondition}
  r(\theta)+b \geq 0 \iff \lfloor r(\theta)+ b \rfloor = \lfloor r(\theta) \rfloor + b \geq 0 \iff b \geq -\lfloor r(\theta) \rfloor.
\end{equation}
We note the independence of the above implications of the number $d$ occurring in the equation $\eta:=\frac{1}{d}(1,\theta)$.

Let $\pi_{1}\colon T^{\ltimes} \ra \GG_m$ be the first projection. For any rational $1$-PS $\eta\colon \GG_m \ra T^{\ltimes}$, we say $\eta$ is positive if $\pi_{1} \circ \eta >0$, negative if $\pi_{1} \circ \eta<0$ and non-zero if $\pi_{1} \circ \eta$ is either positive or negative. In this paper, we will only need to work with $\eta$ which are positive.

Any non-zero $\eta=(a,\theta)$ defines the following positive-loop functors from the category of $k$-algebras to the category of groups and Lie algebras:
 \begin{eqnarray} \label{parlimstozero}
 {\cP}_{\eta}^{\ltimes}(R):= \left\{ \gamma \in L^{\ltimes}G(R) \;\Big|\;  \lim_{s \ra 0} \eta(s) \gamma(z) \eta(s)^{-1}  \text{ exists in }  L^{\ltimes}G(R) \right\}, \\
 \label{liealglimstozero}
 {\mathfrak{P}}_{\eta}^{\ltimes}(R):= \left\{ h \in L^{\ltimes} \gfr(R) \;\Big|\;  \lim_{s \ra 0} {Ad}(\eta(s))(h(z))  \text{ exists in }  L^{\ltimes} \gfr(R) \right\}, \\
 \label{liealglimstozero1}
 {\cP}_{\eta}(R):= \left\{ \gamma \in \LG(R) \;\Big|\;  \lim_{s \ra 0} \eta(s) \gamma(z) \eta(s)^{-1} \text{ exists in }  \LG(R) \right\}, \\
 \label{paraLiealg}
 {\mathfrak{P}}_{\eta}(R)=\left\{ h \in L \gfr(R) \;\Big|\;  \lim_{s \ra 0} Ad(\eta(s))(h(z)) \text{ exists in }  L \gfr(R)\right\},\\
 \label{liealgpara}
\text{Thus,}~~ \cP_{\eta}:=\cP_{\eta}^{\ltimes} \cap (1 \times \LG) \quad \text{and} \quad  {\mathfrak{P}}_{\eta}:={\mathfrak{P}}_{\eta}^{\ltimes} \cap  (0 \oplus L \gfr).
 \end{eqnarray}

A {\it parahoric} subgroup  of $L^{\ltimes}G$ (resp.\ $\LG$) is a subgroup that is conjugate to $\cP^{\ltimes}_{\eta}$ (resp.\ $\cP_{\eta}$) for some $\eta$. In this paper we will mostly be using only the case when $\eta=\frac{1}{d}(1,\theta)$. In this case, letting $L\mathfrak{g}(R)=\mathfrak{g}(R(\!(s)\!))$ as in \eqref{meaning}, we may reformulate \eqref{liealgpara} as 
\begin{equation} \label{meaning1} 
{\mathfrak{P}}_{\eta}(R)=\left\{h \in L\mathfrak{g}(R) \;\Big| \; \lim_{s \rightarrow 0} Ad(\theta(s))(h(s))  \text{ exists in }  L\mathfrak{g}(R) \right\}.
\end{equation}

Let $\mathfrak{t}$ denote the Cartan subalgebra associated to $T$, and let $\mathfrak g_r \subset \mathfrak g$ be the root subspace associated to $U_r$. Then in terms of generators, the Lie algebra functor associated to the group functor $\cP_{\theta}$ (see \eqref{parahoricwithgen})  is given by 
\begin{equation}\label{lieparwg}
\Lie({\cP}_{\theta})(R)=\left\langle \mathfrak{t}(R\llbracket z \rrbracket), \mathfrak{g}_r\left(z^{m_r(\theta)} R\llbracket z \rrbracket\right), r \in \Phi \right\rangle.
\end{equation}

{\sl By the conditions \eqref{parlimstozero} and \eqref{liealgpara}, using \eqref{floorcondition}, we may express $\cP_{\eta}$ in terms of generators  as in \eqref{parahoricwithgen} and its Lie algebra functor $\Lie(\cP_{\eta})$ as in \eqref{lieparwg}}.
Thus, for any $\eta$ of the type $(1,\theta)$,  we get the equality of Lie algebra functors 
\begin{equation} \label{parahorLiealgrel}
\Lie(\cP_{\eta})={\mathfrak{P}}_{\eta}.
\end{equation}
This can be seen by using (\ref{liealglimstozero}) and (\ref{liealgpara}) and then  replacing the conjugation in (\ref{loopconjugation})  by $\Ad(\eta(s))$.

For a rational $1$-PS $\theta$ of $T$ with $\eta=(1,\theta)$, we will have the notation 
\begin{equation} \label{paratheta}
{\mathfrak{P}}_{\theta}:={\mathfrak{P}}_{\eta}.
\end{equation}

\subsection{Pro-groups and pro-Lie algebras} \label{btgpsch}
Recall that to each facet $\Sigma_\mathbb I \subset \ol{\mathbf{a}_0}$, Bruhat--Tits theory associates  a parahoric group scheme $\gG_{\mathbb I}$ on $\spec({\cO})$, which is smooth and affine with connected fibres and whose generic fibre is $G \times _{\spec(k)} \spec(K)$.  Let $L^+(\gG_{\mathbb I})$ be the group functor
\begin{equation}\label{lplus}
L^+(\gG_{\mathbb I})(R) :=  \gG_{\mathbb I}(R\llbracket z \rrbracket).
\end{equation}
This is represented by a  pro-algebraic group over $k$ which is a subgroup
of the loop group $\LG$.  For $\eta=(1,\theta)$, where $\theta$ is any rational $1$-PS lying in $\Sigma_\mathbb I$, we therefore have the following identifications: 
\begin{equation}
  L^+(\gG_{\mathbb I})(k) = \cP_{\eta}  \quad \text{and} \quad \Lie(L^+(\gG_{\mathbb I})(k)) = {\mathfrak{P}}_{\eta}.
\end{equation}

\subsection{{\tt n}-parahoric loop groups in higher dimensions}\label{higherobservation} In this subsection we give a loop-theoretic reformulation of the definitions in  Section~\ref{loopgpsec} to the higher-dimensional base. As always, let $\cO_{n}: =k\llbracket z_{1},\ldots,z_{n} \rrbracket$ and let $R$ denote an arbitrary $k$-algebra. We denote the field of Laurent polynomials by $K_{n}=k(\!(z_{1},\ldots,z_{n})\!)= k\llbracket z_{1},\ldots,z_{n} \rrbracket[\ldots, z_i^{-1},\ldots]$. The {\tt n-loop} group $L_{n}G$ on a $k$-algebra $R$ is given by $L_{n}G(R) := G\big(R(\!(z_{1},\ldots,z_{n})\!)\big)$. Similarly, the {\tt n-loop Lie algebra} $L_{n} \gfr$ is given by $L_{n} \gfr(R): =\gfr\big(R(\!(z_{1},\ldots,z_{n})\!)\big)$.

Let $L^{\ltimes}_{n}G := \GG_{m}^{n} \ltimes L_{n}G$, where the {\it rotational torus} $\GG_m^{n}$ acts on $L_{n}G$ by acting on the uniformizer via the {\it loop  rotation action}. This goes as follows: an element $\mathbf{u} :=(\ldots,u_{i},\ldots) \in \GG_m^{n}(R)$ acts on $\gamma(z_{1},\ldots,z_{n}) \in L_{n}G(R) = G(R(\!(z_{1},\ldots,z_{n})\!))$ by 
\begin{equation} \label{looprotationhd}
  \mathbf{u}\gamma(z)\mathbf{u}^{-1}:= \gamma(\ldots,u_{i}z_{i},\ldots).
\end{equation}

A maximal torus of $L_{n}^{\ltimes}G$ is $T^{\ltimes}=\GG_m^{n} \times T$.  An $\mathbf{\eta} \in \Hom(\GG_m,T^{\ltimes}) \otimes_{\ZZ} \QQ$ over $R(s)$ can be viewed as a rational $1$-PS. For example, if $\eta_{i}=\big((0,\ldots,a_{i},\ldots,0),\theta_{i}\big)$ for $(0,\ldots,a_{i},\ldots,0) \in \mathbb{Q}^{n}$ and $\theta_{i}$ a rational $1$-PS of $T$, then  as before, we have 
\begin{equation} \label{hdconjugation}
\eta_{i}(s) \gamma(z_{1},\ldots, z_{n}) \eta_{i}(s)^{-1} = \theta_{i}(s) \gamma(z_{1},\ldots,s^{a_i} z_{i},\ldots,z_{n}) \theta_{i}(s)^{-1}.
\end{equation}

For $\eta_{1},\ldots,\eta_{n}$ as before, we can make sense of the condition on $\gamma(z) \in L_{n}^{\ltimes}G(R)$
{\begin{center}
    {``$\lim_{s_i \ra 0} \prod \eta_{i}(s_i) \gamma(z_{1},\ldots,z_{n}) \prod \eta_{i}(s_i)^{-1}$ exists in $L_{n}^{\ltimes}G(R)$''.}
\end{center}} 

Similarly, if ${\bT}=(\theta_{1},\theta_{i},\ldots, \theta_{n}) \in (Y(T) \otimes \mathbb{Q})^{n}$ denotes an $n$-tuple of  rational $1$-PSs of $T$ and if $m_r(\theta)=-\lfloor r(\theta) \rfloor$ as in \eqref{mrtheta}, then in terms of generators, we have an {\tt n-parahoric} group ${\cP}_{\bT}(R)$ associated to $\bT$  as in  \eqref{moregen0}. 

Defining $\eta_{i}:=\big((0,\ldots,1,\ldots,0),\theta_{i}\big)$ as before, we can describe the $n$-parahoric groups as follows:
\begin{eqnarray} 
{\cP}_{\eta_{1},\ldots,\eta_{n}}(R):= \left\{ \gamma(z) \in L_{n}^{\ltimes}G(R) \;\Big|\; \lim_{s_i \ra 0} \prod \eta_{i}(s_i). \gamma(z_{1},\ldots,z_{n}). \prod \eta_{i}(s_i)^{-1} \text{ exists} \right\} \\ \label{gphd}
\text{and in fact}~~{\cP}_{\bT}(R) = {\cP}_{\eta_{1},\ldots,\eta_{n}}(R) \cap (1 \times L_{n}G)(R).
\end{eqnarray}

Thus facets correspond to $n$-tuple of usual facets.
Further, let $\lambda \in \mathcal{A}_T$. Then 
\begin{equation}
\lambda(z_{i}) {\cP}_{\bT} \lambda(z_{i})^{-1}= {\cP}_{(\theta_{1},\ldots,\theta_{i}+\lambda, \ldots, \theta_{n})}.
\end{equation}
So conjugacy classes of parahorics  correspond to $n$-tuples of non-empty subsets of $\mathbb{S}$.

We can similarly describe the {\tt n}-parahoric Lie algebras. In particular,
\begin{eqnarray} 
 {\mathfrak{P}}^{\ltimes}_{\eta_i}(R)=\left\{ h \in L_{n}\gfr(R) \;\Big|\;  \lim_{s_i \ra 0} \Ad\left(\prod \eta_i(s_i)\right)(h(z))  \text{ exists in }  L_{n}^{\ltimes} \gfr(R)\right\},
 \\ \label{Liealghd}
\text{Thus,}~~ {\mathfrak{P}}_{\bT}(R):={\mathfrak{P}}_{\eta}^{\ltimes}(R) \cap  (0 \oplus L_{n}\gfr)(R). 
\end{eqnarray}

\subsection{Standard parahoric subgroups} \label{sbtgpsch}
The {\em standard parahoric subgroups} of $G(K)$ are parahoric subgroups of the distinguished hyperspecial parahoric subgroup  $G(\cO)$. These are realized as inverse images under the evaluation map $\ev\colon G(\cO) \to G(k)$ of standard parabolic subgroups of $G$. In particular, the standard {\em Iwahori subgroup} ${\mathfrak I}$ is the group ${\mathfrak I} = \ev^{-1}(B)$.

\section{Examples and Remarks}
The first example is illustrative of the fact that the operations of imposing regularity conditions and setting variables to be equal do not commute.
\begin{ex} Let $G = \SL_2$, $\cO=\cO_{2}$ and $K:=K_{2}$. Let ${\bvt}=(\theta_{1},\theta_{2}) \in \mathbb{Q}^2$. 
Then 
\begin{eqnarray}
\begin{pmatrix} s_{1}^{\theta_{1}}  s_{2}^{\theta_{2}} & 0 \\ 0 & s_{1}^{-\theta_{1}} s_{2}^{- \theta_{2}} \end{pmatrix}  \begin{pmatrix} X_{11}(s_{1}z_1,s_{2}z_2) & X_{12}(s_{1}z_1,s_{2}z_2) \\ X_{21}(s_{1}z_1,s_{2}z_2) & X_{22}(s_{1}z_1,s_{2}z_2) \end{pmatrix} \begin{pmatrix} s_{2}^{-\theta_{2}} s_{1}^{-\theta_{1}} & 0 \\ 0 & s_{2}^{\theta_{2}} s_{1}^{\theta_{1}} \end{pmatrix}  \\ \label{midconju}
\hphantom{\Bigg(}= \begin{pmatrix} X_{11}(s_{1}z_1,s_{2}z_2) & X_{12}(s_{1}z_1,s_{2}z_2) ~s_{1}^{2\theta_{1}} s_{2}^{2\theta_{2}} \\ X_{21}(s_{1}z_1,s_{2}z_2) ~s_{1}^{-2\theta_{1}} s_{2}^{-2\theta_{2}} &  X_{22}(s_{1}z_1,s_{2}z_2) \end{pmatrix} \subset SL_2(K(s_{1},s_{2})). 
\end{eqnarray}
If we put the regularity condition ``$\lim_{s_i \rightarrow 0} \text{exists}$'', we get
\begin{equation}
{\cP}_{\bvt} = \begin{pmatrix} 
\cO & \cO z_1^{- \lfloor 2 \theta_{1} \rfloor } z_2^{- \lfloor 2 \theta_{2} \rfloor } \\
\cO z_1^{-\lfloor -2\theta_{1} \rfloor } z_2^{- \lfloor -2 \theta_{2} \rfloor } & \cO
\end{pmatrix}.
\end{equation}
Now setting $t:=z_1=z_2$, $B:= k \llbracket t \rrbracket$ and ${\bvt} \in (0,\frac{1}{2})^2$, we get 
$$
{\cP}_{\bvt}^{\diag} = \begin{pmatrix}
 B & B \\
t^2 B & B
\end{pmatrix}.
$$
On the other hand, setting $t:=z_1=z_2$ and $s:=s_{1}=s_{2}$ in \eqref{midconju}, we get
\begin{equation}
\begin{pmatrix}
X_{11}(st) & X_{12}(st) ~s^{2 \theta_{1} + 2 \theta_{2}} \\
X_{21}(st) ~s^{-2 \theta_{1} - 2 \theta_{2}} & X_{22}(st)
\end{pmatrix}.
\end{equation}
Now putting the regularity condition ``$\,\lim_{s \rightarrow 0} \text{exists}$'', we get $\begin{pmatrix}
B & B t^{- \lfloor 2 \theta_{1}+2 \theta_{2} \rfloor} \\
B t^{- \lfloor -2 \theta_{1} - 2 \theta_{2} \rfloor} & B
\end{pmatrix}$. Notice for ${\bvt} \in (0,\frac{1}{2})^2$ that when $0< \theta_{1}+  \theta_{2} < \frac{1}{2}$, we get 
$\begin{pmatrix}
B & B \\
t B & B
\end{pmatrix}$, while when $\frac{1}{2} \leq  \theta_{1} +  \theta_{2} < 1$, we get
$\begin{pmatrix}
B & t^{-1} B \\
t^2 B & B
\end{pmatrix}$.
In either case, for ${\bvt} \in (0,\frac{1}{2})^2$, we see that the operations of imposing regularity conditions and setting variables to be equal do not commute.
\end{ex}

\begin{ex} This example shows that the multiplication law in higher-dimensional parahorics is more involved than in the case of discrete valuation rings.

  Let $G:= \SL_{3}$, $\cO:=\cO_{2}$ and $K:=K_{2}$. In this example we study ${\bvt}:=(\frac{\omega^{\vee}_{\alpha_1}}{3},\frac{\omega^{\vee}_{\alpha_2}}{3})$. The co-weights are given by the points $(\frac{2}{3},\frac{-1}{3},\frac{-1}{3})$, $(\frac{1}{3},\frac{1}{3},\frac{-2}{3}) \in \mathcal{A}_T \subset \mathbb{R}^3$;  \textit{i.e.} $ \omega^{\vee}_{\alpha_1}(t)=\diag(t^{\frac{2}{3}}, t^{\frac{-1}{3}}, t^{\frac{-1}{3}})$ and $ \omega^{\vee}_{\alpha_2}(t)=\diag(t^{\frac{1}{3}}, t^{\frac{1}{3}}, t^{\frac{-2}{3}})$. We have
\begin{equation}
  {\bvt}(s_{1},s_{2})= \diag\left(s_{1}^{\frac{2}{9}}s_{2}^{\frac{1}{9}}, s_{1}^{\frac{-1}{9}}s_{2}^{\frac{1}{9}}, s_{1}^{\frac{-1}{9}}s_{2}^{\frac{2}{9}}\right).
\end{equation}
Conjugating the $3 \times 3$ matrix $\begin{pmatrix}
E_{ij}(s_{1}z_1,s_{2}z_2)
\end{pmatrix}$ by ${\bvt}(s_{1},s_{2})$, we get 
\begin{equation}
\begin{pmatrix}
E_{11}(s_{1}z_1,s_{2}z_2) & E_{12}(s_{1}z_1,s_{2}z_2)s_{1}^{\frac{1}{3}} & E_{13}(s_{1}z_1,s_{2}z_2) s_{1}^{\frac{1}{3}}s_{2}^{\frac{1}{3}} \\
E_{21}(s_{1}z_1,s_{2}z_2) s_{1}^{\frac{-1}{3}} & E_{22}(s_{1}z_1,s_{2}z_2) & E_{23}(s_{1}z_1,s_{2}z_2)s_{2}^{\frac{1}{3}} \\
E_{31}(s_{1}z_1,s_{2}z_2) s_{1}^{\frac{-1}{3}}s_{2}^{\frac{-1}{3}}& E_{32}(s_{1}z_1,s_{2}z_2) s_{2}^{\frac{-1}{3}}& E_{33}(s_{1}z_1,s_{2}z_2)
\end{pmatrix}.
\end{equation}
Imposing regularity conditions forces $E_{12}$, $E_{13}$ and $E_{23}$ to be regular and $E_{21}(z_1,z_2)$ (resp.\ $E_{23}(z_1,z_2)$) to have a zero of order at least $1$ with respect to $z_1$ (resp.\ $z_2$) and $E_{31}(z_1,z_2)$ to have a zero of order at least $1$ with respect to both $z_1$ and $z_2$. So we get 
\begin{equation}
{\cP}_{\bvt} = \left\{ \begin{pmatrix}
F_{11}(z_1,z_2) & F_{12}(z_1,z_2) & F_{13}(z_1,z_2) \\
z_1 F_{21}(z_1,z_2) & F_{22}(z_1,z_2) & F_{23}(z_1,z_2) \\
z_1 z_2 F_{31}(z_1,z_2) & z_2 F_{32}(z_1,z_2) & F_{33}(z_1,z_2) 
\end{pmatrix} \;\Bigg|\; F_{ij}(z_1,z_2) \in \cO_{2} \right\}.
\end{equation}
We notice that 
\begin{eqnarray} 
z_1z_2 F_{31}=z_1z_2(g_{31}h_{11}+g_{32}h_{21}+g_{33}h_{31}) \quad \text{but} \\
z_1 F_{21}=z_1(g_{21}h_{11}+g_{22}h_{21}+g_{23}h_{31}z_2).
\end{eqnarray} In other words, the multiplication law for the entry $(3,1)$ is like that of $SL_3(\cO_{2})$, but that of $(2,1)$ has an extra $z_2$. Similarly, the law for the entries $(2,2)$ and $(1,1)$ are unlike the $1$-dimensional case. 
\end{ex}

\subsection{Some remarks on specializations of {\tt n}-parahoric groups and examples}\label{allunifequal0}
The purpose of this subsection is to show that specialization of {\tt n}-parahoric groups leads to a strictly larger scope than that of parahoric groups.

\begin{ex}  This example shows the simplest case of relating higher-dimensional parahorics to lower-dimensional ones. We place ourselves in the context of the previous example. If we set $s_{2},z_2:=1$ or $s_{1},z_1:=1$, we get the two standard maximal parabolics. If we set $s_{1},z_1,s_{2},z_2:=t$, all variables as $t$, then conjugation is by $\diag(t^{\frac{1}{3}}, 1, t^{\frac{-1}{3}})$. This gives the standard Iwahori; see Section~\ref{sbtgpsch}.
\end{ex}

\begin{prop} \label{mrthetaisconcave} Let $\theta \in \mathcal{A}_T$ and let $m_r(\theta) := -\lfloor r(\theta) \rfloor$ as in \eqref{mrtheta}. Then the function ${\cm}_{\theta}\colon\tilde{\Phi} \rightarrow \mathbb{R}$ defined by $\cm_{\theta}(r) := m_r(\theta)$  is concave.
\end{prop}

\begin{proof} Let $r_{i} \in \tilde{\Phi}$ be such that $\sum_{i} r_{i} \in \tilde{\Phi}$ as well.  We need to check that 
\begin{equation}
\cm_{\theta}\left({\sum_{i} r_{i}}\right) = m_{\sum_{i} r_{i}(\theta)} \leq \sum_i m_{r_{i}}(\theta).
\end{equation}
This follows by setting $a_i:=r_{i}(\theta)$ and observing $
\lfloor a_1 +\cdots+a_n \rfloor \geq \lfloor a_1 \rfloor + \lfloor a_2 \rfloor + \cdots + \lfloor a_n \rfloor$.
\end{proof}

\begin{defi}\label{allunifequaldef} For an {\tt n-concave} function ${\bf f} := (f_{1}, \ldots, f_{n})$, by the specialization of ${\cP}_{\bf f} \subset G(k(\!(z_{1},\ldots,z_{n})\!))$ to the diagonal 
\begin{equation}\label{allunifequal}
{\cP}_{\bf f}^{\diag} \subset G(k(\!(t)\!)), 
\end{equation}
we mean the {\em bounded subgroup} obtained by setting $z_1=z_2=\cdots=z_n:=t$.
\end{defi} 
In particular, we have the specializations ${\cP}_{\bwt}^{\diag}$ (of ${\cP}_{\bwt}$) and ${\cP}_{\bvt}^{\diag}$ (of ${\cP}_{\bvt}$) associated to $\bwt = (\Omega_{1}, \ldots, \Omega_{n})$ and ${\bvt}:=(\theta_1,\ldots,\theta_n)$, respectively.

Since sums of concave functions are concave,  in terms of the concave function $\Phi \rightarrow \mathbb{Z}$ given by $r \ms \sum_i m_r(\theta_i)$, we get 
\begin{equation}\label{allunifequal1}
{\cP}_{\bvt}^{\diag} = \left\langle T(k \llbracket t \rrbracket), U_r\left(t^{\sum_i m_r(\theta_i)} k \llbracket t \rrbracket\right), r \in \Phi \right\rangle.
\end{equation}
We set
\begin{equation} \label{allunifequalLie}
  \mathfrak{P}_{\bvt}^{\diag} := \Lie\left({\cP}_{\bvt}^{\diag}\right).
\end{equation}

\brem We now make some remarks  comparing the {\em bounded subgroup} ${\cP}_{\bvt}^{\diag}$ with the {\em parahoric subgroup} ${\cP}_{\sum_i \theta_i}$, both subgroups of $G(K)$. Recall that $m_r(\theta)=-\lfloor r(\theta) \rfloor$. So let us compare $m_r(\sum_i \theta_i)$ with $\sum_i m_r(\theta_i)$.  For simplicity, we assume that all $\theta_i$ lie in the dominant Weyl chamber so that for all positive roots $r$, we have $r(\theta_i)>0$.

Consider the case when for all $r \in R^+$, we have $m_r(\sum_i \theta_i)= \sum_i m_r(\theta_i)$. This occurs for instance when all $\theta_i$ lie in the alcove $\mathbf{a}_0$ and their sum lies in $\mathbf{a}_0 \setminus \text{far wall}$ or when the $\theta_{i}$ lie in the lattice defined by the $\theta_{\alpha_i}$; see \eqref{alcovevertices}. For alcove vertices $\theta_{\alpha_i}$, an example of such a ${\bvt} = (\theta_{1}, \ldots, \theta_{\ell})$ is given by assigning 
\begin{equation} \label{dialatedTitsvectors}
\theta_i:= \frac{\theta_{\alpha_i}}{(\ell+1)}.
\end{equation}
In this case, for any $r \in 
\Phi^+$, although $r(\theta_i)>0$, we  nonetheless have $m_r(\sum_i \theta_i) = 0$ and $m_r(\theta_i)= 0$ and, further, $m_{-r}(\theta_i)=1$ and $m_{-r}(\sum_i \theta_i)=1$. Then the orders of poles that the group entries of ${\cP}_{\bvt}^{\diag}$ and ${\cP}_{\sum_i \theta_i}$ can allow at $r$ are the same. However, the order of zeros of elements of ${\cP}_{\bvt}^{\diag}$ at $-r$ is the number $S_r$ of distinct simple roots $\alpha \in S$ occurring with strictly positive coefficient in the expression of $r$. Thus,
\begin{equation}
{\cP}_{\bvt}^{\diag} = \left\langle T(k \llbracket t \rrbracket ), U_r( R\llbracket t \rrbracket), U_{-r}\left(t^{S_r} k \llbracket t \rrbracket\right), r \in \Phi^+\right\rangle.
\end{equation} 
Further, the bounded subgroup ${\cP}_{\sum_i \theta_i}$ is not a $G(K)$-conjugate of a subgroup of ${\cP}_{\bvt}^{\diag}$ in this case unless $|S|=1$.

Now consider the case wherein, for $r \in R^+$, we have $m_r(\sum_i \theta_i)< \sum_i m_r(\theta_i)$. This occurs for instance  when $G = \SL_2$, $n=2$ and $r(\theta_1)=r(\theta_2)=a_1=a_2=\frac{1}{2}$.  In this case, $m_r(\theta_i)=0$ but $m_r(\sum_i \theta_i)=-1$. Then the orders of poles which the elements of ${\cP}_{\sum_i \theta_i}$ can allow are strictly greater than those of the elements of ${\cP}_{\bvt}^{\diag}$ at~$r$.  On the other hand,  $m_{-r}(\theta_i)=1$ and $m_{-r}(\sum_i \theta_i)=1$ as well. So the order of zeros of ${\cP}_{\sum_i \theta_i}$ at $-r$ is at least~$1$ but that of ${\cP}_{\bvt}^{\diag}$ at $-r$ is at least~$2$. So reasoning on the orders of the poles and zeros, it follows that ${\cP}_{\sum_i \theta_i}$ is not a $G(K)$-conjugate of a subgroup of ${\cP}_{\bvt}^{\diag}$.\erem

\subsection{Constructing concave functions of {\tt type III} from {\tt type I}}\label{explicitexamples}
 The purpose of this subsection is to illustrate the phenomenon  
 that one may begin with a group scheme datum of {\tt type I} at the generic points of the divisors, but the {\tt type} of the {\tt BT}-group scheme at points of depth bigger than $1$ could  strictly be even of {\tt type III}. The computations should be seen in the setting of Theorem~\ref{thedescription}.

To illustrate the issue, we work with  \eqref{dialatedTitsvectors},  which are the simplest non-trivial cases; these also highlight the role of the root system. 
We make explicit constructions  for $A_{2}$, $B_{2}$ and $G_{2}$ cases with $\bvt$ given by \eqref{dialatedTitsvectors}. The contrast between the $B_{2}$ and $G_2$ cases highlights the  critical role that the root system of the group plays in the analysis.

Recall that for a simple root $\alpha$, we denote by $\theta_{\alpha}$, see  \eqref{alcovevertices}, the  vertices of the alcove.

\begin{ex}[$A_{2}$] The set of positive roots is $\Phi^{+} = \{\alpha_{1}, \alpha_{2}, 
\alpha_{1}+ \alpha_{2}\}$. We work with $\theta_{1}  := \theta_{\alpha_1}/3= \omega^{\vee}_{1}/3$ and $\theta_{2} := \theta_{\alpha_2}/3= \omega^{\vee}_{2}/3$. Recall that $m_r(\theta) := -\lfloor r(\theta) \rfloor$. We have 
\begin{equation}
  r(\theta_{1}) =
  \begin{cases}
1/3,0,1/3 & \quad\text{if }r \in \Phi^{+}, \\
-1/3,0,-1/3 & \quad\text{if }r \in \Phi^{-},  
  \end{cases}
\end{equation}

\begin{equation}
  m_r(\theta_{1}) =
  \begin{cases}
0,0,0 & \quad\text{if } r \in \Phi^{+}, \\
1,0,1 & \quad\text{if } r \in \Phi^{-},
  \end{cases}
\end{equation}
and similarly for $\theta_{2}$. We get 
\begin{equation}
  r(\theta_{1} + \theta_{2}) =
\begin{cases}
1/3,1/3,2/3 & \quad\text{if }r \in \Phi^{+}, \\
-1/3,-1/3,-2/3 & \quad\text{if }r \in \Phi^{-},   
\end{cases}
\end{equation}
\begin{equation}
  m_r(\theta_{1} + \theta_{2}) =
  \begin{cases}
0,0,0 & \quad\text{if }r \in \Phi^{+}, \\
1,1,1 & \quad\text{if }r \in \Phi^{-}.
\end{cases} \end{equation}
The function $\cm\colon r \mapsto \sum m_r(\theta_{j})$ is given by 
\begin{equation}
  \cm(r) =
\begin{cases}
0,0,0 & \quad\text{if } r \in \Phi^{+}, \\
1,1,2 & \quad\text{if } r \in \Phi^{-}; 
\end{cases}
\end{equation}
it computes ${\cP}_{\bvt}^{\diag}$.
Let $\Omega \subset \mathcal A_{T}$ be the bounded subset $\Omega = \{b,b'\}$, where $b = \theta_{1} + \theta_{2}$ and $b' = 2\cdot\theta_{1} + 2\cdot\theta_{2}$. Recall that $m_r(\Omega) = -\lfloor \Inf_{\theta \in \Omega} r(\theta) \rfloor$. For the subset $\Omega$, define the concave function $f_{\Omega}(r) := m_r(\Omega)$; see \eqref{mromega}. We have 
\begin{equation}
  r(b) =
\begin{cases}
1/3,1/3,2/3 & \quad\text{if }r \in \Phi^{+}, \\
-1/3,-1/3,-2/3 & \quad\text{if } r \in \Phi^{-}, 
\end{cases} \end{equation} 
\begin{equation}
  r(b') =
  \begin{cases}
2/3,2/3,4/3 & \quad\text{if } r \in \Phi^{+}, \\
~-2/3,-2/3,-4/3 & \quad\text{if }r \in \Phi^{-}, 
\end{cases} \end{equation} 
\begin{equation}
  m_r(\Omega) =
  \begin{cases}
 0,0,0 & \quad\text{if } r \in \Phi^{+}, \\
 1,1,2 & \quad\text{if } r \in \Phi^{-}.
 \end{cases}\end{equation}
Hence, $\cm(r) = f_{\Omega}(r) $ for every $r \in \Phi$; \textit{i.e.} it is of {\tt type II}. In particular, the group scheme $\mathfrak G_{\cm}$ associated to the concave function on $\spec(A)$ is realized by a bounded subset of the affine apartment.
\end{ex}

Before we discuss the next set of examples, we prove a lemma regarding the concave function $f_{\Omega}$ defined above. Observe that
\beqa
f_{\Omega}(r) = a \iff a-1 < -\inf_{\theta \in \Omega} r(\theta) \leq a \iff 1-a > \inf_{\theta \in \Omega} r(\theta) \geq -a. 
\eeqa
The last statement can be interpreted as saying that 
\begin{enumerate}
\item \label{new1} for all $\theta \in \Omega$, we have $r(\theta) \geq -a$; 
\item \label{new2} there exists a $\theta' \in \Omega$ such that $1-a > r(\theta')$. 
\end{enumerate}   

\begin{lem}\label{new3}
Let $\Omega$ be a non-empty bounded region in the affine apartment. Suppose that there is a root $r$ such that $f_{\Omega}(-r) = - f_{\Omega}(r)$. Set $n := f_{\Omega}(r)$. Let $H_r(n)$ be the hyperplane in the affine apartment defined by $H_r(n) := \{\theta \mid r(\theta) = -n\}$. Then, 
\beqa
\Omega \subset H_r(n).
\eeqa
\end{lem}

\begin{proof} This is immediate from property \eqref{new1} above.
\end{proof}

\begin{ex}[$B_{2}$, first example] The set of positive roots is $\Phi^{+} = \{\beta, \alpha, 
\beta+ \alpha, \beta+ 2\cdot\alpha\}$. We work with $\theta_{1}:=\theta_{\beta}/3  = \omega^{\vee}_{1}/3$ and $ \theta_{2}:=\theta_{\alpha}/3 = \omega^{\vee}_{2}/6$. We have
\begin{equation}
  r(\theta_{1}) =
  \begin{cases}
1/3,0,1/3, 1/3 & \quad\text{if }r \in \Phi^{+}, \\
-1/3,0,-1/3, -1/3& \quad\text{if } r \in \Phi^{-}, 
\end{cases} \end{equation}

\begin{equation}
  m_r(\theta_{1}) =
  \begin{cases}
0,0,0,0 & \quad\text{if } r \in \Phi^{+}, \\
1,0,1,1 & \quad\text{if } r \in \Phi^{-}, 
\end{cases} \end{equation}
\begin{equation}
  r(\theta_{2}) =
\begin{cases}
0,1/6,1/6, 1/3 & \quad\text{if } r \in \Phi^{+}, \\
~0,-1/6,-1/6, -1/3 & \quad\text{if } r \in \Phi^{-}, 
\end{cases} \end{equation}

\begin{equation}
  m_r(\theta_{2}) =
\begin{cases}
0,0,0,0 & \quad\text{if } r \in \Phi^{+}, \\
0,1,1,1 & \quad\text{if } r \in \Phi^{-}.
\end{cases} \end{equation}
We get
\begin{equation}
  r(\theta_{1} + \theta_{2}) =
\begin{cases}
1/3,1/3,2/3,2/3 & \quad\text{if } r \in \Phi^{+}, \\
-1/3,-1/3,-2/3,-2/3 & \quad\text{if } r \in \Phi^{-}, 
\end{cases} \end{equation}
\begin{equation}
  m_r(\theta_{1} + \theta_{2}) =
\begin{cases}
0,0,0,0 & \quad\text{if } r \in \Phi^{+}, \\
~1,1,1,1 & \quad\text{if } r \in \Phi^{-}.
\end{cases} \end{equation}
The function $\cm\colon r \mapsto \sum m_r(\theta_{j})$ is given by 
\begin{equation}
  \cm(r) =
\begin{cases}
0,0,0,0 & \quad\text{if } r \in \Phi^{+}, \\
1,1,2,2 & \quad\text{if } r \in \Phi^{-}; 
\end{cases} \end{equation}
it computes ${\cP}_{\bvt}^{\diag}$.
Let $\Omega \subset \mathcal A_{T}$ be the bounded subset $\Omega = \{b,b'\}$, where $b = \theta_{1} + \theta_{2}$ and $b' = (2-e)\cdot\theta_{1} + (2
+e')\cdot\theta_{2}$ with $0 < e,e' < 1$. Recall that $m_r(\Omega) = -\lfloor \text{Inf}_{\theta \in \Omega} r(\theta) \rfloor$. For the subset $\Omega$, define the concave function $f_{\Omega}(r) := m_r(\Omega)$. We have 
\begin{equation}
  r(b) =
\begin{cases}
1/3,1/6,1/2, 2/3 & \quad\text{if } r \in \Phi^{+}, \\
~-1/3,-1/3,-1/2,-2/3 & \quad\text{if } r \in \Phi^{-}, 
\end{cases} \end{equation} 
\begin{equation}
  r(b') =
\begin{cases}
2/3-e/3,1/3+e'/6,1+(e'-2e)/6,4/3+ (e'-e)/3 & \quad\text{if } r \in \Phi^{+}, \\
-2/3 + e/3,-1/3 -e'/6,-1- (2e-e')/6, -4/3- (e-e')/3 & \quad\text{if } r \in \Phi^{-}, 
\end{cases} \end{equation} 
\begin{equation}
  m_r(\Omega) =
\begin{cases}
0,0,0,0 & \quad\text{if } r \in \Phi^{+}, \\
1,1,2,2 & \quad\text{if }  r \in \Phi^{-}.
\end{cases} \end{equation}
Hence, $\cm(r) = f_{\Omega}(r) $ for every $r \in \Phi$; \textit{i.e.} it is again of {\tt type II}. 
\end{ex}

\begin{ex}[$B_{2}$, second example] \label{B2type3} The second example here has special significance for Section~\ref{mckaybtetc}. This time we work with $\theta_{1}:=\theta_{\alpha}/2  = \omega^{\vee}_{2}/4$ and $ \theta_{2}:=\theta_{\alpha} = \omega^{\vee}_{2}/2$. Then we get
\begin{equation}
  m_r(\theta_{1}) =
  \begin{cases}
0,0,0,0 & \quad\text{if } r \in \Phi^{+}, \\
0,1,1,1 & \quad\text{if } r \in \Phi^{-}
\end{cases} \end{equation}
and
\begin{equation}
  m_r(\theta_{2}) =
\begin{cases}
0,0,0,-1 & \quad\text{if } r \in \Phi^{+}, \\
0,1,1,1 & \quad\text{if } r \in \Phi^{-}.
\end{cases} \end{equation}
The function $\cm\colon r \mapsto \sum m_r(\theta_{j})$ is given by 
\begin{equation}\label{newB2}
  \cm(r) =
\begin{cases}
0,0,0,-1 & \quad\text{if } r \in \Phi^{+}, \\
0,2,2,2 & \quad\text{if } r \in \Phi^{-}.
\end{cases} \end{equation}

\begin{claim*}
  The concave function $\cm$ in \eqref{newB2} is not of\, {\tt type II}; \textit{i.e.} there does not exist any bounded region $\Omega \subset \mathcal{A}$ such that $\cm=f_{\Omega}$.
\end{claim*}
        
  If the  claim is not true, then there is an $\Omega$ such that $f_{\Omega}(\beta) = -f_{\Omega}(-\beta) = 0$ and hence $\Omega \subset H(0)$. Again, since $f_{\Omega}(\alpha) = 0, f_{\Omega}(-\alpha) = 2$, see \eqref{newB2}, it follows that for all $\theta \in \Omega$, by \eqref{new1} before Lemma~\ref{new3}, we get the condition
\beqa
0 \leq \alpha(\theta) \leq 2 \quad\forall \theta \in \Omega.
\eeqa
Similarly, $f_{\Omega}(\alpha + \beta) = 0, f_{\Omega}(-(\alpha + \beta)) = 2$, see \eqref{newB2},  forces 
\beqa
0 \leq (\alpha + \beta)(\theta) \leq 2\quad \forall \theta \in \Omega, 
\eeqa
and $f_{\Omega}(2\cdot\alpha + \beta) = -1, f_{\Omega}(-(2\cdot\alpha + \beta)) = 2$ forces
\beqa\label{verynew1}
 (2\cdot\alpha + \beta)(\theta) \leq  2 \quad\forall \theta \in \Omega.
\eeqa
On the other hand, since $f_{\Omega}(-\alpha) = 2$, property~\eqref{new2} says that there exists a $\theta'\in \Omega$ such that
\beqa\label{verynew2}
1 < \alpha(\theta').
\eeqa
Notice that since $\beta(\theta') = 0$, the two inequalities \eqref{verynew1} and \eqref{verynew2} are not compatible. 
\end{ex}

\begin{ex}[$G_{2}$, first example] \label{gratifying} The set of positive roots is $\Phi^{+} = \{\alpha, \beta, 
\alpha+ \beta, 2\cdot\alpha+ \beta, 3\cdot\alpha+ \beta, 3\cdot\alpha+ 2\cdot\beta \}$. We work with $\theta_{1}:= \theta_{\alpha}/3  = \omega^{\vee}_{1}/9$ and $ \theta_{2}:=\theta_{\beta}/3 = \omega^{\vee}_{2}/6$.  We have
\begin{equation}
  r(\theta_{1}) =
\begin{cases}
1/9,0,1/9,2/9, 1/3, 1/3 & \quad\text{if } r \in \Phi^{+}, \\
-1/9,0,-1/9, -2/9, -1/3, -1/3 & \quad\text{if } r \in \Phi^{-}, 
\end{cases}
\end{equation}
\begin{equation}
  r(\theta_{2}) =
\begin{cases}
  0,1/6,1/6,1/6, 1/6, 1/3 & \quad\text{if } r \in \Phi^{+}, \\
0,-1/6,-1/6, -1/6, -1/6, -1/3 & \quad \text{if } r \in \Phi^{-}, 
\end{cases} \end{equation}
\begin{equation}
  m_r(\theta_{1}) =
\begin{cases}
0,0,0,0,0,0 &\quad\text{if } r \in \Phi^{+}, \\
1,0,1,1,1,1 &\quad{\text if }r \in \Phi^{-}.
\end{cases} \end{equation}
\begin{equation}
  m_r(\theta_{2}) =
\begin{cases}
0,0,0,0,0,0&\quad\text{if } r \in \Phi^{+}, \\
0,1,1,1,1,1&\quad{\text if }r \in \Phi^{-}, 
\end{cases} \end{equation}
We get 
\begin{equation}
  r(\theta_{1} + \theta_{2}) =
\begin{cases}
1/9,1/6,5/18,7/18, 1/2, 2/3&\quad\text{if } r \in \Phi^{+}, \\
-1/9,-1/6,-5/18,-7/18, -1/2,-2/3&\quad{\text if }r \in \Phi^{-}, 
\end{cases} \end{equation}
\begin{equation}
  m_r(\theta_{1} + \theta_{2}) =
\begin{cases}
0,0,0,0,0,0&\quad\text{if } r \in \Phi^{+}, \\
1,1,1,1,1,1&\quad{\text if }r \in \Phi^{-}.
\end{cases} \end{equation}
The function $\cm\colon r \mapsto \sum m_r(\theta_{j})$ is given by 
\begin{equation} \label{mG2}
  \cm(r) =
\begin{cases}
0,0,0,0,0,0&\quad\text{if } r \in \Phi^{+}, \\
1,1,2,2,2,2&\quad{\text if }r \in \Phi^{-}; 
\end{cases} \end{equation}
it  computes ${\cP}_{\bvt}^{\diag}$.

\begin{claim*}
In contrast, the concave function $\cm$ in \eqref{mG2} is not of {\tt type II}; \textit{i.e.} there does not exist any bounded region $\Omega \subset \mathcal{A}$ such that $\cm=f_{\Omega}$. This is explicit here, but see also \cite[Section~6.4.4]{bruhattits1}.
  \end{claim*}

To prove this claim, let us suppose to the contrary that there exists such a bounded subset $\Omega \subset \mathcal A_{T}$. Let $\theta \in \Omega$ be given by
$$
\theta := \left(x \omega_{1}^{\vee}, y \omega_{2}^{\vee}\right).
$$
Then  we have
\begin{equation} \label{verygratifying}
  r(\theta) =
\begin{cases}
x,y,x+y,2x+y,3x+y,3x+2y&\quad\text{if } r \in \Phi^{+}, \\
-x,-y,-x-y,-2x-y,-3x-y,-3x-2y&\quad{\text if }r \in \Phi^{-}.
\end{cases} \end{equation} 

For $ r \in \Phi^+$, we have $-\lfloor \inf_{\theta \in \Omega} (r,\theta) \rfloor =0$, so  $\inf_{\theta \in \Omega} (r,\theta) \in [0,1)$. Thus for $r=\alpha$ and $\beta$, for all points $\theta$ in $\Omega$, we have
  $$
  x, y \in [0,1).
  $$

Now let us assume that $\inf_{\theta \in \Omega} (r,\theta)$ is realized at $(x_1,y_1)$ for $r=-(\alpha+\beta)$ and at $(x_2,y_2)$ for $r=-(3 \alpha+2 \beta)$. In particular, $x_1,x_2,y_1,y_2 \in [0,1)$. 

We have $-\lfloor - (x_1+y_1) \rfloor=2$. So $-2 \leq -x_1-y_1 <-1$. Thus 
\begin{eqnarray} \label{1}
1< x_1+y_1 \leq 2,\\
\label{2} x_1 >0.
\end{eqnarray}

By the condition for $-3 \alpha-2\beta$, we have
$-3 x_2-2y_2 \leq -3 x_1 -2 y_1$. Thus
\begin{equation}
\label{310} 3 x_1+2y_1 \leq 3 x_2 + 2 y_2.
\end{equation}

By $-\lfloor - 3 x_2 - 2 y_2 \rfloor=2$, we have $-1 > -3 x_2 - 2y_2 \leq -2$. Thus 
\begin{equation} \label{4}
1 < 3 x_2+2y_2 \leq 2.
\end{equation}

We get the following contradiction: 
$$
2 \stackrel{\eqref{1}}{<} 2 x_1+2y_1 \stackrel{\eqref{2}}{<} 3 x_1+2y_1 \stackrel{\eqref{310}}{\leq} 3 x_2+2y_2 \stackrel{\eqref{4}}{\leq} 2.
$$
This completes the proof of the claim.
\end{ex}

\begin{ex}[$G_{2}$, second example] \label{G2type3} The second example here has special significance for Section~\ref{mckaybtetc}. We now work  with $\theta_{1}:=\theta_{\alpha}  = \omega^{\vee}_{\alpha}/3$ and $ \theta_{2}:= 2.\theta_{\alpha} = 2.\omega^{\vee}_{\alpha}/3$. Then we get
\begin{equation}
  m_r(\theta_{1}) =
\begin{cases}
0,0,0,0,-1,-1&\quad\text{if } r \in \Phi^{+}, \\
1,0,1,1,1,1&\quad{\text if }r \in \Phi^{-}, 
\end{cases} \end{equation}
\begin{equation}
  m_r(\theta_{2}) =
\begin{cases}
0,0,0,-1,-2,-2&\quad\text{if } r \in \Phi^{+}, \\
1,0,1,2,2,2&\quad{\text if }r \in \Phi^{-}.
\end{cases} \end{equation}
The function $\cm\colon r \mapsto \sum m_r(\theta_{j})$ is given by 
\begin{equation}\label{newG2}
  \cm(r) =
\begin{cases}
0,0,0,-1,-3,-3&\quad\text{if } r \in \Phi^{+}, \\
2,0,2,3,3,3&\quad{\text if }r \in \Phi^{-}. 
\end{cases} \end{equation}

\begin{claim*}
The concave function $\cm$ in \eqref{newG2} is not of {\tt type II}.
\end{claim*}

        If the claim is not true, then there is an $\Omega$ such that $f_{\Omega}(\beta) = -f_{\Omega}(-\beta) = 0$, and hence by Lemma~\ref{new3}, we have $\Omega \subset H_{\beta}(0)$. Thus
\beqa\label{verynew3}
\beta(\theta) = 0, \quad\forall \theta \in \Omega.
\eeqa
Similarly, $f_{\Omega}(3\alpha + 2\beta) = 3, f_{\Omega}(-(3\alpha + 2\beta)) = -3$ \eqref{newG2} forces 
\beqa\label{verynew4}
(3\alpha + 2\beta)(\theta) =3 \quad \forall \theta \in \Omega, 
\eeqa
and \eqref{verynew3} and \eqref{verynew4} give us 
\beqa\label{veryverynew}
\alpha(\theta) = 1 \quad \forall \theta \in \Omega.
\eeqa 

Now since $f_{\Omega}(-\alpha) = 2$,  by \eqref{new2} before Lemma~\ref{new3}, we know that there exists  a $\theta' \in \Omega$ such that $\alpha(\theta') > 1$, which is incompatible with \eqref{veryverynew}. 
\end{ex}

\begin{ex} (The case of a constant concave function). Let $a \geq 0$ and $f(r) = a$ for all $r \in \Phi$. In this case, when $G$ is of type $A_{n}$, $B_{n}$ or $C_{n}$, there exist regions $\Omega$ such that $f_{\Omega}=f$.  For $G_{2}$ when $a \geq 3$, there does not exist such a region.
\end{ex}

\subsection{Various cones of concave functions} We consider the question of how to construct general concave functions from points in the apartment.  It turns out that this can always be done in the case $A_{n}$ (see Proposition~\ref{btremarque} and the remarks before it). 
Let $\Phi$ be an irreducible root system.  Let
\begin{equation}
\Phi^{\Sigma}:= \left\{ (\alpha,\beta) \in \Phi^2 \mid \alpha+\beta \in \Phi \right\}.
\end{equation}
For $(\alpha,\beta) \in \Phi^{\Sigma}$, we get functionals $
\psi_{(\alpha,\beta)}$ and $\psi_{\alpha}$ defined on $\{ f\colon \Phi \rightarrow \mathbb{R} \} \rightarrow \mathbb{R}$ by $ f \mapsto f(\alpha)+f(\beta)-f(\alpha+\beta)$ and $ f \mapsto f(\alpha)+f(-\alpha)$, respectively.

Let $\mathcal{C}$ denote the space of concave functions $\{f\colon\Phi \rightarrow \mathbb{R} \}$. It can be expressed as
\begin{equation}
\mathcal{C}=\cap \left\{ \psi_{(\alpha,\beta)} \geq 0 \right\} \cap \left\{ \psi_{\alpha} \geq 0 \right\}.
\end{equation}
Being an intersection of half-spaces, it is a polyhedral cone,   further since the intersection of the open half-spaces is non-empty, we see that $\mathcal{C}$ is  of dimension $|\Phi|$.

Let $\mathcal{C}_{\Omega}$ denote the cone of concave functions defined by bounded regions $\Omega \subset \mathcal{A}_T$.

\begin{prop} Let $\Phi$ be an irreducible root system. The dimension of\, $\mathcal{C}_{\Omega}$ equals $|\Phi|$.
\end{prop}

\begin{proof}  Recall that a bounded subset $\Omega \subset \mathcal{A}_T$  defines the concave function
\begin{equation}
f_{\Omega}\colon\Phi \longrightarrow \mathbb{R}, \quad r \longmapsto - \left\lfloor \inf_{\theta \in \Omega} r(\theta) \right\rfloor.
\end{equation}

For $r \in \Phi$, let $H_r$ be the hyperplane in $\mathcal{A}_T$ where $r$ vanishes. Consider the unit sphere $S$ in $\mathcal{A}_T$. We fix an isomorphism of $T$ with a product of $\mathbb{G}_{m}$. This gives a basis of $Y(T)$ and $X(T)$. Now we may identify $\mathcal{A}_T$ with its dual space so that for a root $r \in \Phi$ and $\theta \in \mathcal{A}_T$, the evaluation $r(\theta)$ equals the inner product of $\theta$ with the image of $r$. Thus, after this identification, every root $r \in \Phi$ defines a canonical unit vector normal to $H_r$ meeting $S$ in $p_r$. Since $\Phi$ is irreducible,  $H_r=H_s$ implies $r={\pm} s$. Further, $- p_r=p_{-r}$.
So $r \mapsto p_r$ defines a bijection. To prove the proposition, we now consider the bounded subset 
\begin{equation}
\Omega:=\left\{p_r \mid r \in \Phi \right\}.
\end{equation}
So $|\Omega|=|\Phi|$. For any $r \in \Phi$, $\inf_{\theta \in \Omega} r(\theta)$ is attained at $p_{-r}$. We now let points vary in radial directions determined by $\{p_r\}$. In small open neighbourhoods of $p_r$, if $\Omega':=\{p_r' \}$ denotes a neighbouring set of points, we again see that
$\inf_{\theta \in \Omega'} r(\theta)$ is attained at $p'_{-r}$. This completes the proof.
\end{proof}

\subsection{Recovering the bounded group \texorpdfstring{$\boldsymbol{\cP_{\Omega}}$}{P\textunderscore Omega }over a DVR in terms of {\tt n}-parahoric groups}
Let $\Omega \subset \mathcal{A}$ be a bounded subset of the apartment. Without loss of generality, we may suppose that it is compact.

We make a small modification in the notation just for this lemma. Let $A_{1} = L[t]$  be a polynomial ring over a field $L$ that is not necessarily algebraically closed. Let $K_{1} = \Fract(A_{1})$. Similarly, let $A_{N} = L[z_{1}, \ldots, z_{N}]$ and $K_{N} = \Fract(A_{N})$. 
Let
\beqa\label{recovering}
{\cP}_{\Omega} := \left\langle T(L[t]),\, U_{r}\left(t^{m_r(\Omega)}L[t]\right),  r \in \Phi \right\rangle 
\eeqa
and for $\bvt=(\theta_{1},\ldots, \theta_{N})$,
\beqa
{\cP}_{\bvt} := \left\langle T(A_{N}]), U_{r}\left(\prod_{1 \leq j \leq N} z_{j}^{- \lfloor (r, \theta_{j}) \rfloor} A_{N}\right), \; r \in \Phi\right\rangle.
\eeqa 
The next lemma shows the natural relationship between {\tt n-bounded} groups associated to bounded subsets $\Omega$ in the affine apartment and {\tt N-parahoric} groups. We give it here for its general interest. 

\begin{lem} \label{regionalasn}
Let $ N = |\Phi|$ and let the inclusion $G(K_{1}) \subset G(K_{N})$ be obtained  by mapping the uniformizer $t$ to $\prod_{r \in \Phi} z_r$. Given a bounded domain $\Omega$, there exist an $N$-tuple $\bvt=(\ldots,\theta_r,\ldots)_{r \in \Phi}$ of points such that taking intersection in $G(K_{N})$, we have the equality
\begin{equation}\label{multiregion}
{\cP_{\Omega}} = G(K_1) \cap {\cP}_{\bvt}.
\end{equation}
\end{lem}

\begin{proof}
For every root $r \in \Phi$, we may choose a $\theta_r \in \Omega$ such that we have 
\begin{equation} \label{extremalcondition}
(r,\theta_r) \leq (r,\theta) \quad \forall \theta \in \Omega.
\end{equation}
Set $L[\Phi] :=L[z_r, r \in \Phi]$ and consider the $\tt N$-parahoric given by 
$$
\bvt=(\ldots,\theta_{j},\ldots)_{j \in \Phi}.
$$ 
Thus by definition, we have
\begin{equation}
{\cP}_{\bvt} = \left\langle T(L[\Phi]),\, U_r\left(\prod_{j \in \Phi} z_{j}^{- \lfloor (r, \theta_{j}) \rfloor}L[\Phi]\right), r \in \Phi \right\rangle.
\end{equation}

For $r \in \Phi$, consider the group $U_r\big(\prod_{j \in \Phi} z_{j}^{- \lfloor (r, \theta_{j}) \rfloor}L[\Phi]\big)$. By (\ref{extremalcondition}), we see that 
\begin{equation}
(r,\theta_r) \leq (r,\theta_{j}) \quad \forall j \in \Phi.
\end{equation} 
Thus $-\lfloor (r,\theta_r) \rfloor \geq - \lfloor (r,\theta_{j}) \rfloor$ for all $j \in \Phi$. In other words, we have
\begin{equation} \label{romega}
-\lfloor (r,\theta_r) \rfloor = -\left\lfloor \inf_{\theta \in \Omega} (r,\theta) \right\rfloor = m_r(\Omega).
\end{equation}

Now $K_{1} = L(t)$ and $A_{1} = L[t]$, so letting $t=\prod_{r \in \Phi} z_r$, for every root $r \in \Phi$, we get the desired order of zeros and poles. Further, $T(L[\Phi]) \cap T(K_{1})=T(A_{1})$. Since the group law is inherited from $G(K_{1})$,  the intersection is the parahoric subgroup ${\cP_{\Omega}}$.
\end{proof}

For $1 \leq i \leq n$, let $f_i\colon \Phi \rightarrow \mathbb{R}$ be concave functions. Let ${\bf F}:=(\ldots, f_i,\ldots)\colon \Phi \rightarrow \mathbb{R}^n$. In \eqref{evenmoregen}, we have  defined the $n$-concave subgroup ${\cP}_{\bf f}$  of $G(K_{n})$ in terms of generators.

\begin{lem}
Let $G(K_1) \subset G(K_{n})$ by mapping the uniformizer $t$ to $\prod_{1 \leq i \leq n} z_i$. Taking intersection in $G(K_{n})$, we have the equality
\begin{equation}\label{sup}
{\cP_{sup_{i} f_i}} = G(K_1) \cap {\cP}_{\bf f},
\end{equation}
\end{lem}

\begin{proof} The proof is almost identical to that of Lemma~\ref{regionalasn}.
\end{proof}

Let $\mathcal{C}^{\ext}_{\Omega}$ denote the cone defined by \begin{enumerate*}\item taking sums of concave functions of the form $\sum_{i=1}^m f_{\Omega_i}$, where we let $m \geq 1$ and the bounded regions $\Omega_i \subset \mathcal{A}_T$ vary, \item changing the uniformizers by the substitution $t=z^n$ or $t^n=z$ and \item  taking supremums.\end{enumerate*}

We have the obvious inclusions
\begin{equation}
\mathcal{C}_{\Omega} \subset \mathcal{C}^{\ext}_{\Omega} \subset \mathcal{C}. 
\end{equation}
\begin{ex} The first inclusion is strict for the concave function \eqref{mG2} by the claim in Example~\ref{B2type3}.
\end{ex}

For irreducible reduced root systems, the notions of quasi-concave and concave functions coincide. Let $f$ be any concave function defined more generally on $\Phi$. Let $E$ be the set of points in $\mathcal{A}_T$ fixed by $G_f$. For a root $\alpha \in \Phi$, let  
\begin{equation} \label{setalpha}
  S(\alpha):= \left\{\left(\alpha_i,\lambda_i\right)_{\{1 \leq i \leq r\}} \in \cup_{r \geq 1} \left(\Phi \times \mathbb{Q}^*_+\right)^r \;\Big|\; \sum_{1 \leq i \leq r} \lambda_i \alpha_i=\alpha \right\}
\end{equation}
denote all possible ways of writing $\alpha$ in the rational positive cone of some subset of roots in $\Phi$. Let $f'$ be the function on $\Phi$ defined by  
\begin{equation} \label{fdash}
  f'(\alpha):= \inf_{ S(\alpha) } \;
  \left\lceil \sum_{1 \leq i \leq r} \lambda_i f(\alpha_i) \right\rceil.
\end{equation}
By  \cite[Proposition 3.3]{courtes}, we have $f'=f_E$ for some subset $E$ of the apartment.

\begin{ex} For the case $G = \SL_2$, the three cones coincide because any concave function is of the form $f_{\Omega}$ (see Proposition~\ref{btremarque}). Indeed, by \cite[Proposition 3.3]{courtes}, for $G = \SL_2$, we have $S( \alpha)=\{(\alpha,1)\}$ and $S(-\alpha)=\{ (-\alpha,1) \}$. Thus for $G = \SL_2$, the subset $E$ determines $f$ in return.
\end{ex}

\begin{prop} We now return to the case $G = G_{2}$. For the concave function $\cm$ in \eqref{mG2},  $\cm'$ equals $\cm$ except at $-\alpha_{1}-\alpha_{2}$, where it is $1$. The region is the dilation of the alcove $\mathbf{a}_0$ by a factor of $2$, \textit{i.e.} enclosed by the origin, $2 \theta_{\alpha_1}$ and 
$2 \theta_{\alpha_2}$.
\end{prop}
\begin{proof} We give an indirect proof. Since $\cm$ does not come from a region, we have the inequality $\cm' < \cm$. Now we can check that the concave function coming from the dilation of $\mathbf{a}_0$ by a factor of $2$ is as stated above. Since there cannot be any other concave function between them, this must be $\cm'$.
\end{proof}

\begin{rem} \label{coneinclusions} It would be interesting to understand the inclusion $\mathcal{C}^{\ext}_{\Omega} \subset \mathcal{C}$ better. Given bounded regions $\{ \Omega_{1}, \ldots, \Omega_{n} \}$, consider the concave function $f$ taking $r$ to 
$\sum m_r(\Omega_{i})$. It would be interesting to understand the region corresponding to $f'$ as a function of $\{ \Omega_{i} \}$.
\end{rem}

Recall from \eqref{evenmoregen} that  if $f\colon\Phi \to {\mathbb R}$ is a concave function, then we can associate to it the bounded subgroup ${\cP}_{f} := \langle T(\cO), U_r (z^{f(r)} \cO), r \in \Phi \rangle$ of $G(K)$. Note that $\gG_{f}(\cO) = {\cP}_{f}$. 

Following \cite[Section~4.5.2]{bruhattits} and \cite[Section~2, p.~540]{courtes}) we say that a concave function $f$ is {\em optimal} if for any concave function $f'\colon\Phi \to {\mathbb R}$ such that $f' > f$, we have a strict inclusion ${\cP}_{ f'} \subsetneq {\cP}_{ f}$. However, in the context of the present paper, it is easy to see that if $f$ is a concave function, then so is $\lceil f \rceil$, which is in fact optimal. Moreover, the bounded groups $\gG_{f}(\cO)$ coincides with $\gG_{\lceil f \rceil}(\cO)$. So we may as well work with optimal concave functions.

In \cite[Remarques 3.9.3]{bruhattits3} it is remarked that on root systems of type $A_{n}$, any optimal concave function is  of the form $f_{\Omega}$ for an {\em enclosed} bounded subset $\Omega$ of the apartment (see \cite[Section~2.4.6]{bruhattits1} for the definition of ``enclosed'' subsets). For the sake of completeness, we give below a proof of this statement.

\begin{prop}\label{btremarque} Let $G$ be of type $A_{n}$. Any optimal concave function $f$ is of {\tt type II}, \textit{i.e.} of the form $f_{\Omega}$ for a {\em enclosed} bounded subset $\Omega$ of the apartment. Thus we have $\mathcal{C}_{\Omega} = \mathcal{C}$; see Remark~\ref{coneinclusions}.
\end{prop}

\begin{proof} For the convenience of the reader, we recast the notions of \cite{courtes}, which were developed for connected reductive groups over local fields, into the setting of this paper. Recall that the residue fields of our local rings are perfect (see Section~\ref{charassum}).
  For a root $\alpha \in \Phi$, the $c(\alpha)$ and $v(\alpha)$ may be assumed to be $0$ and $1$, respectively, and further an optimal concave function takes only integral values.  By  \cite[Proposition 2.3]{courtes}, concavity is equivalent to pseudo-concavity since our group is split over $K$. A function is of the form $f_{\Omega}$ if and only if it is strongly concave according to \cite[Proposition 3.3 ]{courtes}. 

For a root $\alpha \in \Phi$, let $S'(\alpha) \subset S(\alpha)$ (see \eqref{setalpha}) be the restricted set such that 
\begin{itemize}
\item the $\alpha_i$ are linearly independent;
\item for every $i \neq j$, we have $(\alpha_i,\alpha_j) \geq 0$; we can assume that $\alpha_i+\alpha_j$ is not a root;
\item for every $i$, $\lambda_i<1$.
\end{itemize}

In \cite[Lemma 5.1]{courtes} F.~Court\`es gives an equivalent reformulation of strong concavity in terms of a technical numerical inequality which should be checked for each element in $S'(\alpha)$, where $\alpha$ varies in $ \Phi$. We claim that when $G=A_n$,  this numerical condition is vacuously true because $S'(\alpha)= \emptyset$. We see this as follows.

Let $\alpha=\epsilon_l - \epsilon_m$ be a root of $A_n$. Since $\alpha=\sum_{i=1}^r \lambda_i \alpha_i=\epsilon_l -\epsilon_m$ and $0 < \lambda_i <1$, we see that $r \geq 2$ and there exist at least two indices $1 \leq i_1 < i_2 \leq r$ such that $\alpha_{i_1}=\epsilon_l - \epsilon_{j_1}$ and $\alpha_{i_2}=\epsilon_l - \epsilon_{j_2}$. We suppose firstly that $j_1$ is different from $m$. In this case, there must exist another index $1 \leq i_3 \leq r$ such that $\alpha_{i_3}=\epsilon_{j_1}-\epsilon_{j_2}$. But then either $\alpha_{i_1}+\alpha_{i_3}=\epsilon_l - \epsilon_{j_2}$ is a root of $A_n$, or $\alpha_{i_1}=-\alpha_{i_3}$. Both these possibilities violate the conditions for $\{(\lambda_{i_1}, \alpha_{i_1}),(\lambda_{i_2}, \alpha_{i_2}) \}$ to be a subset of an element in $S'(\alpha)$. Now suppose that $j_1$ equals $m$. Then $\alpha_{i_1}$ equals $\alpha$. So we have $\sum_{i=1}^r \lambda_i \alpha_i - \lambda_{i_1} \alpha_{i_1}= (1-\lambda_{i_1}) \alpha=(1-\lambda_{i_1})(\epsilon_l - \epsilon_m) \neq 0$. 
Repeating the above argument, we find that $\alpha_{i_2}$ equals $\alpha$ as well. But then $\alpha_{i_1}$ and $\alpha_{i_2}$ are the same. This violates the conditions for $\{(\lambda_{i_1}, \alpha_{i_1}),(\lambda_{i_2}, \alpha_{i_2}) \}$ to be a subset of an element in $S'(\alpha)$.
\end{proof}

\begin{rem} Let $f\colon\tilde{\Phi} \rightarrow \mathbb{Z}$ be a concave function. We have $f(r) =f(0 + r) \leq f(0) + f(r)$. Thus, $0 \leq f(0)$. Let us fix a value $f(0)$ of $0 \in \tilde{\Phi}$. For each simple root $\alpha \in S$, by the condition $f(0)=f(\alpha+ -\alpha) \leq f(\alpha)+f(-\alpha)$, we have
  $$(f(\alpha),f(-\alpha)) \in \left\{(x,y) \in \mathbb{Z}^2\mid x+y \geq f(0) \right\}.$$
  On the other hand, the condition $f(r_{1})=f(r_{1}+r_2-r_2) \leq f(r_{1}+r_2)+f(-r_2)$ implies
  $$f(r_{1})-f(r_2) \leq f(r_{1}+r_2). \leq f(r_{1})+f(r_2).$$
  Thus, for an integral concave function $f$, we have infinitely many choices for values on $ S^{\pm} \cup \{0\}$ but only finitely many choices for non-simple roots.
\end{rem}

\begin{part}
{Schematization}
\end{part}

\section{The {\tt n}-parahoric Lie algebra bundle \texorpdfstring{$\boldsymbol{\mathcal{R}}$}{R} on an affine space}  \label{constructionR}
The aim of this section is to present an {\em ab initio} construction of the {\tt n-parahoric} Lie algebra bundle on the affine $n$-space together with a full description of the fibres at points of higher depth. This played the initially important role for us to guess the definition of the {\tt n}-parahoric group. The reader who is interested only in the group scheme constructions may skip this section.

For $1 \leq i \leq n$, let $H_i \subset {\bf A} := \mathbb{A}^{n}_k$ denote the coordinate hyperplanes.
Let ${\bf A_{0}} \subset {\bf A}$ be the open subscheme which is the complement of the union of hyperplanes.  Let $\bvt := (\theta_{1}, \ldots, \theta_{n}) \in  \cA^{n}$ be a point of the $n$-apartment satisfying the conditions of Section~\ref{charassum}.
The aim of this section is to construct an {\em {\tt n}-parahoric Lie algebra} bundle $\cR$ on  ${\bf A}:= \mathbb{A}^{n}_k$ such that its restrictions to the generic points of the hyperplane $H_{i}$ are the parahoric Lie algebra bundles associated to the weight $\theta_{i}$ and to describe it on the whole of ${\bf A}$. 

\subsection{Elementary remarks on reflexive sheaves} 
We begin with three elementary  lemmas which should be well known.

\begin{lem}\label{langton}
  Let $E$ be a locally free sheaf on an irreducible smooth scheme $X$. Let $\xi \in X$ be the generic point, and let $W \subset E_{\xi}$ be an $\mathcal O_{\xi}$-submodule. Then there exists a unique coherent subsheaf $F \subset E$ such that $F_{\xi} = W$ and $Q := \Coker(F \hra E)$ is torsion-free. Moreover, $F$ is a reflexive sheaf.
\end{lem} 

\begin{proof} Define a sheaf $\tilde{F}$ by the following condition: its sections on an affine open $U \subset X$ are given by $\tilde{F}(U) := E(U) \cap W$.  Then it is easily seen that $\tilde{F}$ defines a coherent subsheaf $F \subset E$ and that it is the maximal coherent subsheaf of $E$ whose fibre over $\xi$ is $W$. To check that $Q$ is torsion-free, let $T$ be the torsion submodule of $Q$. Let $K := \Ker(E \to Q/T)$. Then since $\xi \notin \Supp(T)$, $K_{\xi} = W$, and hence, by the maximality of $F$, we have $K = F$. Since $E$ is locally free,  we have $F^{\vee \vee}/F \hra E^{\vee \vee}/F=E/F$. But since $F^{\vee \vee}/F$ is only torsion and $E/F=Q$ is torsion-free, it follows that $F$ is automatically a reflexive sheaf.\end{proof}

\begin{lem}\label{langton1}
  Let $X$ be as above and $i\colon U \hookrightarrow X$ an open subset such that $X \setminus U$ has codimension at least~$2$ in $X$. Let $F_{U}$ be a reflexive sheaf on $U$. Then $i_{*}(F_{U})$ is a reflexive sheaf on $X$ which extends $F_{U}$.
\end{lem}

\begin{proof} By \cite[Corollaire 9.4.8]{ega}, there exists a coherent $\mathcal O_{X}$-module $F_1$ such that $F_1|_{U} \simeq F_{U}$. Set $F := F_1^{**}$ to be the double dual. Then  $F$ is reflexive, and also since $F_{U}$ is reflexive, $F|_{U} \simeq F_{U}$. Hence, we have  $i_{*}(F_{U}) = F$; see \cite[Proposition 1.6]{hartshorne}. \end{proof} 

\begin{lem}\label{langton2}
  Let $X$ be integral and factorial. Then any reflexive sheaf of rank $1$ is locally free.
\end{lem}

\begin{proof} (\textit{cf.} \cite[Proposition 1.9]{hartshorne}) Let $\cF$ be a reflexive sheaf of rank $1$. Then since $X$ is normal, there is an {\em open} $U \subset X$ such that $\codim_{X}(U) \geq 2$ and $\cF_{U}$ is locally free. Since $X$ is locally factorial, we have $\Pic(X) \to \text{Pic}(U)$ is bijective. Hence, there is a locally free sheaf $\mathcal L$ of rank $1$ on $X$ such that $\mathcal L_{U} = \cF_{U}$. It is clear that $\mathcal L \simeq \cF$ on $X$. \end{proof}

\subsection{The construction of \texorpdfstring{$\boldsymbol{\cR}$}{R}}\label{constructionofr}

\begin{thm}\label{gpschLiestab}
Associated to $\bvt$,  there exists a canonical Lie algebra bundle $\mathcal{R}$ on ${\bf A}$ which extends the trivial bundle with fibre $ \mathfrak g$ on $ {\bf A_{0}} \subset {\bf A}$; furthermore, in a formal neighbourhood $V := \spec(\mathcal O_{n})$ of the origin in the affine $n$-space ${\bf A}$,  we have the identification of functors from the category of $k$-algebras to $k$-Lie algebras 
\begin{equation}\label{loopfunct}
 L_{n}^+( \mathcal R\mid_{V{}})=L_{n}^+( \mathfrak{P}_{\bvt}).
\end{equation}
\end{thm}

\begin{proof} 
Let $b_{i \alpha} \in \mathbb{Q}$ be defined by the relations
\begin{equation}
\theta_{i}:= \sum_{\alpha \in S} b_{i \alpha} \omega_{\alpha}^{\vee} \quad \text{for }  i \leq n.
\end{equation}
Let $d \in \mathbb{N}$ be such that $d b_{i \alpha} \in \mathbb{Z}$ for all $1 \leq i \leq n$ and $\alpha \in S$.

Consider an inclusion of $k$-algebras $B_{0} \subset B$, where
\begin{equation} B:= k\left[y_{i}^{\pm 1}\right]_{1 \leq i \leq n}, \quad 
B_{0}:= k\left[x_{i}^{\pm 1}\right]_{1 \leq i \leq n}
\end{equation}
and  $\{y_{i}^{d}=x_{i} \}_{1 \leq i \leq n}$. Let $
{\bf Y^{0}}:=\spec(B)$ and ${\bf A_{0}} =\spec(B_{0})$, and let $p\colon {\bf Y^{0}} \to {\bf A_{0}}$ be the natural morphism. We define the ``roots'' map
\begin{eqnarray}
\mathfrak r\colon{\bf Y^{0}} \lra {\bf A_{0}} \quad\text{as} \\
{\mathfrak r}^{\#}\left(\left(x_{i}\right)\right):=\left(y_{i}\right).
\end{eqnarray}
Note that as a map between tori, $\mathfrak r$ is an isomorphism. Let $T_{\ad}=\spec(k[z_{\alpha}^{\pm 1}])$. We define the action map 
\begin{eqnarray}
\mathfrak a\colon{\bf Y^{0}} \lra T_{\ad}  \quad\text{as} \\
{\mathfrak a}^{\#}\left(\left(z_{\alpha} \right)\right):=\left( \prod_{1 \leq i \leq n} y_{i}^{d b_{i \alpha}} \right).
\end{eqnarray}

We consider the map
\begin{eqnarray}\label{thebbAd}
\Ad \circ \mathfrak a\colon {\bf Y^{0}} \times \mathfrak{g} \lra {\bf Y^{0}} \times \mathfrak{g} \\
({\bf t},x) \longmapsto \left({\bf t}, \Ad\big(\mathfrak a({\bf t})\big)(x)\right).
\end{eqnarray}

We define the embedding of modules
\begin{equation}
  j\colon \mathfrak{g}(B_{0}) \longhookrightarrow\mathfrak{g}(B) \xrightarrow{\Gamma(\Ad \circ \mathfrak a)} \mathfrak{g}(B), 
\end{equation}
where the second map is the one induced by \eqref{thebbAd} on sections.
The $T$-weight space decomposition of $\mathfrak{g}$ induces a $T$-weight space decomposition on $\mathfrak{g}(B_0)$ by viewing the latter as sections of the  trivial bundle on $\bf A_{0}$ with fibres $\mathfrak{g}$.

Let $B^+  = k[y_{i}]_{1 \leq i \leq n}$ and $B_{0}^+:=k[x_{i}]_{1 \leq i \leq n}$.  Taking intersection as Lie submodules of $\mathfrak{g}(B)$, we define the following $B_{0}^+$-submodule of $\mathfrak{g}(B_{0})$:
\begin{equation}\label{sheafofR}
\mathcal{R}\left(B_{0}^+\right):= j^{-1}\left(j\left(\mathfrak{g}(B_{0})\right) \cap \mathfrak{g}\left(B^+\right)\right).
\end{equation}
In words, it is the set of elements in $\mathfrak{g}(B_{0})$ which lie in $\mathfrak{g}(B^+)$ under $j$.
Taking $E$ as the bundle defined by the free $B_{0}^+$-module $\mathfrak{g}(B^+)$ and the submodule $W:= j(\mathfrak{g}(B_0))$ in Lemma~\ref{langton}, we observe that $\mathcal R$ is a reflexive sheaf   on the affine embedding ${\bf A_{0}} \hookrightarrow {\bf A} := \spec(B_{0}^+) (\simeq {\mathbb A}^{n})$. Further, this is a sheaf of Lie algebras  with its Lie bracket induced from $\mathfrak{g}(B)$.

We now check that $\mathcal{R}$ is locally free on ${\bf A}$. Observe  that the intersection \eqref{sheafofR} also respects $T_{\ad}$-weight space decomposition on the sections. More precisely, for any root $r \in \Phi$, since by definition we have $\mathcal{R}(B_{0}^+)_r=j^{-1}\big(j(\mathfrak{g}_r(B_{0})) \cap \mathfrak{g}_r(B^+)\big)$,  we get the following equalities:
\begin{equation} \label{wtspacedecomp}
j^{-1}\left(j\left(\mathfrak{g}\left(B_{0}\right)\right) \cap \mathfrak{g}_r(B^+)\right) = j^{-1}\left(j\left(\mathfrak{g}_r\left(B_{0}\right)\right) \cap \mathfrak{g}(B^+)\right) = \mathcal{R}\left(B_{0}^+\right)_r.
\end{equation} 

Since $\mathcal{R}$ is reflexive, there is an open subset $U \subset {\bf A}$ with complement  of codimension at least $2$ such that the restriction $\mathcal{R}':= \mathcal{R}|_{U}$ is locally free. 
Clearly, $U$ contains ${\bf A_{0}}$ and the generic points $\zeta_{\alpha}$ of the divisors $H_{\alpha}$. This gives a  decomposition on $\mathcal{R}'$ obtained by restriction from $\mathcal{R}$.   The locally free sheaf $\mathcal{R}'$ is a direct sum of the trivial bundle $\Lie(T_{\ad}) \times U$ (coming from the $0$-weight space) and the invertible sheaves coming from the root decomposition. Now since invertible sheaves extend as invertible sheaves across codimension at least~$2$, see Lemma~\ref{langton2}, the reflexivity of $\mathcal{R}$ implies that this direct sum decomposition of $\mathcal{R}'$ extends to  ${\bf A}$. Hence  $\mathcal{R}$ is locally free.

We introduce some convenient notation to check the main assertion of the theorem.
Recall that $\lambda_i\colon \mathbb{G}_m \rightarrow {\bf A_{0}}$ is the $\supth{i}$ axis of ${\bf A_{0}}$. For example, $\lambda_i$  may be expressed in $n$ coordinates as follows:
\begin{equation}
\lambda_i({\tt x}) = (1,\ldots, {\tt x}, \ldots 1)
\end{equation}
with ${\tt x}$ at the coordinate corresponding to $1 \leq i \leq n$. The map $\mathfrak r\colon {\bf Y^{0}} \to {\bf A_{0}}$ is abstractly an isomorphism of tori, and we can therefore consider the $1$-PS $\boldsymbol{\lambda_i}\colon\mathbb G_{m} \to {\bf Y^{0}}$ defined by $\boldsymbol{\lambda_i} := \mathfrak r^{-1} \circ \lambda_i$.  We observe the following:
\begin{eqnarray}
p \circ \boldsymbol{\lambda_i} = \lambda_i, \\
\label{3} \mathfrak a \circ (\boldsymbol{\lambda_1}, \ldots, \boldsymbol{\lambda_n})= (\theta_1,\ldots,\theta_n) =\bvt.
\end{eqnarray}

Recall that $V = \spec({\cO}_{n})$ denotes a formal neighbourhood of the origin. Let us check the isomorphism $L_{n}^+({\mathcal{R}}|_{V}) \simeq L_{n}^+(\mathfrak P_{\bvt})$ by first evaluating at $k$-valued points. Since $d\cdot {\bvt}$ is an $n$-tuple of $1$-PSs of ${\bf A_{0}}$, we see that
 $\bvt(\ldots,{\tt x}_{i},\ldots)=\bvt(\ldots,{\tt y}^d_{i},\ldots)=(d \bvt)(\ldots,{\tt y}_{i},\ldots)$ has become integral in $\{{\tt y}_i \}$. 

 By \eqref{sheafofR}, a section $\text{\cursive s}$ of $\mathcal R$ over $V$  is first of all given by a local section $\text{\cursive s}_{K}$ over the generic point $\spec(K_{n})$ of $\spec(\cO_{n})$. More precisely, this is firstly an element in  $j(\gfr(K))$. Hence  this element may be written as
 $$
 \text{\cursive s}_{K} =\Ad \left( \mathfrak{a} \prod_{1 \leq i \leq n} \boldsymbol{\lambda_i} ({\tt x}_{i}) \right)(x_{K}) = \Ad\left(\prod_{1 \leq i \leq n} \theta_i ({\tt x}_{i})\right)(x_{K}).$$

Therefore, the membership of $\text{\cursive s}$ in $\mathfrak{g}(B^+)$, see \eqref{sheafofR}, gets interpreted as follows:
\begin{equation}\label{conditions1}
  \Ad\left(\prod_{1 \leq i \leq n} \theta_i({\tt x}_{i}) \right) (x_{K}) \in L \mathfrak{g} (B^+).
\end{equation} 

But this is exactly the Lie algebra version  of the conditions ``$\lim_{s \rightarrow 0} \text{exists}$''  for $\eta=(1,\theta)$ in the observation \eqref{observation} and Section~\ref{higherobservation}, for each individual ${\tt x}_{i}$.
In our situation we have assumed $R=k$, and so by \eqref{3}, the above is equivalent to the membership
\begin{equation}  \text{\cursive s} \in   \mathfrak{P}_{\bvt_{}}.
\end{equation}
This proves the assertion for $k$-valued points.
The above proof goes through for all $k$-algebras because the  underlying module structure  is already defined on $k$, and so is Lie bracket. This completes the proof of the theorem.
 \end{proof}
 
\begin{Cor} \label{restrictiontocurves}
Let the notation be as in Theorem~\ref{gpschLiestab}. Then $\lambda := \sum_{i \leq n} \lambda_i$ is a  $1$-PS of ${\bf A_{0}}$ defining the {\it diagonal curve} $D$  whose limit point is the origin. Let $\mathcal{D}$ be a formal neighbourhood of the origin in $D$.
We get the  identification \textup{(}see \eqref{allunifequalLie} for the notation\textup{)}
\begin{equation}\label{loopfunct1}
 L^+\left( \mathcal R\mid_{\mathcal{D}}\right)=L^+\left( \mathfrak{P}_{\bvt}^{\diag}\right) 
\end{equation}
of the specialized objects. 

Let $\lambda=\sum k_{i} \lambda_{i}$, where $k_{i} \geq 0$ but not all are zero. Let $\mathcal{D}$ be the formal neighbourhood of the origin associated to the curve defined by $\lambda$. Let $f\colon \Phi \rightarrow \mathbb{Z}$ be the concave function defined by the assignment $r \ms \sum k_{i} m_r(\theta_i)$. Then for any $k$-algebra $R$, we have
\begin{equation}
L^+( \mathcal R\mid_{\mathcal{D}})(R)=\left\langle \mathfrak{t}(R \llbracket t \rrbracket), \mathfrak{u}_r\left(t^{f(r)} R \llbracket t \rrbracket\right), r \in \Phi \right\rangle.
\end{equation}
\end{Cor}

\begin{proof} Note that restriction to $\mathcal{D}$ followed by taking global sections of $\mathcal{R}$ sets all variables $z_i$ to $t$ in the sections of $\mathcal{R}$ on $U$. Now the first assertion follows by the definition \eqref{allunifequalLie}. The restriction of $\mathcal{R}$ to $\mathcal{D}$ in the second case followed by taking global sections corresponds to  setting $z_i:=t^{k_i}$ in the sections of $\mathcal{R}$ on $U$.
\end{proof}

\begin{rem} \label{torickawamata}  In the setting of Theorem~\ref{gpschLiestab}, we may view $p\colon\spec(B^{+}) \to \spec(B_{0}^{+})$ as a ramified covering space of affine toric varieties induced by an inclusion  of lattices. The Galois group for this covering is the dual of the quotient of lattices.  The computation in~Theorem~\ref{gpschLiestab} can be seen as a higher-dimensional analogue of \cite{base}.
\end{rem}

\section{Schematization of {\tt n}-parahoric groups} \label{schematization1}
In this section and the next couple of sections, we will freely use the theorems from  \cite{sga3}, the paper of Bruhat--Tits \cite{bruhattits} and M.~Raynaud's book \cite{raynaud}. We begin by recalling a few essential facts about Weil restrictions of scalars and related things.

\subsection{Weil restrictions  and Lie algebras}
Let $X$ be an arbitrary $k$-scheme.
For an affine (or possibly ind-affine) group scheme $\cH \rightarrow X$, we denote by $\Lie(\cH)$ the sheaf of Lie algebras on $X$ whose sections on $U \rightarrow X$ are given by
\begin{equation} \label{Lieaffinegpsch}
  \Lie (\cH)(U)= \ker(\cH(U \times k[\epsilon]) \longrightarrow \cH(U)).
\end{equation}

\begin{lem}\label{Weilrestriction}
Let $p\colon \tilde{X} \ra X$ be a finite flat map of Noetherian schemes. Let $\Res_{\tilde{X}/X}$ denote the ``Weil restriction of scalars'' functor. Let $\cH  \ra \tilde{X}$ be an affine group scheme. Then we have a natural isomorphism
\begin{equation}
\Lie\left(\Res_{\tilde{X}/X} \cH\right) \simeq \Res_{\tilde{X}/X} \Lie(\cH).
\end{equation}  
When $p$ is also Galois with Galois group $\Gamma$,  under tameness assumptions \textup{(}see~Section~\ref{tameness}\;\textup{)}, we have a natural isomorphism
\begin{equation}
\Lie\left((\Res_{\tilde{X}/X} \cH)^{\Gamma}\right) \simeq \left(\Res_{\tilde{X}/X} \Lie(\cH)\right)^{\Gamma}.
\end{equation}
\end{lem}

\begin{proof} See \cite[Section~A.7, p.~533]{cgp} and \cite[Proposition~3.1]{edix}, respectively.
\end{proof}

\begin{notat}\label{weilresaspushforward}
  Following \cite{blr},  we will denote the Weil restriction of scalars  as $p_{*}$. The functor which composes with taking group invariants will be denoted by $p^{\Gamma}_{*}$. This is often called the ``invariant direct image'' functor following \cite{base}. The conclusion of Lemma~\ref{Weilrestriction} can be expressed as 
\begin{equation}
\Lie\left(p^{\Gamma}_{*}(\cH)\right) \simeq p^{\Gamma}_{*}(\Lie(\cH)). 
\end{equation}
\end{notat}

\subsection{Some constructions from SGA} \label{onWfromsga} We recall briefly some facts about vector group schemes from \cite{sga3}. For a coherent $S$-module $\mathcal F$, we recall, see \cite[Expos\'e I, Section~4.6.3 and Proposition~4.6.5]{sga3}
(see also \cite[Section~1.4.1]{bruhattits}), the functor from the category of $S$-schemes to Abelian groups given by 
\beqa\label{thewconstruct}
W\colon S' \longmapsto \Gamma(S', \cF \otimes_{\cO_{S}} \cO_{S'}).
\eeqa 
When $\cF$ is locally free, this  is represented by a smooth group scheme $W(\cF)$ defined by the affine scheme $\Sym(\cF^{\vee})$. Its Lie algebra is canonically isomorphic to $\cF$. In fact, there is a canonical isomorphism
\beqa
W(\cF) \simeq W(\Lie(W(\cF))).
\eeqa

Following \cite{sga3}, we will call $W(\cF)$ a {\em vector group scheme}, noting that 
\beqa
\cF \longmapsto W(\cF)
\eeqa
gives a functor from the category of locally free $\cO_{S}$-modules of finite rank to the category of vector $S$-groups. Observe that over a field, a vector group is simply $\mathbb G_{a}^{n}$ for some $n$. We now state a few important properties satisfied by the functor $W$:
\begin{enumerate}[label=(\alph*), ref=\alph*]
\item \label{summandproperty} If $\cF'$ is a locally free $S$-submodule of $\cF$, then $W(\cF')$ is a closed subgroup scheme of $W(\cF)$ if and only if $\cF'$ is a direct summand of $\cF$ (see \cite[Section~2.1, p.~148]{demgab}  and \cite[Section~1.4.1]{bruhattits}). 

\item \label{linebundleW} In particular, if $S = \spec(\Lambda)$ with fractional field $\tt F$, a projective $\cO_{S}$-module $\mathcal L$ of rank $1$ corresponds to an $S$-group scheme $W(\mathcal L)$ whose generic fibre is isomorphic to the additive group $\mathbb G_{a,\tt F}$.
\item It is easy to check that Weil restriction and taking invariants commute with the functor $W$.
\end{enumerate}

\subsection{Restriction of \texorpdfstring{$\boldsymbol{\mathcal{R}}$}{R} to formal neighbourhoods of generic points of coordinate hyperplanes~\texorpdfstring{$\boldsymbol{H_{i}}$}{H\textunderscore i}} \label{restriction} 
Let $C_{i}$  be the  coordinate curve  defined by the vanishing of $\lambda_{j}$ (see Section~\ref{constructionofr}) for $j \neq i$. Let $U_{i}$ be the formal neighbourhood of the origin in $C_{i}$. For $1 \leq i \leq n$, let $\theta_{i}$ be the point in the apartment corresponding to the hyperplane $H_{i}$. 

For each $\theta_{i}$, let $d_{i} \in \mathbb{N}$ be such that $d \theta_{i} \in Y(T)$. We first suppose that $\theta_{i}$ lies in the fundamental domain of $Y(T)$ in $\cA_T$. Recall (see Section~\ref{grstuff}) that by \cite[Proposition 5.1.2]{base}, there exists a ramified cover $q\colon U'_{i} \to  U_{i}$ of ramification index $d_{i}$ and Galois group $\Gamma$,  together with a $\Gamma$-equivariant $G$-torsor $E_{i}$ such that the adjoint group  scheme $\cH_{i}=E_{i}(G)$ has simply connected fibres isomorphic to $G$ and we have the {\it canonical} identification of $U$-group schemes 
\beqa\label{frombalses1}
q_{*}^{\Gamma}(\cH_{i}) \simeq  \gG_{\theta_i}.
\eeqa

 For the general case, move the points $\theta_{{\tt i}}$ into the fundamental domain for the action of $Y(T)$ on $\cA_T$ by elements $h_{i} \in Y(T)$.  These $h_{i} \in Y(T)$ are uniquely determined. By an abuse of notation, we may view them as $h_{i} \in T(K)$ as well.  We now consider the modified embedding of $G$ in the generic fibre of $\gG_{h_{i}.\theta_{i}}$ given by conjugation by $h_{i}$. Then the {\em schematic closure} via this embedding in $\gG_{h_{i}.\theta_{i}}$ gives the group scheme $\gG_{\theta_{i}}$. Now \eqref{frombalses1} generalises without any restriction on $\theta_{i}$. Henceforth in this paper, we fix the identification of $U$-group schemes thus obtained:
 \beqa\label{frombalses}
\phi(h_{i})\colon q_{*}^{\Gamma}(\cH_{i}) \simeq  \gG_{\theta_i}.
\eeqa

We therefore get the following useful corollary to Theorem~\ref{gpschLiestab}.

\begin{Cor}\label{gpschLiestab3}
As sheaves of Lie algebras, we have an isomorphism 
\begin{equation} \label{Liestrendow} Lie \big(\phi(h_{i})\big)\colon q^{\Gamma}_{*}(E_{i}(\mathfrak g)) \simeq \mathcal{R}|_{U_{i}}.
\end{equation}
 \end{Cor}

\begin{proof} This is an immediate consequence of  Theorem~\ref{gpschLiestab}  and Lemma~\ref{Weilrestriction}. \end{proof}

\subsection{Group schemes on smooth varieties with normal crossing divisors}

We now prove a result  on constructing group schemes on varieties with normal crossing divisors  under some general hypotheses. This plays a role in this paper as a useful substitute for  the Artin--Weil theorem on the existence of group scheme structures stemming from birational group laws. Let $\bvt := (\theta_{1}, \ldots, \theta_{n})$ be an $n$-tuple of rational points in the affine apartment $\mathcal{A}_T$ satisfying the conditions of Section~\ref{charassum}. Let $k$ be a {\em perfect field} with characteristic assumptions as in Section~\ref{charassum}.

\begin{thm} \label{Artin-Weil-Kawamata}
Let $X$ be a smooth quasi-projective scheme over $k$, and let  $D = \sum_{i=1}^{n} D_i$
be a reduced normal crossing divisor. Let $\{\zeta_{i}\}_{i = 1}^{n}$ be the set of generic points of the components $\{D_{i}\}_{i = 1}^{n}$, and let $A_{i} = \mathcal O_{X, \zeta_{i}}$, $X_{i} := \spec(\hat{A}_{i})$ and  $X_{o} := X - D$. Let $\gG_{\theta_{i}}$ be the {\tt parahoric} group scheme on $X_{i}$ for each $i$  associated to the point $\theta_{i}$ of the affine apartment. Let $\mathcal R$ denote a Lie algebra bundle on $X$ which is such that $\mathcal R|_{X_{o}} \simeq \gfr \times X_{o}$ and $\mathcal R|_{X_{i}} \simeq \Lie(\gG_{\theta_{i}})$ for each $i$. Then there is a smooth, affine group scheme $\gG$ on $X$ with connected fibres, unique up to isomorphism, such that 
\begin{enumerate}
\item $\gG|_{X_{o}} \simeq G \times X_{o}$,
\item $\gG|_{X_{i}} \simeq \gG_{\theta_{i}}$ for each $i$,
\item $\Lie(\gG) \simeq \mathcal R$.
\end{enumerate}
\end{thm}

\begin{proof}
 Let $L_{i} :=\Fract(A_{i}) $ (resp.\ $\hat{L}_{i} := \Fract(\hat{A}_{i})$)  for each $i$.  By \eqref{frombalses}, for each $i$, there exists  a $\Gamma_{i}$-equivariant $G$-torsor $E_{i}$ on a ramified cover $q_{i}\colon \tilde{X}_{i} \to  X_{i}$ such that the adjoint group  scheme $\cH_{i}=E_{i}(G)$ has simply connected fibres isomorphic to $G$ and we have the identification of $A_{i}$-group schemes 
\beqa\label{invdi}
\phi(h_{i})\colon q_{i,*}^{\Gamma}(\cH_{i}) \simeq  \gG_{\theta_{i}}.
\eeqa
Furthermore, by Corollary~\ref{gpschLiestab3}, we also get the identifications 
\beqa\label{invdi1}
\Lie \big(\phi(h_{i}) \big)\colon  q_{i,*}^{\Gamma}(E_{i}(\mathfrak g)) \simeq  \Lie(\gG_{\theta_{i}}).
\eeqa
Let $\gG_{o} := G \times X_{o}$.
Let $X' \subset X$ be an open subset containing $X_{o}$ whose complement in $X$ has codimension at least~$2$.  We restrict the Lie algebra bundle $\mathcal R = \mathcal{R}_{X}$ further to $X'$. By assumption,  over the open subset ${ X_{o}}\subset X'$, we have $\mathcal R \simeq {X_{o}} \times \mathfrak{g}$. Observe that $\Aut(\gfr)=\Aut(G)$ since $G$ is simply connected. Further, since the local parahoric group schemes $\gG_{\theta_{i}}$ are assumed to be {\em generically split}, it follows that for each $i$, the transition functions of $\mathcal R$ give an isomorphism
$$\tau_{i}\colon\gG_{o} \times _{X_{o}} \hat{L}_{i} \simeq \gG_{\theta_{i}} \times _{ X_{i}} \hat{L}_{i}.
$$
We can now apply \cite[Section~6.2, Proposition D.4(b)]{blr} to each $\tau_{i}\colon \gG_{o} \times _{ X_{o}} L_{i} \times_{L_{i}}  \hat{L}_{i} \simeq \gG_{\theta_{i}} \times _{ X_{i}} \hat{L}_{i}$ to obtain by ``descent'' an affine, finite-type group scheme $\gG'_{\theta_{i}}$ over $\spec(A_{i})$ for each i, such that $\gG'_{\theta_{i}} \times_{A_{i}} \hat{A}_{i} \simeq \gG_{\theta_{i}}$ and $\gG'_{\theta_{i}} \times_{A_{i}} L_{i} \simeq \gG_{o} \times _{ X_{o}} L_{i}$. {\em By an abuse of notation, we will denote the group scheme $\gG'_{\theta_{i}}$ simply by $\gG_{\theta_{i}}$}.

Since the group schemes $\gG_{\theta_{i}}$ on $ \spec(A_{i})$ are of finite type, they can be extended (by clearing denominators) to open subschemes $X_{f_{i}}$ of $X$. By a further shrinking of this neighbourhood $X_{f_{i}}$ of $\zeta_{i}$, one can glue it to $\gG_{o}$ along the intersection. Since the group schemes are smooth and affine, 
we obtain the connected, smooth, affine group scheme
\begin{equation}\label{onx'} 
\gG' \lra X'.
\end{equation}

Transporting structures by Corollary~\ref{gpschLiestab3} (see also the appendix), the Lie algebra bundle $\mathcal R$ gets canonical parabolic structures at the generic points $\zeta_{i}$ of the components $\{D_{i}\}_{i}$.
By Definition~\ref{sestype}, for a (generic) parabolic vector bundle  with prescribed rational weights such as $\mathcal R^{{}}$, we firstly get a global Kawamata cover (see Section~\ref{kawa}) $p\colon Z \ra { X}$ ramified over $D$ with ramification prescribed by the weights $\{d_{i}\}$, see \eqref{dalpha}, with Galois group $\Gamma$ which ``realizes the local ramified covers $q_{i}$ at the points $\zeta_{i}$''; \textit{i.e.} the isotropy subgroup of $\Gamma$ at $\zeta_{i}$ is $\Gamma_{i}$.

We {\em claim} that we in fact get  an $\Gamma$-equivariant vector bundle $V$ on $Z$ such that 
\begin{equation}\label{gunja}
p_{*}^{\Gamma} (V) \simeq \mathcal R.
\end{equation}

Let $Z':=  p^{-1}(X')$. As the first step, following the reasoning in the appendix, it follows by the dimension $1$ strategy that there exists a $\Gamma$-equivariant vector bundle $V'$ on $Z'$ such that 
$p_{*}^{\Gamma} (V') \simeq \mathcal R\mid_{X'}$.

Furthermore, in formal neighbourhoods over the generic points $\zeta_{i}$ of the components $D_{i}$, we may identify $V'$ with $E_{i}(\mathfrak g)$ and $\Gamma$ with $\Gamma_{i}$. Also, over $Z_{o}:= p^{-1}(X_{o})$ we can identify $V'$ with $\mathfrak g \times Z_{o}$.  The transition functions of $\cR|_{X'}$ can be lifted to give transition functions for  $V'$, and we get a canonical $\Gamma$-equivariant Lie algebra bundle structure on $V'$ over $Z'$ (with fibre type $\mathfrak g$) such that 
$p_{*}^{\Gamma} (V') \simeq \mathcal R\mid_{X'}$
as Lie algebra bundles. 

Since the weights $\theta_{i}$ are all chosen to lie in the same affine apartment $\cA_T$, the local representations for the stabilizers $\Gamma_{i} \subset \Gamma$ in $G$ take values in the maximal torus $T$. This determines a uniform Cartan decomposition on the fibres of the Lie algebra bundle $V'$ on $Z'$ (compare with the proof of Theorem~\ref{gpschLiestab}, where one has similar arguments to check that $\mathcal{R}$ is locally free on ${\bf A}$). This gives a decomposition of the underlying vector bundle $V'$ as a direct sum of line bundles.

Now since invertible sheaves extend as invertible sheaves across codimension at least~$2$, see Lemma~\ref{langton2}, and since $\codim(Z \setminus Z') \geq 2$, the reflexivity of the double dual $V$ of $V'$ implies that this direct sum decomposition of  $V'$ extends term by term to  $Z$. Hence  the double dual $V$ is locally free. The equivariant structure on $V'$ canonically gives an equivariant structure on $V$, and it is easy to check that \eqref{gunja} holds. This completes the  proof of the claim.

The structure of Lie algebra bundle on $V'$ is such that the Killing form is non-degenerate everywhere on~$Z'$ by virtue of our assumptions; see Section~\ref{charassum} (\textit{cf.} \cite[Section~6.4, p.~199]{jantzen}). By a Hartogs-type argument (see for example \cite[Theorem 5.10.5]{ega4} or \cite[Theorem 3.8]{hartshornelocal}), the Lie bracket $[,]$ on $V'$ extends to a Lie bracket on the locally free sheaf $V$ with the Killing form being non-degenerate on the whole of $Z$. {\sl In other words, $V$ is now a locally free sheaf of Lie algebras on the whole of $Z$ with semisimple fibres; these fibres are isomorphic to the Lie algebra $\mathfrak g$ by the rigidity of semisimple Lie algebras}. We  denote this bundle of Lie algebras by $V_{\mathfrak g}$.

By  construction,  $V'$ has the structure of an equivariant Lie algebra bundle. The $[,]$ on $V_{\mathfrak g}$ is a section of $V_{\mathfrak g} \otimes V^{*}_{\mathfrak g} \otimes V^{*}_{\mathfrak g}$. For each $g \in \Gamma$, the sections $g\cdot [,]$ and $[,]$ agree on $Z'$. Hence they agree everywhere on $Z$; \textit{i.e.} the underlying equivariant structure on  $V_{\mathfrak g}$ gives $V_{\mathfrak g}$ the structure of an {\em equivariant Lie algebra bundle}. The transition functions of this equivariant Lie algebra bundle lie in $\Aut(\mathfrak g) = \Aut(G)$. Hence we obtain an equivariant $\Aut(G)$-torsor $\cE$ on $Z$. Note that $\cE$ is only \'etale locally trivial. Also observe that the associated Lie algebra bundle $\cE(\mathfrak g) := \cE \times^{\Aut(G)} \mathfrak g$ is isomorphic to  $V_{\mathfrak g}$ as a Lie algebra bundle.

Let $\cH := \cE \times ^{\Aut(G)} G$ denote the associated fibre space for the canonical action of $\Aut(G)$ on $G$. Since the group $\Aut(G)$ acts on $G$ as group automorphisms, it preserves the group structure, and hence $\cH$ is a group scheme. Further, since $\cE$ is \'etale locally trivial, by \'etale descent, it follows that $\cH$ is an {\em affine} smooth group scheme on $Z$. Since $\cE$ is an equivariant $\Aut(G)$-torsor, the group scheme $\cH$ is furthermore an equivariant group scheme on $Z$ for the action of the Galois group $\Gamma$, and so we define
\beqa\label{justnow}
{\mathfrak G}_{ X}:= p_{*}^{\Gamma}(\cH).
\eeqa 
By the properties of Weil restriction of scalars and taking invariants, it follows that ${\mathfrak G}_{X}$ is a {\em smooth, affine group scheme}.
Also, ${\mathfrak G}_{ X}|_{X'}=\gG'$.

Note that we have the identification $\Lie(\cH) \simeq V_{\mathfrak g}$ as Lie algebra bundles on $Z$. Since by Lemma~\ref{Weilrestriction}, the  ``invariant direct image'' functor commutes with taking Lie algebras, we moreover get isomorphisms of locally free sheaves of Lie algebras
\begin{equation}\label{liestr's}
\Lie({\mathfrak G}_{ X}) \simeq \Lie\left(p_{*}^{\Gamma}(\cH)\right) \simeq p_{*}^{\Gamma} \left(\Lie(\cH)\right) \simeq p_{*}^{\Gamma}\left(V_{\mathfrak g}\right) \simeq  \mathcal R^{{}}.
\end{equation}
This proves the third claim in the theorem.
\end{proof}

\subsection{The ``big cell''~structure on \texorpdfstring{$\gG$}{G}}\label{BigCell} Recall that the choice of the set of positive roots $\Phi^{+}$ determines a {\em big cell} of $G$ over $k$. 

By \eqref{onx'}, the group scheme $ \gG_{X´}$ is obtained by gluing the parahoric group schemes $\gG_{\theta_{i}}$ with $G \times X_{o}$ using the identity map $\id_{G}$.  Therefore, the parahoric group schemes $\gG_{\theta_{i}}$ come with a canonical {\sl big cell} which glues with the standard big cell of the split group scheme $G \times X_{o}$. This gives a big cell $\mathfrak B_{X´} = \mathfrak U_{X´}^{+} \times \cT_{X´} \times \mathfrak U_{X´}^{-}   \subset \gG_{X´}$. This also gives the decomposition
\beqa\label{bigcell 2}
\Lie\left(\gG_{X´}\right) = \Lie\left(\mathfrak U_{X´}^{+}\right) \oplus \Lie\left(\cT_{X´}\right) \oplus \Lie\left(\mathfrak U_{X´}^{-}\right).
\eeqa

In the next proposition, we show the existence of an open subscheme $\mathfrak B_{{\Phi}}$ of $\gG$  obtained in Theorem~\ref{Artin-Weil-Kawamata}. This open subscheme is analogous to the  {\sl big cell in a split reductive group scheme}; see \cite[Expos\'e XXII, Section~4.1]{sga3}. Following \cite{bruhattits}, we will call this open subscheme a {\em ``big cell''} of $\gG$. The {\sl big cell} in particular provides a neighbourhood of the identity section of the group scheme $\gG$.

\begin{prop}\label{bigcell1} We work in the setting of Theorem~\ref{Artin-Weil-Kawamata} and choose the transition function for the gluing to be  $\id_{G} \in \Aut(G)$.  Let $\Phi^{+}$ be a system of positive roots in $\Phi$. Let  $\mathfrak B_{\mathcal H}$ denote the big cell of\, $\mathcal H$. Let
  $$\mathfrak B_{\Phi} := p_{*}^{\Gamma}\left(\mathfrak B_{\mathcal H}\right)$$
  denote the invariant direct image of  $\mathfrak B_{\mathcal H}$. 

Then $\mathfrak B_{\Phi}$ restricts to the standard {\sl big cell} of $G$ over the open subset $X_{o}$. In fact,  there exist unipotent subgroup schemes $\mathfrak U^{\pm}$ and a toral subgroup scheme $\mathcal T$ of\, $\gG$ such that the morphism 
\begin{equation}
\mathfrak U^{-} \times \mathcal T \times \mathfrak U^{+} \lra \gG
\end{equation}
to $\gG$
induced by multiplication 
is an open immersion with image $\mathfrak B_{\Phi}$.  \end{prop}

\begin{proof} The notation is as in the proof of Theorem~\ref{Artin-Weil-Kawamata}. Since the Galois action of $\Gamma$ is via the fixed maximal torus $T$, we observe that the Cartan decomposition of the Lie algebra bundle is preserved under the functor that takes $\Gamma$-invariants. Notice that the functor $W$ of Section~\ref{onWfromsga} commutes with Weil restriction and taking invariants. Hence the root group schemes and the toral subgroup scheme of the reductive group scheme $\mathcal H$ can be pushed down by $p_{*}^{\Gamma}$ to give the corresponding root group schemes and toral group scheme below. The group structure on $\mathcal H$ induces one on the product of the root groups for the positive and negative roots (with a fixed and prescribed order) to give closed unipotent subgroup schemes on $\mathcal H$. These again carry the action of the Galois group $\Gamma$.

Hence these unipotent subgroup schemes  get pushed down by $p_{*}^{\Gamma}$ to give closed subgroup schemes $\mathfrak U^{\pm}$ of~$\gG$. Further, Weil restriction and taking invariants preserve the product structure, and this shows that the {\em big cell} structure on the reductive group scheme $\mathcal H$ goes down by Weil restriction and invariants to give the required {\em big cell} structure on $\gG$. \end{proof}

\subsection{Bruhat--Tits group scheme associated to concave functions}\label{btforconcave}
Let $\tilde{\Phi}:=\Phi \cup \{0 \}$.
Recall that a function $f\colon \tilde{\Phi} \rightarrow \mathbb{R}$ is a concave function such that $f(0) = 0$; see \eqref{concave}. 
We recall a few salient features of the main results from \cite[Section~4.6.2]{bruhattits} and \cite[Theorem~0.1]{yu} in our special setting, where $G_{K}$ is a generically split, simple, simply connected group over $K$, \textit{i.e.} $G_{K} \simeq G \times \spec(K)$.  

{\em Bruhat--Tits \cite[Section~4.6.2]{bruhattits}}: Let $\mathcal P_{f}$ be the subgroup of $G(K)$ generated by $U_{r,f(r)}$ and the subgroup $T(A)$. 

There exists a canonical affine smooth group scheme $\mathfrak G_{f}$ on $\spec({\cO})$ with generic fibre $G_{K}$ such that
\begin{enumerate}
\item $\mathfrak G_{f}(A) = \mathcal P_{f}$,
\item the multiplication morphism 
\begin{equation}\label{bigcellinconvexcase}
\left(\prod_{r \in \Phi^{-}} {\mathfrak U}_{r,f(r)} \right) \times \cT_{f,\cO} \times \left(\prod_{r \in \Phi^{+}} {\mathfrak U}_{r,f(r)} \right) \lra \mathfrak G_{f}
\end{equation}
\end{enumerate}
is an open immersion. The group scheme $\cT_{f,\cO}$ is the schematic closure of $T_{K}$ in $\mathfrak G_{f}$, and since $T_{K}$ is a split torus, $\cT_{f,\cO}$ is a connected multiplicative group scheme; see \cite[Section~10.1, Example 5]{blr}.  By computing the tangent space at the identity section which lies in the image of the immersion, \textit{i.e.}, in the {\em big cell} $\mathfrak B_{f}$ of~$\mathfrak G_{f}$, we see that 
\begin{equation}\label{cartanforconvex}
\Lie\left(\mathfrak G_{f}\right) = \bigoplus_{r \in \Phi} \Lie\left({\mathfrak U}_{r,f(r)}\right) \oplus \Lie\left(\cT_{f,\cO}\right).
\end{equation}

\subsection{On the Lie algebra bundle \texorpdfstring{$\boldsymbol{\cR}$}{R} in Theorem~\ref{Artin-Weil-Kawamata}}
We now prove a small but significant variant of Theorem~\ref{Artin-Weil-Kawamata}. In the hypothesis in Theorem~\ref{Artin-Weil-Kawamata}, we have assumed the existence of an {\tt n-parahoric} Lie algebra bundle $\cR$ on $X$. We will give sufficient general conditions for the existence of a Lie algebra bundle alone. 

\begin{prop}\label{onliestr}  Let the notation be as in Theorem~\ref{Artin-Weil-Kawamata}. Let ${\bf f} = \{f_{j}\}_{j = 1}^{n}$ be a collection of $n$-concave functions. Let $\hat{A}_{i}$ be the completion of $A_{i}$ at the maximal ideal. Suppose that we are given an assignment of\, {\tt BT}-group schemes $\gG_{f_{j}}$ on $X_{j} := \spec(\hat{A}_{j})$ for each $j$. Then there exists a Lie algebra bundle $\cR$ on $(X,D)$ such that $\mathcal R|_{X_{o}} \simeq \gfr \times X_{o}$, the Cartan decomposition on $\mathfrak{g}$ extends to $\mathcal{R}$ and $\mathcal R|_{X_{j}} \simeq \Lie(\gG_{f_{j}})$ for each $j$.

\end{prop}
\begin{proof} The notation is as in Theorem~\ref{Artin-Weil-Kawamata} and its proof. But unlike in Theorem~\ref{Artin-Weil-Kawamata}, we do not begin with a Lie algebra bundle. So we are free to choose the initial gluing data according to our requirements. Since the {\tt BT}-group schemes for us are always connected, smooth and generically split, we can choose (as in Proposition~\ref{bigcell1}) $\id_{G}$ as the map for gluing the product group scheme $\gG_{o} = G \times X_{o}$ with the $\gG_{f_{i}}$.

Exactly as in the proof of Theorem~\ref{Artin-Weil-Kawamata}, by a descent argument, we get smooth, affine group schemes $\gG_{f_{i}}$, now on $\spec({A}_{i})$ for each $i$, and hence we get a smooth, affine group scheme $\gG'$ on $X'$.
We now define
\beqa
\cR':= \Lie\left(\gG'\right), 
\eeqa
which is a Lie algebra bundle on $X' \subset X$.

Since the gluing functions are $\id_{G}$, exactly as in the proof of Proposition~\ref{bigcell1}, we also have an extended {\sl big cell} structure on the group scheme $\gG'$ over the whole of $X'$. As in  \eqref{bigcell 2}, we see that the Cartan decomposition extends to give one for $\cR'$. In  particular, this gives a direct sum decomposition of  $\cR'$ in terms  of line bundles.

Now define $\cR$ to be the reflexive closure of $\cR'$; see Lemma~\ref{langton1}. Since we are on an irreducible smooth scheme $X$, by Lemma~\ref{langton2}, reflexive closures of invertible sheaves remain invertible sheaves. Thus $\cR$ is also a locally free sheaf on $X$. The Lie bracket on $\cR'$ now extends by a Hartogs-type argument, and we get a Lie algebra bundle structure on $\cR$ with the desired properties.
\end{proof}

\subsection{Construction of the {\tt n}-parahoric group scheme}\label{atthehyper}
The notation in this section is as in the previous one; namely, $\bf A \simeq \mathbb A^{n}$ and ${\bf A_{0}} \subset {\bf A}$ is the complement of the coordinate hyperplanes. Recall that $\cO_{n} = k\llbracket z_{1}, \ldots, z_{n}\rrbracket$ and $K_{n} = \Fract(\cO_{n})$.

Let $\{H_{i}\}$ be the hyperplanes  of ${\bf A}$.    For each $i$,  let $\zeta_{i}$ denote the generic point of the divisor $H_{i}$. Let  
\begin{equation} \label{Aalpha0}
A_{i}= \mathcal O_{{\bf A}, \zeta_{i}}
\end{equation} 
be the DVR obtained by localizing at the height $1$ primes given by  $\zeta_{i}$,  and let $X_{i} := \spec(A_{i})$. Base changing by the local morphism $X_{i} \to \bf A$, we have a Lie algebra bundle $\mathcal R|_{X_{i}}$ for each $i$. Moreover,  the Lie algebra bundle $\mathcal R$ on ${\bf A}$, see Theorem~\ref{gpschLiestab},  gives  canonical gluing data to glue $\mathcal R|_{X_{i}}$ with the trivial bundle $\mathfrak g \times \bf A_{0}$.

The following first case  is the key to proving the main result, Theorem~\ref{multipargrpsch}. 

 \begin{thm}\label{n-paragrpsch}
The main theorem, Theorem~\ref{multipargrpsch}, holds when the $n$-tuple of concave functions is of {\tt type I} \textup{(}see Definition~\ref{types}\,\textup{)}.
 \end{thm}
 
\begin{proof} We begin by making constructions on ${\bf A}$. The existence of the Lie algebra bundle $\cR$ on ${\bf A}$ follows by applying Proposition~\ref{onliestr} to the {\tt n-concave} function of {\tt type I} defined by the point $\bvt=(\theta_{1},\ldots,\theta_{n})$. Since we are in the {\tt type I} case, we could have obtained the Lie algebra bundle directly from Theorem~\ref{gpschLiestab} as well. Again since we are in {\tt type I}, the existence of the smooth {\em affine} group scheme $\gG_{\bvt}$ with connected fibres is a consequence of Theorem~\ref{Artin-Weil-Kawamata}; its cell structure and properties \eqref{mt4} and \eqref{mt5} of Theorem~\ref{multipargrpsch} follow from  Proposition~\ref{bigcell1}. 

We now restrict the group scheme $\gG_{\bvt}$  to $\spec(\cO_{n}) \subset {\bf A}$ and continue to denote the restricted group scheme  by $\gG_{\bvt}$.  Since the base is smooth and the group scheme is also smooth, by a theorem of Weil, see \cite[Section~4.4, Theorem 1]{blr}, it follows that the global sections are determined on any open subset $U$ whose complement is of codimension at least~$2$. Let us call such an open subset {\em a big open subset}.  Sections over $U$ are obtained by intersecting the sections of the parahoric groups defined at the height $1$ primes in $G(K_{n})$. It follows immediately that the sections over the big open subset $U$ coincide with the {\tt n}-parahoric group $\cP_{\bvt}$, proving assertion~\eqref{mt2} of Theorem~\ref{multipargrpsch}.

We now discuss {\em the uniqueness of the group scheme}; we note that although we do this here in the setting of type I, the arguments are identical and will imply the same for the most general type III case, once the construction along with properties \eqref{mt4} and \eqref{mt5} of Theorem~\ref{multipargrpsch}  is carried out in Section~\ref{schematization3}.

If $\mathfrak H$ is a connected, affine smooth group scheme over $\cO_{n}$ which is generically split and such that it has closed toral and root subgroup schemes as in \eqref{mt4}  and a {\em big cell} as in \eqref{mt5} of Theorem~\ref{multipargrpsch}, then by \cite[Sections~1.2.13 and 1.2.14]{bruhattits}, we get the desired isomorphism of $\gG_{\bvt}$ and $\mathfrak H$, showing the uniqueness.

A subtler feature regarding the uniqueness can be extracted from these arguments. Once the smooth connected group scheme $\gG_{\bvt}$ is constructed abstractly, the big cell structure over the big open subset $U$ (which contains all points of height $1$) can be obtained  at each of the height $1$ primes as in \cite{bruhattits}. So the cell structure coincides with the {\em big cell} of the {\tt parahoric} group schemes $\gG_{\theta_{j}}$ for each $j$, and hence by the uniqueness of the group schemes over discrete valuation rings, see  \cite[Section~3.8.3]{bruhattits}, we get the identification of the group schemes $\gG_{\bvt}$ and $\mathfrak H$ over $U$. Now, we can appeal to the deeper results in \cite[Corollaire IX.1.5]{raynaud} to extend the  isomorphism of group schemes from $U$ to $\spec(\cO_{n})$.
\end{proof}

\brem \label{diagonalremark} \label{schematicnature} In continuation of Proposition~\ref{onliestr}, the description of the diagonal restriction of $\gG_{\bvt}$  can be obtained directly from the facts that it is affine by Theorem~\ref{n-paragrpsch} and the group $\gG_{\bvt}(\cO_{n})$ of $\cO_{n}$-valued points is the {\tt n-parahoric} group ${\cP}_{\bvt}$. Let the notation be as in Theorem~\ref{thedescription}. Let $\Delta:=\mathbb A_{\tt I}$ for $I=\{1,\ldots,n\}$ be the diagonal curve in the affine space and $\cO$ be the ring of functions on $\Delta$ at the origin. Hence the restriction $\gG_{\bvt}\mid_{\Delta}$  is such that the space of its sections $\gG_{\bvt}\mid_{\Delta}(\cO)$ is identified with $\gG_{\bvt}(\cO_{n})$, in which we set all variables as equal. Thus we get the subgroup ${\cP}_{\bvt}^{\diag} \subset G(k(\!(t)\!))$; see \eqref{allunifequal}.  By the \'etoff\'e property of {\tt BT}-group schemes, this forces that $\gG_{\bvt} \mid_{\Delta}$ is the {\tt BT}-group scheme given by the concave function $\cm_{\bvt}$, where $\cm_{\bvt}(r) := \sum_i m_r(\theta_i)$. We remark that the above discussion still holds if we work with subdiagonal curves $\mathbb A_{\tt I}$ as Theorem~\ref{thedescription}\eqref{grat1-b}, which are defined by any subcollection of the hyperplanes indexed by a non-empty ${\tt I} \subset \{1, \ldots, n\}$.
\erem

\begin{Cor} \label{moreaffineness}
Consider an $n$-tuple ${\bf f}$ of concave functions, each of which is a sum of  such functions of {\tt type I}. Then the group scheme $\mathfrak{G}_{\bf f}$  is {\tt affine}.
\end{Cor}

\begin{proof} Let ${\bf f}=(f_{1},\ldots,f_{n})$, where $f_{i}=\sum_{j} f_{\theta_{ij}}$. Set $\bvt=( \theta_{ij} )$. Then $\mathfrak{G}_{\bvt}$ is affine by Theorem~\ref{n-paragrpsch}. Note that this group scheme constructed out of $N$-parahoric group schemes is over an affine space of dimension $N$, for $N \gg 0$, and obtained from the data of the $\{\theta_{ij}\}$. 

Observe that  the group scheme $\mathfrak{G}_{\bf f}$ can be seen as a restriction to a subdiagonal in the larger affine space of $\mathfrak{G}_{\bvt}$. Now the assertion of affineness follows by  Remark~\ref{diagonalremark}. 
\end{proof}

\subsection{Revisiting two \guillemotleft contre-exemples \guillemotright~in Bruhat--Tits \cite{bruhattits}} \label{revisit}
In this subsection we take a closer look at a couple of counter-examples in Bruhat--Tits \cite[Section~3.2.15]{bruhattits}. The examples highlight how the procedure developed in Section~3 of \cite{bruhattits} of taking schematic closures in certain linear groups could give schemes which are neither {\em flat} nor {\em group schemes}. This phenomenon happens when the base scheme is $\mathbb A^{2}$. By discarding certain bad components in the fibre over $o \in \mathbb A^{2}$,  they extract a {\tt 2BT}-group scheme whose {\em affineness} is left in doubt. 

The base ring is $\cO = k[\xi, \eta]$, the group is $G=\SL_{2}$, the concave function $f$ is defined by $f(a) = \xi\cdot\cO$ and $f(-a) = \eta\cdot\cO$, and the representation is $\rho = \id \times \Ad$ on the module $K^{2} \times \mathfrak g$. The $\cO$-module $M$  is a free module generated by the basis of $K^{2}$ and the elements $\eta\cdot e_{-a}, h, \xi\cdot e_{a}$. We view them as $\begin{pmatrix} 0 & 0 \\ -\eta & 0 \end{pmatrix}, \begin{pmatrix} 1/2 & 0 \\ 0 & -1/2 \end{pmatrix}$ and $\begin{pmatrix} 0 & \xi \\ 0 & 0 \end{pmatrix}$. Thus we have  
\begin{equation}
\rho \begin{pmatrix}
x & y \\
z & t
\end{pmatrix} = \begin{pmatrix} \begin{pmatrix} 
x & y \\
z & t
\end{pmatrix},
\begin{pmatrix}
t^{2} & -\eta\cdot zt & \xi\eta^{-1}\cdot z^{2} \\
-2\eta\cdot yt & xt + yz & -2\xi\cdot xz \\
\eta\cdot \xi^{-1} y^{2}& -\xi\cdot xy& x^{2}
\end{pmatrix}
\end{pmatrix}
\end{equation}
The schematic closure of $\rho(G)$ in $\mathfrak{GL}(M)$ is $\mathfrak G$. The fibre over $o \in \mathbb A^{2}$ is
\beqa
\mathfrak G_{o} = \mathfrak G^{1}_{o} \cup \mathfrak G^{2}_{o}, 
\eeqa
where $\mathfrak G^{1}_{o}$ has dimension $4$ and $\mathfrak G^{2}_{o}$ dimension $3$. Hence $\mathfrak G$ is not flat over $\cO$, while $\mathfrak H := \mathfrak G \setminus \mathfrak G^{1}_{o}$ is a {\em flat group scheme}, albeit in the words of \cite[Section~3.2.15, p.~60]{bruhattits} \guillemotleft {\em mais probablement pas affine} \guillemotright, \textit{i.e.} {\em but probably not affine}.

The fibre of $\mathfrak H$ at $o$ is obtained from \cite[Section~3.2.14, Equation~(2)]{bruhattits}, which is obtained by a diagonal restriction.

We now reinterpret this example as the  simplest case of Theorem~\ref{n-paragrpsch}, for $G=\SL_{2}$. The surprising picture which emerges  from our approach is that we even get the {\em affineness} of their group scheme $\mathfrak H$ as a consequence.

In our setting of Section~\ref{explicitexamples}, we have   $\Phi^{+} = \{a\}$, and in terms of the alcove vertices \eqref{alcovevertices}, we have $\theta_{1}:= \theta_{a}/2  = \omega^{\vee}/2$ and $ 
\theta_{2}:= -\theta_{a}/2 = -\omega^{\vee}/2$. With $m_r(\theta) := -\lfloor r(\theta) \rfloor$, we get  
\begin{equation}
  a(\theta_{1}) =
\begin{cases}
1/2&\quad\text{if } a \in \Phi^{+}, \\
-1/2&\quad{\text if }a \in \Phi^{-},  
\end{cases} \end{equation}

\begin{equation}
  m_{a}(\theta_{1}) =
\begin{cases}
0&\quad\text{if } a \in \Phi^{+}, \\
1&\quad{\text if }a \in \Phi^{-}, 
\end{cases} \end{equation}
and similarly for $\theta_{2}$. We get 
\begin{equation}
  a(\theta_{2}) =
\begin{cases}
-1/2&\quad\text{if } a \in \Phi^{+}, \\
1/2&\quad{\text if }a \in \Phi^{-}, 
\end{cases} \end{equation}
\begin{equation}
  m_{a}(\theta_{2}) =
\begin{cases}
1&\quad\text{if } a \in \Phi^{+}, \\
0&\quad{\text if }a \in \Phi^{-}.
\end{cases} \end{equation}
By Theorem~\ref{n-paragrpsch}, if $\bT = (\theta_{1}, \theta_{2})$, the group scheme $\gG_{\bT}$ is precisely $\mathfrak H$ and also a {\em smooth, affine} group scheme over $\cO$. Furthermore, the diagonal restriction is the affine group scheme $\mathfrak G_{\cm_{\bT}}$ given by the concave function
$\cm_{\bT}\colon r \mapsto \sum m_r(\theta_{j})$ with 
\begin{equation}
\cm_{\bT}(a) = 1 \quad \forall  a \in \Phi  , \end{equation}
which computes ${\cP}_{\bT}^{\diag}$ and recovers  the concave function considered in \cite[Section~3.2.14]{bruhattits}.

\section{Schematization of {\tt n}-bounded groups \texorpdfstring{${\cP}_{\bwt}$}{P\textunderscore Omega}}\label{schematization2}
The notation in this section is as in the previous one; namely, $\bf A \simeq \mathbb A^{n}$ and ${\bf A_{0}} \subset {\bf A}$ is the complement of the coordinate hyperplanes. \emph{In this entire section, we work with $G$ being one of $A_{n}$, $G_{2}$, $F_{4}$ or $E_{6}$.}

We begin by remarking that in the foundational paper \cite{bruhattits1}, Bruhat and Tits study bounded groups associated to concave functions on the $\Phi \cup\{0\}$. In the sequel \cite{bruhattits}, the main result is to prove that these groups are {\sl schematic}.  

A class of concave functions, which play a basic role even in the general study, are the ones which come associated to bounded subsets $\Omega$ of the affine apartment. In this section, we extend the above theory of {\tt n}-parahoric groups (with their schematic properties) to the bounded groups associated to products of bounded regions $\boldsymbol\Omega := (\Omega_{1}, \ldots, \Omega_{n})$ in the apartment. 

In the following {\sl key} remark, we recall a few basic facts from \cite{bruhattits3}.

\brem\label{frombtclassiques}  Let $G=\SL(n)$ for some $n$. Let $\Omega$  be a non-empty {\em enclosed} bounded subset of the apartment $\mathcal{A}$. It is the convex hull of a finite subset $\{\theta_{\tt 1}, \ldots, \theta_{\tt s}\}$ of points in $\mathcal{A}$. {\sl In terms of the points $\theta_{j}$, the {\tt BT}-group scheme $\gG_{\Omega}$ on $\spec(\cO)$ is precisely the {\tt schematic closure} of the image of $G$ in $\prod_{j = 1}^{\tt s} \gG_{\theta_{j}}$ under the diagonal map}. This is a consequence of the discussion in \cite[Remarques~3.9]{bruhattits3}, together with \cite[Section~5.3]{bruhattits3}, where it is shown that the group scheme $\gG_{\Omega}$ over $\spec(\cO)$ can be realized as the {\em closed image} in $\prod_{j = 1}^{\tt s} \gG_{\theta_{j}}$ of the diagonal morphism  of $G$  in the generic fibre $G^{\tt s}$ of $\prod_{j = 1}^{\tt s} \gG_{\theta_{j}}$. 

We further remark that the stated result for $\SL(n)$ also works by the results \cite[Theorem 10.1]{ganyu1} for $G_{2}$ and \cite[Theorem 9.1]{ganyu2} for $F_{4}$ and $E_{6}$. 
 \erem
 
\brem\label{goodhull} An important aspect of the remark  made above which will be used in what follows is to view the group scheme $\prod_{j = 1}^{\tt s} \gG_{\theta_{j}}$ as an {\em invariant direct image} from a suitable ramified cover (see Section~\ref{grstuff}). Now the group scheme $\gG_{\theta_{j}}$ arises from an invariant direct image if $\theta_{j}$ belongs to the alcove $\mathbf{a}_0$ (see Section~\ref{liedata}) by \cite{base}. This generalizes to the case when $\theta_{j}$ belongs to the fundamental domain of $Y(T)$ by Remark~\ref{uptoyt} and Section~\ref{bstomixed}. For the general case, we may have to move the points $\{\theta_{\tt 1}, \ldots, \theta_{\tt s}\}$  in apartment $\mathcal{A}$ into the fundamental domain for the action of $Y(T)$. Let $h_{j} \in T(K)$ be elements such that the point $\{h_{1}\cdot\theta_{\tt 1}, \ldots, h_{\tt s}\cdot\theta_{\tt s} \}$ now lies in the $\tt s$-fold product of the fundamental domain of $Y(T)$. We now consider the modified diagonal embedding of $G$ in the generic fibre of $\prod_{j = 1}^{\tt s} \gG_{h_{j}\cdot\theta_{j}}$ which is given by a simultaneous conjugation by the $h_{j}$, \textit{i.e.} $g \mapsto (h_{1}\cdot g\cdot h_{1}^{-1}, \ldots, h_{\tt s}\cdot g\cdot h_{\tt s}^{-1})$. Then the {\em schematic closure} via this embedding in $\prod_{j = 1}^{\tt s} \gG_{h_{j}\cdot \theta_{j}}$ gives the group scheme $\gG_{\Omega}$ up to an isomorphism. Moreover, the new ambient group scheme $\prod_{j = 1}^{\tt s} \gG_{h_{j}\cdot \theta_{j}}$ comes as an invariant direct image from a single cover. \erem
 
\subsection{The schematic hull \texorpdfstring{$\mathbb S_{\bwt,\Lambda}$}{S\textunderscore \{Omega, Lambda\} }  of an {\tt n}-bounded set \texorpdfstring{$\bwt$}{Omega} and its big cell structure} \label{schematichullslambda}
{\sl For simplicity of  notation, let  $\Lambda$ stand for the {\em polynomial ring} $k[z_{1}, \ldots, z_{n}]$ and ${\tt F}$ stand for  $\Fract(\Lambda)$ throughout this proof}.

Let $\bwt = (\Omega_{1}, \ldots, \Omega_{n})$, where we may assume that the $\Omega_{i}$ are {\em enclosed} bounded subsets of the affine apartment $\mathcal A$. Let  $\Omega_{i}$ be the convex hull of points $\{\theta_{\tt i1}, \ldots, \theta_{\tt is}\}$ of $\cA_T$. Note that {\em a priori} the number {\tt s} of points also varies with $i$.
By Remarks~\ref{frombtclassiques} and~\ref{goodhull}, we may suppose
\begin{enumerate}
\item that these points lie in a single fundamental domain of $Y(T)$ in $\cA_T$ and
\item that over a DVR $\spec(\cO)$, we realize each $\gG_{\Omega_{i}}$ as the schematic closure of $G$ in $\prod_{j = 1}^{\tt s} \gG_{\theta_{\tt ij}}$.
\end{enumerate} By choosing $\tt s$ to be the maximal number of terms for varying $i = 1, \ldots, n$, and by repeating the $\theta_{\tt ij}$  when the number is less than $\tt s$, we may assume that all the group schemes $\gG_{\Omega_{i}}$ are schematic closures of  products of $\tt s$ parahoric group schemes on $\spec(\cO)$. 

Fix a $j \in \{\tt 1, \ldots, \tt s\}$. For each hyperplane $H_{i}$ ($i = 1, \ldots, n$),  we attach the point $\theta_{\tt ij}$. For this configuration of points in the apartment, we have by Theorem~\ref{n-paragrpsch} an {\tt n-parahoric} group scheme $\gG_{\theta_{\tt 1j}, \ldots, \theta_{\tt nj}}$ on $\bf A$. We set 
\beqa\label{mathfraks}
\mathbb S_{\bwt,\Lambda} := \prod_{1 \leq j \leq s} \gG_{\theta_{\tt 1j}, \ldots, \theta_{\tt nj}}
\eeqa
to be the {\tt s}-fold product of {\tt n-parahoric} group schemes on $\bf A$. Thus $\mathbb S_{\bwt,\Lambda}$ is a {\em smooth, affine group scheme} on $\spec(\Lambda) = \bf A$ with generic fibre being the split group scheme $G^{\tt s} \times \spec({\tt F})$. Since each {\tt n-parahoric} group scheme $\gG_{\theta_{\tt 1j}, \ldots, \theta_{\tt nj}}$ is realized as a Weil restriction of scalars from a finite flat cover of $\bf A$ (see Theorem~\ref{n-paragrpsch}), so is the group scheme $\mathbb S_{\bwt,\Lambda}$.

\begin{defi}\label{shell} We call the product group scheme $\mathbb S_{\bwt,\Lambda}$, see \eqref{mathfraks}, the {\tt schematic hull} of the {\tt n-bounded} set $\bwt$ and will denote it simply as $\mathbb S_{\Lambda}$. 
\end{defi}
This group scheme $\mathbb S_{\Lambda}$ will play a central role in all that follows. 

Recall that each of the {\tt n-parahoric} group schemes $\gG_{\theta_{\tt 1j}, \ldots, \theta_{\tt nj}}$ occurring in \eqref{mathfraks} is equipped with a  {\em big cell}  by Theorem~\ref{n-paragrpsch} (see Section~\ref{BigCell}).  Taking their product, we define the {\em big cell} of the group scheme $\mathbb S_{\Lambda}$. Let us write it as   the image of the product morphism 
\beqa\label{bigcellinambient}
\mathfrak B^{-} \times \cT \times \mathfrak B^{+} \lra \mathbb S_{\Lambda}.
\eeqa 
Let $\mathfrak{U}_{r,j}$ denote the root group of $\gG_{\theta_{\tt 1j}, \ldots, \theta_{\tt nj}}$ corresponding to $r \in \Phi$. For $r \in \Phi$, let us set 
\begin{equation} \label{br}
\mathfrak{B}_{r} := \prod_{1 \leq j \leq s} \mathfrak{U}_{r,j}.
\end{equation}
Then under the induced group law from $\mathbb S_{\Lambda}$,  each $\mathfrak{B}_{r}$   is Abelian.
 
\subsection{The diagonal embedding}
Let $U$ be an open subset of $\spec(\Lambda)$ which contains all primes of height at most~$1$. At the completion of the local rings at the height $1$ prime ideal associated to the generic point of each $H_{i}$  indexed by $i = 1, \ldots, n$, we assign the Bruhat--Tits group scheme $\gG_{\Omega_{i}}$. By definition, these group schemes are also generically split. 

We let $\gG_{\bwt, {\tt F}}$ denote the {\em image of  the diagonal embedding} of $G \times \spec({\tt F})$ in $G^{\tt s} \times \spec({\tt F})$. As in Proposition~\ref{onliestr} and in the proof of Theorem~\ref{Artin-Weil-Kawamata}, by gluing (using the map $\id_{G}$) with the $\gG_{\Omega_{j}}$, the group scheme $\gG_{\bwt, {\tt F}}$ extends to each of the generic points of the coordinate hyperplanes.  Let $\gG_{\bwt, U}$ denote this smooth, affine group scheme on $U$. We begin with the closed embedding $\gG_{\bwt,{\tt F}} \subset \mathbb S_{\tt F}$, where $\mathbb S_{\tt F} = \mathbb S_{\Lambda}|_{\spec(\tt F)}$ . Let
\beqa\label{mathfraky}
{\mathfrak Y}_{\Lambda} := \overline{\gG_{\bwt, {\tt F}}}
\eeqa 
denote the {\em schematic closure} of $\gG_{\bwt, {\tt F}}$ in  $\mathbb S_{\Lambda}$.

\begin{lem}\label{moreonschclos}
  The group scheme $\gG_{\bwt, U}$ is in fact the schematic closure $\mathfrak Y_{U}$ of\, $\gG_{\bwt, {\tt F}}$ in  ${\mathbb S}_{U}$.
\end{lem}

\begin{proof} Since $U$ contains all primes of height at most~$1$, we see that   the schematic closure $\mathfrak Y_{U}$  of $\gG_{\bwt, {\tt F}}$ in  ${\mathbb S}_{U}$ is in fact a {\em flat} group scheme on $U$, and since $\mathbb S_{U}$ is affine, it is in fact {\em affine} as well.

Since $\gG_{\bwt, {\tt F}}$ is smooth and hence reduced, ${\mathfrak Y}_{\Lambda}$ is in fact the closure of $\gG_{\bwt, {\tt F}}$ in $\mathbb S_{\Lambda}$ with its reduced scheme structure. Since the base is of dimension at least~$2$, this schematic closure, however, need not be a group scheme (see \cite[Section~3.2.15]{bruhattits}). 

To prove the contention of the lemma, we need to verify that the restriction of $\mathfrak Y_{U}$ to the completion of the local rings of each height $1$ prime associated to $H_{i}$ is in fact isomorphic to the {\tt BT}-group scheme $\gG_{\Omega_{i}}$. This follows from  Remarks~\ref{frombtclassiques} and~\ref{goodhull}. 
\end{proof}

\subsection{{\tt n}-bounded-groups and schematization}
The schematization of the {\tt n}-bounded analogue is now carried out in two steps.  As in the construction of the {\tt n}-parahoric, we  construct the Lie algebra bundle $\mathcal{R}_{\bwt}$ first  (see Theorem~\ref{gpschLiestab}). This will be eventually turn out to be $\Lie(\gG_{\bwt})$.

\begin{prop}\label{gpschLiestabomega} Let  the notation be as in Theorem~\ref{gpschLiestab} and Proposition~\ref{onliestr}. The trivial bundle with fibre $ \mathfrak g$ on $ {\bf A_{0}}$ extends to ${\bf A}$, preserving the Cartan decomposition as a canonical Lie algebra bundle $\mathcal{R}_{\bwt}$. Further, for each $\zeta_{i}$, the restriction $\mathcal{R}_{\bwt}|_{X_{i}}$ is isomorphic to  $\mathcal{R}_{\Omega_{i}}$. \end{prop}

\begin{proof} This is special case of Proposition~\ref{onliestr}. \end{proof}
We will denote the Cartan decomposition of $\mathcal{R}$ as follows:
\begin{equation} \label{Rdecomp} \mathcal{R}_{\bwt} = \Lie(T) \otimes \mathcal{O}_{\bf A} \oplus_{r \in \Phi} \mathcal{R}_{r}.
\end{equation}

The group scheme construction on the higher-dimensional base has two obstructions. Firstly, schematic closures  need not be flat when the base is not a  Dedekind domain (see Section~\ref{revisit}), and secondly, there is no immediate analogue of Theorem~\ref{Artin-Weil-Kawamata}. This is because the bounded subsets $\Omega_{i}$ need not be singletons, and so  the local BT-group schemes at the height $1$ primes do not necessarily come from ramified covers. 

Towards this end, we begin with a useful lemma.

\begin{lem}[\textit{cf}.~\protect{\cite[Section 1.2, pp.~17--18]{bruhattits}}]\label{[1.2.7, p.~18]bruhattits}
  Let $K$ be an infinite field and $A \subset K$ be a subring such that $K$ is the quotient field of $A$. Let $f\colon X \to Y$ be a morphism of schemes over $A$, and let $i\colon Z \hookrightarrow Y$ be a closed subscheme. Suppose that $X$ is $A$-flat and that $f_{K}\colon X_{K} \to Y_{K}$ factorises through $Z_{K}$. Then $f$ factorises through $Z$.
\end{lem}

\begin{proof} By going to an affine cover of $Y$, we may assume that $X$ and $Y$ are $A$-affine with coordinate rings $\cO[X]$ and $\cO[Y]$. Since $X$ is $A$-flat, the map $j_{X}\colon\cO[X] \to  K \otimes _{\cO} \cO[X]$ is {\sl injective}. Let $a \in I(Z)$ be an element in the ideal defining $Z$ in $Y$. Then $j_{X} \circ f^{*}(a) = f_{K}^{*} \circ j_{Y}(a) = 0$ since $f_{K}$ has a factorization. By the injectivity of $j_{X}$, we get $f^{*}(a) = 0$. \end{proof}  

\subsection {Continuation of proof of Theorem~\ref{multipargrpsch}}

\begin{thm}\label{n-regiogrpsch}
  The main theorem, Theorem~\ref{multipargrpsch}, holds when the $n$-tuple of concave functions is of\, {\tt type II} \textup{(}see Definition~\ref{types}\,\textup{)} and when  $G$ is of type $A_{n}$, $G_{2}$, $F_{4}$ or $E_{6}$ \textup{(}see Remark~\ref{frombtclassiques}\,\textup{)}.
\end{thm}

\begin{proof}
We will break up the rather long and technical proof into two parts. \begin{enumerate*}[label=Part~(\Roman*), ref=\Roman*]\item\label{part1}  constructs the candidate for {\em big cell} in the type II case.  \item\label{part2} then constructs  the group scheme by using this cell structure. \end{enumerate*} In this sense, the key issue is the extension of the {\em schematic root datum}.

\medskip
\noindent {\em Part~\eqref{part1}  of the proof}.  We begin with the closed subgroup scheme $\gG_{\bwt, {\tt F}}$ of $\mathbb S_{\tt F}$ (obtained by the diagonal embedding of $G$ in $G^{s}$). 

The closed subgroup scheme $\gG_{\bwt, {\tt F}} \hookrightarrow \mathbb S_{\tt F}$  over $\spec({\tt F})$ is such that each factor of the product $ \mathfrak U_{\bwt, {\tt F}}^{-} \times T_{\bwt, {\tt F}} \times \mathfrak U_{\bwt, {\tt F}}^{+}   $ embeds factor by factor in the corresponding term in the cell of $\mathbb S_{{\tt F}}$. The full cell $ \mathfrak C_{\bwt, {\tt F}} := \Im\big( \mathfrak U_{\bwt, {\tt F}}^{-} \times T_{\bwt, {\tt F}} \times \mathfrak U_{\bwt, {\tt F}}^{+} \to \gG_{\bwt, {\tt F}} \big)$ therefore sits as a flat closed subscheme of the open subscheme 
\beqa\label{bigbigcell}
\Im\left(\mathfrak B^{-}_{\tt F} \times \cT_{\tt F} \times \mathfrak B^{+}_{\tt F} \lra \mathbb S_{\tt F}\right).
\eeqa

As before, let $U$ be an open subset of $\spec(\Lambda)$ which contains all primes of height at most~$1$. By Proposition~\ref{gpschLiestabomega}, we have a locally free Lie algebra bundle $\mathcal{R}_{\bwt}$ on $\spec(\Lambda) = {\bf A}$.  This gives an extension of $\Lie(\gG_{\bwt, {\tt F}})$ as a Lie subalgebra bundle $\mathcal{R}_{\bwt}$ of $\Lie(\mathbb S_{\Lambda})$. Under the Cartan decomposition, we get the following refinement of subbundles:
\begin{equation} \label{RinLieSlambda}
\mathcal{R}_{r,\bwt} \longhookrightarrow \Lie\left(\mathfrak B_{r}\right).
\end{equation}

In  the proof, as a preliminary step we will construct a {\em big cell} $\mathfrak C_{\bwt,\Lambda}$ as a subscheme of the group scheme $\mathbb S_{\Lambda}$. This is achieved in \eqref{mathfrakc} below. This scheme $\mathfrak C_{\bwt,\Lambda}$ will be such that when restricted to the open subset $U$ of $\spec(\Lambda)$, we have an open immersion 
\beqa
\mathfrak C_{U} \longhookrightarrow \gG_{\bwt, U}.
\eeqa
 
We recall that by \eqref{mathfraky} and the proof of Lemma~\ref{moreonschclos}, the scheme ${\mathfrak Y}_{\Lambda}$ 
is in fact the {\em schematic closure} of $\gG_{\bwt, U}$ in  $\mathbb S_{\Lambda}$.  
The main aim of this first part of the proof is not only to  show the existence of a  {\em big cell}  $\mathfrak C_{\bwt,\Lambda}$ in \eqref{mathfrakc}, but to show in fact that Proposition~\ref{2.2.10} may be applied to it because it satisfies the following key property. 

\begin{claim}\label{claim1}
There is an open immersion from the scheme $\mathfrak C_{\bwt,\Lambda}$ into ${\mathfrak Y}_{\Lambda}$.
\end{claim}

We begin with the {\em big cell} $\mathfrak C_{\bwt,{\tt F}} :=  \Im\big( \mathfrak U_{\bwt, {\tt F}}^{+} \times T_{\bwt, {\tt F}} \times \mathfrak U_{\bwt, {\tt F}}^{-}  \to \gG_{\bwt, {\tt F}} \big)$. This gives the decomposition 
$$\Lie\left(\gG_{\bwt, {\tt F}}\right) = \Lie\left(\mathfrak U_{\bwt, {\tt F}}^{+}\right) \oplus \Lie\left(T_{\bwt, {\tt F}}\right) \oplus \Lie\left(\mathfrak U_{\bwt, {\tt F}}^{-}\right) .$$ 
The goal now is to extend $\mathfrak C_{\bwt,{\tt F}}$ to a {\em big cell} $\mathfrak C_{\bwt,\Lambda}$, as an {open subscheme} of ${\mathfrak Y}_{\Lambda}$. We prove this in several steps.

\begin{enumerate}[wide,label={\emph{Step} \arabic*:}, ref=\arabic*,itemsep=.5em]
\item\label{step1}
    \emph{Toral extension to $\Lambda$ and closed embedding in $\mathfrak Y_{\Lambda}$.} Since we work with the {\em generically split} case, the maximal torus $T_{\bwt, {\tt F}} \subset \gG_{\bwt, {\tt F}}$  is {\sl split}. We now work with a $U$ which contains all primes of height at most~$1$. The torus $T_{\bwt, {\tt F}}$ canonically extends to a smooth closed multiplicative subgroup scheme $\cT_{\bwt, U}$ of $\gG_{\bwt, U}$  over $U$. This can be explicitly constructed as in \cite[Section 4.4.18 and 4.6.2]{bruhattits} (or by using the identity component of the N\'eron-lft model from \cite[Section~10.1, Proposition 6]{blr}, as in  \cite[Proposition 3.2]{land}). 

By \cite[Proposition IX.2.4]{raynaud},  it follows that the inclusion $\cT_{\bwt, U} \hookrightarrow \gG_{\bwt, U} \hookrightarrow  \mathbb S_{U}$ extends uniquely to a smooth group scheme $\cT_{\bwt, \Lambda}$ together with a closed immersion $\cT_{\bwt, \Lambda} \hookrightarrow \mathbb S_{\Lambda}$  as a subtorus of the group scheme $\mathbb S_{\Lambda}$. Since $\cT_{\bwt, \Lambda}$ is {\em flat and hence torsion-free}, it follows that $\cT_{\bwt, \Lambda}$ is the schematic closure of  $\cT_{\bwt, U}$ in~$\mathbb S$.  Again, since $\mathfrak Y_{\Lambda}$ is the schematic closure of $\gG_{\bwt, U}$ in $\mathbb S$, we conclude that the embedding $\cT_{\bwt, \Lambda} \hookrightarrow \mathbb S_{\Lambda}$ factors via a closed embedding $\cT_{\bwt, \Lambda} \hookrightarrow \mathfrak Y_{\Lambda}$.

\item \label{step2}
\emph{Unipotent and big cell extension to $U$.} Next, we need to extend the unipotent group schemes $\mathfrak U_{\bwt, {\tt F}}^{\pm}$, firstly as closed subschemes of $\gG_{\bwt, U}$. This is easy since these are precisely the schematic closures of the root groups $\mathfrak U_{r, {\tt F}}^{\pm}$, for $r \in \Phi$, and hence of $\mathfrak U_{\bwt, {\tt F}}^{\pm}$ in $\gG_{\bwt, U}$; we  obviously rely on Bruhat--Tits theory over discrete valuation rings. Call these extended group schemes $\mathfrak U_{\bwt, U}^{\pm}$ and observe that since we work over height $1$ primes, these are {\em flat} over $U$.

Thus we have the {\em big cell} 
\beqa
\mathfrak C_{U}:= \Im\left(\mathfrak U_{\bwt, U}^{-} \times \cT_{\bwt, U} \times \mathfrak U_{\bwt, {\tt U}}^{+}   \longhookrightarrow \gG_{\bwt, U}\right)
\eeqa
in $\gG_{\bwt, U}$ which has the property that at each of the height $1$ primes, it restricts to the big cell in the usual Bruhat--Tits group schemes $\gG_{\Omega_{j}}$.

\item\label{step3}
 \emph{Unipotent extension to $\spec(\Lambda)$ as closed subschemes of\, $\mathfrak Y_{\Lambda}$.} We next ensure that the group schemes $\mathfrak U_{\bwt, U}^{\pm}$ extend as subschemes of the schematic closure ${\mathfrak Y}_{\Lambda}$ to ensure that a {\em big cell} $\mathfrak C_{\bwt,\Lambda}$ can be openly embedded in ${\mathfrak Y}_{\Lambda}$. 

We do this in three substeps. 

\begin{enumerate}[wide,label={\emph{Step} 3.\arabic*.}, ref=3.\arabic*,itemsep=.5em]
\item\label{step31}
We first extend the unipotent group schemes abstractly from $U$ to the whole of $\spec(\Lambda)$. For each $r \in \Phi$, consider the {\it fractional invertible ideal} of $\Lambda$, \textit{i.e.} a finitely generated projective sub $\Lambda$-module of $\tt F$ given by 
\begin{equation}
f(r):=z_{1}^{m_r(\Omega_{1})} \cdots z_{n}^{m_r(\Omega_{n})} \Lambda \subset \tt F.
\end{equation}
By \cite[Section~3.2.6]{bruhattits}, this ideal defines smooth $\Lambda$-group schemes $\mathfrak U_{r, \bwt, \Lambda}$ extending
the root groups $U_{r, {\tt F}}$ as well as the root groups $\mathfrak U_{r, \bwt, U}$ over $U$. 

\item\label{step32}
  {\sf We now prove  that the group schemes $\mathfrak U_{r, \bwt, \Lambda}$  are closed subgroup schemes of the Abelian subgroup schemes $\mathfrak{B}_r$ (see \eqref{br}) of $\mathbb S_{\Lambda}$  and therefore closed subschemes of ${\mathfrak Y}_{\Lambda}$. We do so by showing in fact that they are {\em schematic closures} of the root groups $\mathfrak U_{r, \bwt, U}$ in $\mathfrak{B}_r$. Our method is to reduce verifications by Section~\ref{onWfromsga} to an inclusion of Lie algebra bundles on $\Lambda$. This check further reduces to $U$.} 

\begin{enumerate}[wide,label={\emph{Step} 3.2.\arabic*.}, ref=3.2.\arabic*,itemsep=.5em]
\item
  By Section~\ref{onWfromsga}, the following group schemes arise from the $W$-construction of their Lie algebra bundles:
\begin{equation} \mathfrak U_{r, \bwt, U} \simeq W\left(\Lie\left(\mathfrak U_{r, \bwt, U}\right)\right), \quad \quad \mathfrak B_r \simeq W\left(\Lie\left(\mathfrak B_r\right)\right).
\end{equation}

Note that the root groups $\mathfrak U_{r, \bwt, U}$ are closed subgroup schemes (over $U$) of the restrictions $\mathfrak B_{r,U}$ to $U$. Moreover, the $\mathfrak B_{r}$ are themselves closed Abelian subgroup schemes of $\mathbb S_{\Lambda}$ (see \eqref{br}). Thus, by \cite[Expos\'e I, Proposition 4.6.6]{sga3} (recalled in Section~\ref{onWfromsga}), we conclude that the corresponding rank $1$ Lie algebra bundle $\Lie(\mathfrak U_{r, \bwt, U})$ is a direct summand of the Lie algebra bundle $\Lie(\mathfrak B_{r}|_{U})$.

\item
  We claim that   $\Lie(\mathfrak U_{r, \bwt, U})$ extends to $\spec(\Lambda)$ as the line bundle $\mathcal{R}_{r}$ of \eqref{Rdecomp}. Indeed, since each $\mathcal{R}_{r}$ restricts to $\Lie(\mathfrak U_{r, \bwt, U})$ on $U$, it follows that by taking their reflexive closures inside $\mathcal{R}_{\bwt}$  (see Proposition~\ref{gpschLiestabomega}), each $\Lie(\mathfrak U_{r, \bwt, U})$ extends to $\spec(\Lambda)$ as the line bundle $\mathcal{R}_r$. This extension $\mathcal{R}_{r}$ is a summand of $\Lie(\mathfrak{B}_{r})$ because its restriction to $U$ has this property.

Moreover, these are extensions as Lie algebra subbundles of $\mathcal{R}_{\bwt}$ because  by the Hartogs lemma, the Lie bracket extends across codimension bigger than $2$.  

\item
  By Step~\ref{step31}, the root group schemes $\mathfrak U_{r, \bwt, U}$  already have the extensions as unipotent group schemes given by  $\mathfrak U_{r, \bwt, \Lambda}$. Again, since the Lie algebra bundles $\mathcal{R}_{r}$ and $\Lie(\mathfrak U_{r, \bwt, \Lambda})$  
agree on $U$,  we have an isomorphism of Lie algebra bundles
\beqa\label{liealgisom}  \mathcal{R}_{r} \simeq  \Lie\left(\mathfrak U_{r, \bwt, \Lambda}\right).
\eeqa
This therefore identifies $\Lie(\mathfrak U_{r, \bwt, \Lambda})$ as a Lie subalgebra bundle of $\mathcal{R}_{\bwt}$ as well as a {\em direct summand} of $\Lie(\mathfrak{B}_{r})$.

\item
  We observe that $\mathfrak U_{r, \bwt, \Lambda}$ arises from the functor $W$ construction of Section~\ref{onWfromsga}. More precisely, we have scheme-theoretic isomorphisms of vector group schemes
 \beqa\label{liealgisomandmore}
\mathfrak U_{r, \bwt, \Lambda} = W\left(\Lie\left(\mathfrak U_{r, \bwt, \Lambda}\right)\right).
\eeqa
\end{enumerate}
 
Since the $\Lie(\mathfrak U_{r, \bwt, \Lambda})$ are direct summands of $\Lie(\mathfrak{B}_{r})$, by Section~\ref{onWfromsga}  it follows that the $\mathfrak U_{r, \bwt, \Lambda}$ are closed subgroup schemes of $\mathfrak{B}_{r}$ and therefore of $\mathbb{S}_{\Lambda}$. It follows in particular that the inclusions $\mathfrak U_{r, \bwt, \Lambda} \subset {\mathfrak Y}_{\Lambda}$ are {\em closed embeddings},  as claimed. 

\item\label{step33}
  We have arrived at the setting of \cite[Section~2.2.2]{bruhattits}. For $r \in \Phi^{\pm}$, we denote the scheme-theoretic product of group schemes $\mathfrak U_{r, \bwt, \Lambda}$ as the $\Lambda$-scheme $\mathfrak U_{\bwt, {\Lambda}}^{\pm}$. Since each $\mathfrak U_{r, \bwt, \Lambda}$ is the schematic closure of $\mathfrak U_{r, \bwt, U}$ in $\mathbb{S}_{\Lambda}$, it follows that the $\mathfrak U_{\bwt, {\Lambda}}^{\pm}$ are the {\em schematic closures} of $\mathfrak U_{\bwt, U}^{\pm}$ in $\mathbb{S}_{\Lambda}$.  The $\mathfrak U_{\bwt, U}^{\pm}$ are smooth unipotent subgroup schemes of $\mathbb S_{\Lambda}$ over $U$, and since the $\mathfrak U_{\bwt, {\Lambda}}^{\pm}$ are flat, by \cite[Section~1.2.7]{bruhattits}, it follows that the $\mathfrak U_{\bwt, {\Lambda}}^{\pm}$ are closed subgroup schemes of $\mathbb S_{\Lambda}$. Hence we get the smooth unipotent group scheme structure on the $\mathfrak U_{\bwt, {\Lambda}}^{\pm}$ as well as realize them as closed subschemes of ${\mathfrak Y}_{\Lambda}$.

This completes the proofs of Steps~\ref{step33},~\ref{step32} and~\ref{step31} and, as a consequence, of Step~\ref{step3}.
\end{enumerate}

\item\label{step4}
    \emph{Closed embedding of the big cell in $\mathfrak{Y}_{\Lambda}$.} We finally need to ensure that the open immersion $j_{U}\colon \mathfrak U_{\bwt, U}^{-} \times \cT_{\bwt, U} \times \mathfrak U_{\bwt, U}^{+}   \hookrightarrow \gG_{\bwt, U}$ extends as an open immersion $j_{\Lambda}\colon \mathfrak U_{\bwt, \Lambda}^{+} \times \cT_{\bwt, \Lambda} \times \mathfrak U_{\bwt, \Lambda}^{-}   \hookrightarrow {\mathfrak Y}_{\Lambda}$.

More precisely, we show that the schematic closure of the image of the multiplication map of $\mathfrak U_{\bwt, \Lambda}^{-} \times \cT_{\bwt, \Lambda} \times \mathfrak U_{\bwt, \Lambda}^{+}   $ in $\mathbb S_{\Lambda}$ is actually $\mathfrak Y_{\Lambda}$ and then derive the open immersion property as a consequence.

By using the discussion preceding \eqref{bigbigcell} and taking closures term by term, we see that  $\mathfrak U_{\bwt, \Lambda}^{-} \times \cT_{\bwt, \Lambda} \times \mathfrak U_{\bwt, \Lambda}^{+}   $ is a flat closed subscheme of $\mathfrak B^{-} \times \cT \times \mathfrak B^{+}$, which is itself an open subscheme of $\mathbb S_{\Lambda}$. In other words, the product defines an isomorphism of schemes of $\mathfrak U_{\bwt, \Lambda}^{-} \times \cT_{\bwt, \Lambda} \times \mathfrak U_{\bwt, \Lambda}^{+} $ onto a flat closed subscheme  
\beqa\label{mathfrakc}
\mathfrak C_{{\bwt,\Lambda}} \subset \Im\left(\mathfrak B^{-} \times \cT \times \mathfrak B^{+}\right)
\eeqa 
and, moreover, $\Im\big(\mathfrak B^{-} \times \cT \times \mathfrak B^{+}\big)$ is  an open subscheme of  $\mathbb S_{\Lambda}$.

Let $\bar{\mathfrak C}$ be the schematic closure of $\mathfrak C_{\bwt,\Lambda}$ in $\mathbb S_{\Lambda}$. We {\em claim} that $\bar{\mathfrak C} = {\mathfrak Y}_{\Lambda}$. By what we have checked above, it is clear that $\bar{\mathfrak C}$ is contained in $ {\mathfrak Y}_{\Lambda}$.

Conversely, the big cell $ \mathfrak C_{\bwt, {\tt F}}$ is an open dense subscheme of   $\gG_{\bwt, {\tt F}}$, hence $\bar{\mathfrak C} \supset \gG_{\bwt, {\tt F}}$, implying that $\bar{\mathfrak C} $ contains ${\mathfrak Y}_{\Lambda}$. Hence, by applying \cite[Section~1.2.6]{bruhattits}, we conclude that $\mathfrak C_{\bwt,\Lambda}$ is an {\em open dense subscheme} of  its schematic closure which has been identified with ${\mathfrak Y}_{\Lambda}$.
\end{enumerate}
This completes the proof of Claim~\ref{claim1} and hence of Part~\eqref{part1} of the argument.

\medskip
\noindent{\em Part \eqref{part2} of the proof}.
We now state and give the proof of   Bruhat and Tits \cite[Proposition 2.2.10]{bruhattits}, minimally tweaked and tailored for our purposes.  The result is referred to in a couple of places in their work (see \cite[Section~3.9.4]{bruhattits}), where they talk of the possibility of an extension of their theory to {\tt 2BT}-group schemes; but for us it is indispensable. In \cite[Proposition 2.2.10]{bruhattits}, the role of the group scheme $\mathfrak{GL}(M)$  is played by the group scheme $\mathbb S_{\bwt,\Lambda}$ (see Definition~\ref{shell}) in our situation. Hence we almost ``reproduce'' their  rather long and technical proof in its entirety, with suitable changes.
 
\begin{prop}[\textit{cf.} \protect{\cite[Proposition 2.2.10]{bruhattits}}] \label{2.2.10}
  Let $\mathfrak H$ be the image of\, $\mathfrak C_{\bwt,\Lambda} \times \mathfrak C_{\bwt,\Lambda}$ in $\mathbb S_{\Lambda}$ by the product morphism $\pi$ of\, $\mathbb S_{\Lambda}$. Then $\mathfrak H$ is open in 
${\mathfrak Y}_{\Lambda}$ and is a subgroup scheme of\, $\mathbb S_{\Lambda}$, with fibre $\gG_{\bwt, U}$ over $U$, flat,  of finite type, containing $\mathfrak U_{\bwt, \Lambda}^{-}$, $\cT_{\bwt, \Lambda}$ and $\mathfrak U_{\bwt, U}^{+}$ as closed subgroup schemes and containing $\mathfrak C_{\bwt,\Lambda}$ as an open subscheme \textup{(}{\em a big cell}\textup{)}. Furthermore, the group scheme $\mathfrak H$ is  a quasi-affine smooth group scheme over $\Lambda$.
\end{prop}

\begin{proof} For simplicity of notation, we will do away with all {\em multiple Greek letters}  in subscripts. 

As a preliminary step so as to follow the rest of the argument in \cite{bruhattits}, we need to check that $\mathfrak U^{-}_{\Lambda}\times\cT_{\Lambda}$ and $\cT_{\Lambda}\times\mathfrak U^{+}_{\Lambda}$ are group schemes.

For this we consider the morphism
\beqa
\mathfrak U^{-}_{\Lambda} \times \cT_{\Lambda} \lra \mathbb S_{\Lambda}
\eeqa
given by $(u,t) \mapsto t\cdot u\cdot t^{-1}$. By \cite[Section~4.4.19]{bruhattits}, this gives an action of $T_{U}$ on $\mathfrak U^{-}_{U}$, and hence over $U$, the image lands in $\mathfrak U^{-}_{U}$. Since $\mathfrak U^{-}_{\Lambda}$ is a closed subscheme of $\mathbb S_{\Lambda}$ and since $U$ contains all points of height at most~$1$, it follows that the morphism given by conjugation has its image in $\mathfrak U^{-}_{\Lambda}$. In other words, $\mathfrak U^{-}_{\Lambda}\times\cT_{\Lambda} = \cT_{\Lambda}\times\mathfrak U^{-}_{\Lambda}$ in $\mathbb S_{\Lambda}$; \textit{i.e.} it is a subgroup scheme; see \cite[Section~3.10, p.~166]{demgab}.
Similarly, $\cT_{\Lambda}\times\mathfrak U^{+}_{\Lambda}$ is also a group scheme.
\linebreak

As $\mathfrak C \times \mathfrak C$ is flat and $\pi(\mathfrak C_{U} \times \mathfrak C_{U}) \subset \gG_{U}$, by Lemma~\ref{[1.2.7, p.~18]bruhattits} (where  $\mathfrak Y$ plays the role of $Z$), we have the inclusion $\mathfrak H \subset 
\mathfrak Y$.

By what we have seen in part~\eqref{part1} above, $\mathfrak C$ is an open subset of $\mathfrak Y$. So the inverse image $\Gamma$ of $\mathfrak C$ by $\pi\colon\mathfrak C \times \mathfrak C \to \mathfrak Y$ is therefore an open subset of $\mathfrak C \times \mathfrak C$. Since $\mathfrak U^{-}_{\Lambda}\times \cT_{\Lambda} $ and $\cT_{\Lambda}\times \mathfrak U^{+}_{\Lambda}$ are groups, we get the inclusion
\beqa\label{one}
\Gamma \supset \left(\left(\mathfrak U^{-}_{\Lambda} \times \cT_{\Lambda} \times \{1\}\right) \times \left(\mathfrak U^{-}_{\Lambda} \times \cT_{\Lambda} \times \{1\}\right)\right) \cup \left(\left(\{1\} \times \cT_{\Lambda} \times \mathfrak U^{+}_{\Lambda}\right) \times \left(\{1\} \times \cT_{\Lambda} \times \mathfrak U^{+}_{\Lambda}\right)\right).
\eeqa

Let $p \in \spec(\Lambda)$ be a closed point, and let $k$ be the residue field of the local ring $\Lambda_{p}$. The fibre $\Gamma_{p}$ of $\Gamma$ is an open subset of $\mathfrak C_{p} \times \mathfrak C_{p}$, dense in $\mathfrak C_{p} \times \mathfrak C_{p}$: in fact, if $u,u' \in \mathfrak U^{-}_{p}$ and $z,z' \in \cT_{p}$, the set $V$ of $(v,v') \in \mathfrak U^{+}_{p} \times \mathfrak U^{+}_{p}$ such that $\big((u,z,v),(u',z',v')\big) \in \Gamma_{p}$ is an open subset of  $\mathfrak U^{+}_{p} \times \mathfrak U^{+}_{p}$ and is non-empty since $(1,1) \in V$ by \eqref{one}, therefore dense in $\mathfrak U^{+}_{p} \times \mathfrak U^{+}_{p}$ since $\mathfrak U^{+}_{p}$ is connected.

Let $H_{p}$ be the closure of $\mathfrak C_{p}$ in $\mathbb S_{p}$. As $\mathfrak C_{p}$ is open in $\mathfrak Y_{p}$, {\em a fortiori}, $\mathfrak C_{p}$ is open in $H_{p}$. Moreover, $\Gamma_{p}$ is dense in $H_{p} \times H_{p}$, and $\pi(H_{p} \times H_{p}) \subset H_{p}$ since by definition $\pi(\Gamma_{p}) \subset \mathfrak Y_{p}$. In other words, $H_{p}$ is a subgroup scheme of $\mathbb S_{p}$. As $\mathfrak C_{p}$ is open dense in $H_{p}$, we have $H_{p} = \mathfrak C_{p}\times \mathfrak C_{p}$, hence
\beqa\label{two}
H_{p} = \mathfrak H \cap \mathbb S_{p}.
\eeqa
We also have $H_{p} = \mathfrak C_{p}^{-1}\times\mathfrak C_{p} = \mathfrak U^{+}_{p} \cT_{p} \mathfrak U^{-}_{p}\times \mathfrak U^{-}_{p}\cT_{p} \mathfrak U^{+}_{p} = \mathfrak U^{+}_{p} \mathfrak C_{p}$,
and hence 
\beqa\label{three}
H_{p} = \bigcup_{a \in \mathfrak U^{+}_{p}(k)} a\times \mathfrak C_{p}.
\eeqa
We now show that $\mathfrak H$ {\em is open in $\mathfrak Y$ and is flat.} Firstly, $\mathfrak H$ is a {\em constructible} subset of $\mathfrak Y$. To show that it is open, (since $\Lambda$ is Noetherian), it suffices to show that if $x \in \mathfrak H$ and $y \in \mathfrak Y$ are such that $x \in \overline{\{y\}}$, then $y \in \mathfrak H$ (see \cite[Corollaire~3.4, p.~76]{demgab}).

Let $p$ be the projection of $x$ in $\spec(\Lambda)$. Since the smooth group scheme $\mathbb S_{\Lambda}$  is obtained by a Weil restriction of scalars (see the paragraph just before Definition~\ref{shell}), it behaves well under base change. Let $j\colon\spec(\Lambda_{p}) \to \spec(\Lambda)$ be the induced morphism. As the local ring $\Lambda_{p}$ is flat over $\Lambda$, the schematic closures of $\gG_{U}$ and $\mathfrak C \times_{\Lambda} \Lambda_{p}$ in $j^{*}(\mathbb S_{\Lambda})$ coincide with $j^{*}(\mathfrak Y)$   (see \cite[Section~2.5, Proposition 2]{blr}, \cite[Section~4.14, p.~56]{demgab}). Hence we may reduce to the case when $\Lambda$ is a local ring with maximal ideal $p$.

Let $B_{x}$ be the local ring at $x$ in $\mathbb S_{\Lambda}$, and let $\tilde{B}$ be the strict henselization of $B_{x}$. Let $x' \in \spec(\tilde{B})$ be the closed point, and let $\alpha\colon  \spec(\tilde{B}) \to \spec(B_{x})$ be the canonical morphism: we have $\alpha(x') = x$. For us the residue field at $x$ is algebraically closed and hence by \eqref{three}, there exists an $a \in \mathfrak U^{+}_{p}(k)$ such that $x \in a\cdot \mathfrak C_{p}$. Since $\mathfrak U^{+}_{\Lambda}$ is smooth, by Hensel's lemma, there exists a section $\hat{a} \in \mathfrak U^{+}_{\Lambda}(\tilde{B})$ of $\mathfrak U^{+}_{\Lambda}$ above $\spec(\tilde{B})$ such that $\hat{a}(x') = a$.

We now do a  base change by $\Lambda \to \tilde{B}$. Let $\mathfrak X$ (resp.\ $X$) be a $\Lambda$-scheme (resp.\ a $U$-scheme, where $U$ has been base changed to the local ring at $p$), set $\tilde{\mathfrak X} := \mathfrak X \times_{\spec(\Lambda)} \spec(\tilde{B})$ (resp.\ $\tilde{X} := X \times_{U} \tilde{U}$, with $\tilde{U} = U \times_{\spec(\Lambda)} \spec(\tilde{B})$). If $f\colon \mathfrak X \to \mathfrak Z$ is a morphism of $\Lambda$-schemes, set $\tilde{f} = f \times \id\colon \tilde{\mathfrak X} \to \tilde{\mathfrak Z}$. Finally, let $\beta\colon\tilde{\mathbb S} \to \mathbb S_{\Lambda}$ be the first projection. As $\tilde{B}$ is flat over $B_{x}$ and $B_{x}$ is flat over $\Lambda$ (as $\mathbb S_{\Lambda}$ is flat), $\tilde{B}$ is flat over $\Lambda$, and  we see as above that $\tilde{\mathfrak Y}$ is the schematic closure of both $\tilde{\mathfrak C}$ and $\tilde{\gG}_{U}$ in $\tilde{\mathbb S}$. On the other hand, the left-translation by the section $\hat{a}$ is an automorphism of the scheme $\tilde{\mathbb S}$, and the restriction $\hat{a}_{U}$ of $\hat{a}$ to $\tilde{U}$ belongs to $\mathfrak U^{+}_{p}(\tilde{U}) \subset \tilde{\gG}_{U}(\tilde{U})$; thus we have $\hat{a}\cdot\tilde{\gG}_{U} = \hat{a}_{U}\cdot\tilde{\gG}_{U} = \tilde{\gG}_{U}$ and $\hat{a}\cdot\tilde{\mathfrak Y} = \tilde{\mathfrak Y}$. As $\tilde{\mathfrak C}$ is an open in $\tilde{\mathfrak Y}$, it follows that $\hat{a}\cdot\tilde{\mathfrak C}$ is {\em open in} $\tilde{\mathfrak Y}$.

As  $\tilde{B}$ is faithfully flat over $B_{x}$, the morphism $\alpha$ is surjective. We have $x \in \overline{\{y\}}$, thus $y \in \spec(B_{x})$ and there exists a $y' \in \spec(\tilde{B})$ such that $\alpha(y') = y$. Moreover, $x' \in \overline{\{y'\}}$ since the closed point of $\spec(\tilde{B})$ belongs to the closure of any other point of $\spec(\Lambda)$. Set $x'' := (\alpha \times \id)(x') = (x,x')$ and $y'':= (\alpha \times \id)(y') = (y,y')$; then $x'' \in \hat{a}\cdot \tilde{\mathfrak C}$, $y'' \in \tilde{\mathfrak Y}$ and $x'' \in \overline{\{y''\}}$. As $\hat{a}\cdot\tilde{\mathfrak C}$ is open in $\tilde{\mathfrak Y}$, it follows that
\beqa
y'' \in \hat{a}\cdot\tilde{\mathfrak C} = \tilde{\pi} \circ \left(\hat{a} \times \id\right) \left(\spec\left(\tilde{B} \times_{\tilde{B}} \tilde{\mathfrak C}\right)\right) \subset \tilde{\pi}\left(\tilde{\mathfrak C} \times_{\tilde{B}} \tilde{\mathfrak C}\right).
\eeqa
We then deduce that
\beqa
y = \beta\left(y''\right) \in \beta \circ \tilde{\pi}\left(\tilde{\mathfrak C} \times_{\tilde{B}} \tilde{\mathfrak C}\right) = \pi \circ \left(\beta \times \beta\right)\left(\tilde{\mathfrak C} \times_{\tilde{B}} \tilde{\mathfrak C}\right) \subset \pi\left(\mathfrak C \times \mathfrak C\right) = \mathfrak H 
\eeqa 
and hence that $\mathfrak H$ is {\em open in} $\mathfrak Y$.

Moreover,  $\hat{a}\cdot\tilde{\mathfrak C} $ is flat over $\spec(\tilde{B})$. As it is a neighbourhood of $x$ in $\mathfrak H_{\tilde{B}}$, it follows that the local ring at $x$ in $\mathfrak H_{\tilde{B}}$ is flat over $\spec(\tilde{B})$, and by faithfully flat descent, the local ring of $x$ in $\mathfrak H$ is flat over $\Lambda$. Hence $\mathfrak H$  {\em is flat}.

It follows that $\mathfrak H$, being an open subscheme of $\mathfrak Y$, is a subscheme of $\mathbb S_{\Lambda}$ (which is no reason for it to be affine!). It is a subgroup scheme of $\mathbb S_{\Lambda}$. In fact, the restriction to $\mathfrak H$ of the product $\pi$ of $\mathbb S_{\Lambda}$ sends $\mathfrak H_{U} \times \mathfrak H_{U}$ into $\gG_{U}$, and hence, since $\mathfrak H$ is flat, it factorises by a morphism $\mathfrak H \times \mathfrak H \to \mathfrak Y$; see Lemma~\ref{[1.2.7, p.~18]bruhattits}. But set-theoretically, we have $\pi(\mathfrak H \times \mathfrak H) \subset \mathfrak H$ since the fibres of $\mathfrak H$ are subgroups; see \eqref{three}. Since $\mathfrak H$ is an open in $\mathfrak Y$, it implies that the morphism $\pi\colon \mathfrak H \times \mathfrak H \to \mathfrak Y$ factorises through $\mathfrak H$. A similar argument for the {\em inverse} morphism shows that $\mathfrak H$ is a subgroup scheme.

Now $\mathfrak H$ is an open subscheme of $\mathfrak Y$ which is closed in the {\tt schematic hull} $\mathbb S_{\Lambda}$. Further, by  \eqref{mathfraks}, $\mathbb S_{\Lambda}$ is affine and hence $\mathfrak H$ is {\em quasi-affine}. That it has the remaining stated properties is easy to check. \end{proof}

\brem The quasi-affineness can also be deduced as follows. Since the fibre of $\mathfrak H$ over $U$ is $\gG_{U}$, which is affine, and since $\mathfrak H$, being the image of $\mathfrak C \times \mathfrak C$, has connected fibres, by \cite[Corollary VII.2.2]{raynaud}, we conclude that $\mathfrak H$ is {\em quasi-affine}. \erem

\brem\label{2.2.10p} When the characteristic of the residue field is zero, the group scheme $\mathfrak H$ is {\em flat} and hence smooth. Over residue fields of positive characteristics, the smoothness of $\mathfrak H$ is ensured by the smoothness of the torus $\cT_{\bwt, \Lambda}$ since the unipotent groups $\mathfrak U_{\bwt, \Lambda}^{\pm}$ are smooth. By the argument in the paragraph following \eqref{bigcell 2}, the torus on the big open subset extends as a torus $\cT_{\bwt, \Lambda}$ and is therefore smooth  even in the general case. \erem

To complete the proof of Theorem~\ref{n-regiogrpsch},  we now set 
\beqa
\gG_{\bwt} := {\mathfrak H}.
\eeqa 
This group scheme clearly satisfies the properties in the statement of Theorem~\ref{multipargrpsch}. Indeed, in general it is quasi-affine, its sections can be computed by restriction to $U$, \eqref{mt4} and \eqref{mt5} follow from part~\eqref{part1}  of the proof, and  the proof of the characterization and uniqueness of the group scheme follows exactly as in Theorem~\ref{n-paragrpsch}.
\end{proof}

\subsection{Proof of Theorem~\ref{thedescription}}

\begin{proof} Since $\gG_{\bwt}$ is quasi-affine by Theorem~\ref{n-regiogrpsch}, as in Remark~\ref{diagonalremark} we can identify the restriction of $\gG_{\bwt}$ to a subdiagonal (see Theorem~\ref{thedescription}\eqref{grat1-b}) with the group scheme $\gG_{f_{(i_{1}, \ldots, i_{m})}}$, where $f_{i_{\ell}} := f_{\Omega_{i_{\ell}}}$. The description of the closed fibre follows from \cite[Proposition~4.6.5 and Corollaire~4.6.12]{bruhattits} applied to the concave function  $ f_{(i_{1}, \ldots, i_{m})} = \sum_{\ell = 1}^{m} f_{i_{\ell}} $.\end{proof}

\section{Schematization of {\tt n}-bounded groups \texorpdfstring{${\tt P}_{\bf f}$}{P\textunderscore f}}
\label{schematization3}
The notation in this section is as in the previous one; namely, $\bf A \simeq \mathbb A^{n}$ and ${\bf A_{0}} \subset {\bf A}$ is the complement of the coordinate hyperplanes. In this section we return to the general case of an almost simple group~$G$. We complete the picture by proving Theorem~\ref{multipargrpsch} for $n$-concave functions of {\tt type III}. More precisely, we construct the group scheme over $\bf A$, given the data of $n$-concave functions ${\bf f} = \{f_{j}\}$ on the root system $\Phi$ of $G$, along with an assignment of the {\tt BT}-group schemes $\{\gG_{f_{j}}\}_{j =1}^{n}$ at the generic points of each of the coordinate hyperplanes $H_{j} \subset \bf A$.

Although the strategy of the proof is broadly a reduction to the proof in the $\SL(n)_{\bwt}$ case, in the present case of $(G,{\bf f})$, there are two crucial differences. First, unlike in the case of the concave function $f_{\Omega}$, we do not have any way of assigning points in the affine apartment (see Example~\ref{gratifying} for an example) to get a product {\tt n}-parahoric group scheme ${\mathbb S}_{\Lambda}$ (see Section~\ref{schematichullslambda}) for $G$ which gave us an ambient group scheme to carry out the strategy. And second, as it will become clearer by the discussion below, the condition on the characteristic of the residue field needs to be weakened to $p > {\tt h}_{G}$.

As a first step, we recall some of the main results of \cite{bruhattits}. These will play the key role in this section. In the setting of  \cite[Section~3]{bruhattits}, $K$ is an infinite field containing a subring $A$ having $K$ as its field of fractions. The aim of that section (see \cite[Section~3.1.3]{bruhattits}) is to show that if $G$ has a {\em donn\'ee radicielle sch\'ematique} $\mathcal{D}=(\mathcal{Z}, (\mathcal{U}_{a})_{a \in \Phi})$ (see \cite[D\'efinition~3.1.1]{bruhattits}) or a {\em schematic root datum}, then there exists an $A$-group scheme $\gG$ extending $G \times \spec(K)$  with big cell given by $\mathcal{D}$. This is done by showing that for any faithful representation $G \hookrightarrow \GL(V)$, there is an $A$-lattice $M \subset V_{K}$ such that the {\tt BT}-group scheme $\gG$ is the schematic closure of $G_{K}$ in $\mathfrak{GL}(M)$. In \cite[Section~4.5.4]{bruhattits}, when $K$ is equipped with a valuation, it is shown that given any concave function $f$ on $\Phi$, one can construct a {\em schematic root datum} for $G$. We apply these results to our context as follows. Since $G$ is almost simple, we may assume that we have a factorization $G \hookrightarrow \SL(V)$. Hence, $\gG_{f}$ is the schematic closure of $G_{K}$ in the schematic closure of $\SL(V)_{K}$ in $\mathfrak{GL}(M)$. By choosing an $\cO$-basis for $M$, it is easy to see that we get a concave function $h$ on the root system $\Phi_{V}$ of $\SL(V)$ such that the schematic closure of $\SL(V)_{K}$ in $\mathfrak{GL}(M)$ is precisely the {\tt BT}-group scheme $\SL(V)_{h}$.

By Proposition~\ref{btremarque} and the remarks preceding it, we may assume that $h$ is an {\em optimal concave function} on $\Phi_{V}$, and hence it is of the form $h_{\Omega}$ for a bounded subset $\Omega$ in the affine apartment of $\SL(V)$.

Thus, we are reduced to the following situation: the {\tt BT}-group scheme $\gG_{f}$ is the schematic closure of $G_{K}$ in $\SL(V)_{\Omega}$. This can be carried out for each of the $\{f_{j}\}_{j =1}^{n}$, and we get bounded subsets $\bwt = (\Omega_{1}, \ldots, \Omega_{n})$ in the affine apartment of $\SL(V)$.

By Theorem~\ref{n-regiogrpsch}, taking $\bwt = (\Omega_{1}, \ldots, \Omega_{n})$ and the group $\SL(V)$, we obtain an {\tt nBT}-group scheme $\mathfrak{SL}_{\bwt}$ on~$\bf A$. Although only quasi-affine, this group scheme will play the role of the schematic hull $\mathbb{S}_{\Lambda}$ (see \eqref{mathfraks}) of Section~\ref{schematization2}. Further, on $U$ whose complement in $\bf A$ has codimension at least~$2$, we have an inclusion of group schemes 
\beqa
\gG_{\bf f, U} \subset \mathfrak{SL}_{\bwt, U}.
\eeqa

The proof strategy for proving  Theorem~\ref{n-regiogrpsch}  now almost applies. Let us explain the difference. In that proof, it was shown in Step~\ref{step32} that the root groups $\mathfrak U_{r, {\bf \Omega}, \Lambda}$ are the schematic closures of $\mathfrak U_{r, {\bf \Omega}, U}$ in $\mathbb S_{\Lambda}$ by the intermediary of {\em vector subgroup} $\mathfrak B_r$ of $\mathfrak S_{\Lambda}$. This vector group exists because of the specific {\em diagonal embedding} into the ambient group scheme $\mathbb S_{\Lambda}$. In the present setting, however, our chosen representation, namely $G \subset \SL(V)$, is arbitrary and cannot be expected to have the desirable property that for every root $r \in \Phi$, we have a vector subgroup scheme of $\SL(n)_{\bwt}$ whose restriction to $U$ contains $\mathfrak U_{r, {\bf f}, U}$. Consequently, to replace the arguments of Step~\ref{step32}, we recover the group law of unipotent group schemes by the Baker--Campbell--Hausdorff formula. We do this as follows. 
   If $\{c_{\alpha}\}$ denote the coefficients of the higher root, recall that the Coxeter number  ${\tt h}_{G}$  of $G$ is defined as $1 + \sum c_{\alpha}$. 
Let $\mathfrak{N}$ be a nilpotent Lie algebra bundle on $\mathbf{A}$.
By \cite[Section 2.1]{serre} (see also \cite[Section~2.2]{epiga} and other references there), {\em when the characteristic satisfies $p > {\tt h}_{G}$}, the Baker--Campbell--Hausdorff  group law equips the scheme $W(\mathfrak{N})$ with the structure of a {\em smooth, affine, unipotent group scheme with connected fibres} together with an isomorphism $\exp\colon{\mathfrak N} \to \Lie(W(\mathfrak{N}))$ of nilpotent Lie algebra bundles. Conversely, if $\mathfrak U$ is an affine, smooth, unipotent group scheme with connected fibres, we have a unique isomorphism of unipotent group schemes $\exp\colon W(\Lie(\mathfrak U)) \to \mathfrak U$.  Note that the Baker--Campbell--Hausdorff formula realizes the possible non-Abelian unipotent group scheme structures on~$ \mathfrak U$. In the absence of vector group schemes $\mathfrak B_{r}$, this is the key substitute. Let us now go over arguments that replace Step~\ref{step32}.

As in Step~\ref{step32}, we need to prove  that the extended root group schemes $\mathfrak U_{r, \bf f, \bf A}$  are closed subgroup schemes of $\mathfrak{SL}_{\bwt}$ on ${\bf A}$. We begin by noting that over the open subset $U \subset {\bf A}$, the root group schemes $\mathfrak U_{r, \bf f, U}$ are closed subgroup schemes of the restriction to $U$ of one of the big cells, say $\mathfrak B^{+}$ of $\mathfrak{SL}_{\bwt}$. By Theorem~\ref{n-regiogrpsch}, $\mathfrak B^{+}$ of $\mathfrak{SL}_{\bwt}$ is affine.  We set $\mathfrak N^{\pm} := \Lie(\mathfrak B^{\pm})$. This is a nilpotent Lie algebra bundle on $\bf A$. When  $p > {\tt h}_{G}$, we have the natural isomorphism
\beqa
\exp\colon W\left(\mathfrak N^{\pm}\right) \simeq \mathfrak B^{\pm}
\eeqa
of affine, smooth, connected unipotent group schemes.  Since $\mathfrak U_{r, \bf f, U} \subset \mathfrak B^{+}_{U}$ is a closed subgroup scheme, we may view this inclusion in terms of the ``W''-functor as follows: 
\beqa
W\left(\Lie\left(\mathfrak U_{r, \bf f, U}\right)\right) \subset W\left(\Lie\left(\mathfrak B^{+}_{U}\right)\right).
\eeqa
By this closed embedding of affine $U$-schemes,  we get the Lie algebra subbundle $\Lie(\mathfrak U_{r, \bf f, U}) \subset \Lie(\mathfrak B^{+}_{U})$ structure. This is in fact a {\em direct summand} of locally free sheaves on $U$. Now taking the {\em reflexive closure} $\overline{\Lie(\mathfrak U_{r, \bf f, U})}$ of the rank $1$ Lie algebra bundle $\Lie(\mathfrak U_{r, \bf f, U})$, we see that $\overline{\Lie(\mathfrak U_{r, \bf f, U})}$ is a locally free {\em direct summand} of $\mathfrak N^{+}$ over $\bf A$. 

As in Step~\ref{step32}, the reflexive closure $\overline{\Lie(\mathfrak U_{r, \bf f, U})}$ gets identified with the Lie algebra line bundle $\mathcal{R}_{r}$. It is therefore a {\em direct summand} of $\mathfrak N^{+}$ over ${\bf A}$. This implies that as affine schemes, we have a closed embedding
\beqa
W\left(\mathcal{R}_{r}\right) \subset W\left(\mathfrak N^{+}\right) \simeq \mathfrak B^{+}. 
\eeqa
Setting  $\mathfrak U_{r, \bf f, \bf A}:= W(\mathcal{R}_{r})$, it follows that the scheme underlying $\mathfrak U_{r, \bf f, \bf A}$ is a closed subscheme of $\mathfrak B^{+}$. By the Baker--Campbell--Hausdorff formula, this inclusion is in fact one of closed subgroup schemes over $\bf A$ and hence of $\mathfrak{SL}_{\bwt}$ on $\bf A$. This completes the proof of the key difference we mentioned above. In particular, at this stage the {\em big cell} $\mathfrak B_{\Phi}$ can also easily be seen to exist with the standard properties exactly as in the proof of Theorem~\ref{n-regiogrpsch}.

With these changes in place, the second part of the proof of  Theorem~\ref{n-regiogrpsch} may be applied. We get the required {\tt nBT}-group scheme $\gG_{\bf f, \bf A}$ associated to ${\bf f}$ on ${\bf A}$. It is open inside a closed subscheme, say $\mathcal Z$, of $\mathfrak{SL}_{\bwt}$. Recall that $\mathfrak{SL}_{\bwt}$ is open inside a closed subscheme, say $C$, of $\mathbb{S}_{\Lambda}$. Let $\overline{\mathcal Z}$ denote the closure of $\mathcal Z$ in $C$. The intersection of $\overline{\mathcal Z}$ with $\mathfrak{SL}_{\bwt}$ in $C$ is $\mathcal Z$. It follows that $\mathcal Z$ is open in $\overline{\mathcal Z}$. Therefore, $\gG_{\bf f, \bf A}$ is open inside $\overline{\mathcal Z}$. Now $\overline{\mathcal Z}$ is affine because $\mathbb{S}_{\Lambda}$ is affine. We conclude that $\gG_{\bf f, \bf A}$ is quasi-affine (this can also be checked using \cite[Corollary VII.2.2]{raynaud}). In conclusion, when $\charr(k)$ is at least $ h_{G}$ and satisfies the hypothesis of Section~\ref{charassum}, we have constructed  $\gG_{\bf f, \bf A}$ as a {\em smooth, quasi-affine} group scheme with all the properties stated in Theorem~\ref{multipargrpsch} for $n$-concave functions of {\tt type III}.

\section{Schematization of {\tt n}-Moy--Prasad groups} \label{schematizationmp}
In this section we discuss Moy--Prasad groups over complete discrete valuation rings in the equal-characteristic case. The aim is to recover these anew  in the spirit of \cite{base}. This then gives the generalization to the higher-dimensional bases. 

\subsection{Tameness Assumptions} \label{tamp} In this section the given data $(\bvt, {\bf e}') := \big((\theta_{1}, e'_{1}), \ldots, (\theta_{n}, e'_{n})\big)$ will consist of an $n$-tuple of rational one-parameter subgroups $\theta_{i}$ and positive rational numbers $e_{i}'$. We suppose that there exists a positive number $d$ coprime to the characteristic of the residue field such that $d \theta_{i}$ is a one-parameter subgroups and the $d e_{i}'$ are positive numbers.

\subsection{The unit group}
Let $\cO := k\llbracket z \rrbracket$ and $K$ be its quotient field. Let $B :=  k\llbracket \omega \rrbracket$ and $L$ be its quotient field, where $\omega^{d} = z$. We view $N := \spec(B)$ as a Galois cover over $D := \spec(\cO)$ with $\Gal(N/D) = {\sf\mu_{d}}$. Let $p\colon N \to D$ be the quotient morphism. Let $\rho\colon {\sf\mu_{d}} \ra G$ be a representation. Consider the twisted ${\sf\mu_{d}}$-action on $N \times G$ defined by 
\begin{equation} \label{twistedactionontrivbundle}
\gamma (u,g) =(\gamma u, \rho(\gamma) g) \quad \text{for }  u \in N, \gamma \in {\sf\mu_{d}} .
\end{equation}
This defines a $({\sf\mu_{d}},G)$-bundle on $N$. {\sl We will denote it by $E$}.

 Define the {\sl unit group} of $E$ as an automorphisms of the $({\sf\mu_{d}},G)$-torsor $E$:
\begin{equation} \label{ugp}
{\tt U}_{E} := \Aut_{({\sf\mu_{d}},G)}(E).
\end{equation}

Let $e \geq 0$ be a non-negative integer. Consider the thickened fibre $j_{e}\colon E|_{\Spec(B/{\omega}^{e}B)} \hookrightarrow E$. Define ${\tt U}^{(e)}_{E}$ to be
\begin{equation}\label{mpg}
{\tt U}^{(e)}_{E} := \ker\left(\Aut_{({\sf\mu_{d}},G)}(E) \lra \Aut_{({\sf\mu_{d}},G)}\left(j_{e}^{*}(E)\right)\right), 
\end{equation}
\textit{i.e.}~the subgroup of $({\sf\mu_{d}},G)$-automorphisms of $E$ over $N$ 
inducing identity on the restriction of  $E$ to $\Spec(B/{\omega}^{e}B)$. For $\underline{e=0}$, this is just the unit group ${\tt U}_{E}$ of $E$.

Consider the adjoint group scheme of $E$
\begin{equation} \label{adjointgroupscheme}
\gG_{E} \simeq G \times N, 
\end{equation}
which is the constant group scheme together with a $\Gamma$-action.

It is well known (\textit{cf.} \cite[Section~2.4, p.~234]{yu})
that for each $e \geq 0$, by a {\sl dilatation} of $\gG_{E}$, there exists a smooth group scheme $\gG_{E}^{(e)}$ such that 
\begin{equation} \label{diladjointgroupscheme} \gG_{E}^{(e)}(B) = \Gamma\left(\omega^{e},\gG_{E}\right), 
\end{equation}
where $\Gamma(\omega^{e},\gG_{E})$ is the {\sl $\supth{e}$ congruence subgroup} defined by
$$\Gamma(\omega^{e},\gG_{E}) = \ker\left(\gG_{E}(B) \lra \gG_{E}\left(B/{\omega}^{e}B\right)\right).$$

\subsection{The representation \texorpdfstring{$\boldsymbol{\rho\colon\Gamma \rightarrow T}$}{rho: Gamma to T} and rational one-parameter subgroups \texorpdfstring{$\boldsymbol{\theta}$}{theta} of \texorpdfstring{$\boldsymbol{T}$}{T}}
Following \cite[Lemma 2.2.8]{base}, make the identification
\begin{equation}\label{elemlemma}
  \Hom( {\sf\mu_{d}},T)= \frac{Y(T)}{d Y(T)} \simeq \frac{1}{d} Y(T) \mod(Y(T)).
\end{equation}

We will interpret this identification explicitly to suit our specific needs. Given a one-parameter subgroup
 \begin{equation}\label{deltaasrat}
\Delta\colon \GG_m    \lra T,
\end{equation}
by projecting it onto $\frac{Y(T)}{d Y(T)}$ and using  \eqref{elemlemma}, we obtain from this $1$-PS $\Delta$
a rational one-parameter subgroup $\theta_{\Delta} \in \frac{1}{d} Y(T) \mod(Y(T))$. By \eqref{elemlemma} again, this gives a homomorphism 
\begin{equation} \label{rhodelta}
\rho_{\Delta}\colon {\sf\mu_{d}} \lra T.
\end{equation}

In the reverse construction, given a homomorphism  $\rho \in \Hom({\sf\mu_{d}},T)$, we choose    a rational one-parameter subgroup $\theta_{\rho}$ determined by the class of $\rho$ (we will need to make a choice here). Then $\theta_{\rho}$ defines a one-parameter subgroup 
\begin{equation}\label{deltaandtheta}
\Delta = d \theta_{\rho}\colon \GG_m \lra T.
\end{equation} 

\brem\label{uptoyt} The point $\theta_{\rho}$ determined above by the choice of $\Delta$ firstly gives a point of the Weyl alcove $\mathbf{a}_0$ (see Section~\ref{liedata}). This point gives a precise {\em weight} as obtained by the process of invariant push-forward in \cite{base}. More precisely, the point to be noted is that this construction $\rho \mapsto \theta_{\rho}$ gives a point of the {\em fundamental domain} for the action of $Y(T)$ and not just the affine Weyl group. Further, the proofs in Section~\ref{isogpsch} and that of Theorem~\ref{mixexcharbs} show that \cite[Theorem 2.3.1 and Proposition 5.1.2]{base} hold more generally for $\theta_{\rho}$ in the fundamental domain of $Y(T)$ and not just the affine Weyl group. \erem 

\subsection{Brief background on Moy--Prasad groups}
Moy--Prasad groups have been defined in \cite[Section~2.6]{moyprasad}.
Using a Chevalley basis of $G$,  below we give an equivalent reformulation which is suitable for our setup.
 Let $\cA = \Hom(\GG_m,T) \otimes \RR$ denote the apartment corresponding to $T$. For a point $\theta \in \cA$, recall that the parahoric group $\cP_{\theta} \subset G(K)$ is defined as
\begin{equation} \label{parahoric}
{\cP}_{\theta} = \left\langle T(\cO),\, U_r\left( z^{- \lfloor r(\theta) \rfloor}\cO\right), r \in \Phi \right\rangle,  
\end{equation}
and for a non-negative integer $e' \geq 0$, the Moy--Prasad group  is defined as
\begin{equation} \label{moyprasad}
G_{\theta,e'} := \left\langle T\left(1+ z^{\lceil e' \rceil} \cO\right),\, U_r( z^{-\lfloor r(\theta)-e' \rfloor} \cO), r \in \Phi \right\rangle. 
\end{equation}

While considering parahoric or Moy--Prasad groups, it suffices to assume that $\theta \in \cA_T$ is a rational point and $e' \in \QQ_{\geq 0}$. This follows because for any $(r,n) \in \Phi \times \ZZ$, the functional $\cA_T \times \RR_{\geq 0} \ra \RR$ defined by $(x,e') \ms r(x) +n - e'$ is defined over rationals.

\subsection{Invariant direct image of dilatations of reductive group schemes} \label{isogpsch}
The proof of the following theorem follows  the proofs of  \cite[Theorem 2.3.1 and Proposition 5.1.2]{base} closely.

\begin{thm} \label{mpasunit}
Let $\gG_{\theta,e'}$ be the Bruhat--Tits group scheme satisfying $\gG_{\theta,e'}(\cO) = G_{\theta,e'}$ \textup{(}see~\eqref{moyprasad}\textup{)}. Then \textup{(}see~\eqref{diladjointgroupscheme}\textup{)} we have a natural isomorphism of group schemes 
\begin{equation}\label{moypasweilrest}
p_{*}^{\sf\mu_{d}}\left(\gG_{E}^{(e)}\right) \simeq \gG_{\theta_\rho,e'}. 
\end{equation}
\end{thm}

\begin{proof} Note that the $\cO$-valued points are enough to determine the group scheme since $\gG_{\theta,e'}$ is \'etoff\'e in the sense of \cite[D\'efinition~1.7.1]{bruhattits}.

We begin with a few observations. 
Now an element $\phi_0 \in {\tt U}_{E}$ corresponds to a ${\sf\mu_{d}}$-equivariant automorphism of the $({\sf\mu_{d}},G)$-bundle, \textit{i.e.} $\phi_0\colon N \times G \ra N \times G$ together with the equivariance. In other words, $\phi_0 \in {\tt U}^{(e)}_{E}$ corresponds to a regular morphism $\phi\colon N \ra G$ (\textit{i.e.} $\phi \in G(B)$), given by
\begin{equation}
\phi_0(u,g)=(u, \phi(u) g),
\end{equation}
which satisfies the following conditions: 
\begin{enumerate}
\item ${\sf\mu_{d}}$-equivariance, \textit{i.e.} 
  \begin{equation}\label{gammaphi}
    \phi(\gamma u)=\rho(\gamma) \phi(u) \rho(\gamma)^{-1}, \quad u \in N, \gamma \in {\sf\mu_{d}}; 
\end{equation}
\item when $\underline{e \geq 1}$, the map $\phi$ restricted to $\Spec(B/{\omega}^{e}B)$ is the identity, \textit{i.e.} 
\begin{equation}\label{2conds}
\phi \in \ker\left(G(B) \lra G(B/{\omega}^{e}B)\right).
\end{equation} 
\end{enumerate}
Thus, as in \cite[Equation~(2.3.1.5)]{base}, we can identify ${\tt U}_{E}$ with $\phi$ satisfying \eqref{gammaphi}; let  ${\tt U}^{(e)}_{E} \subset {\tt U}_{E}$ be the subgroup of ${\tt U}_{E}$ given by the above two conditions.

From here onward, we follow closely the proof of \cite[Theorem 2.3.1]{base}.  

For each $\phi \in {\tt U}_{E}$, define 
\begin{equation}\label{deltabusiness}
\psi(\omega) = \Delta^{-1}\cdot\phi\cdot\Delta.
\end{equation} 
Then $\psi(\gamma\cdot u) = \psi(u)$ for all $\gamma \in {\sf\mu_{d}}, u \in N$. It therefore descends to a rational function $\tilde{\psi}\colon\spec(K) \to G$ such that 
\begin{equation} \label{descends}
\tilde{\psi}(z) = \psi(\omega).
\end{equation}

Now for $\underline{e \geq 1}$, by \eqref{2conds}, we have $\phi(o) = 1$, which lies in $G^\circ$, where $G^{\circ} \subset G$ is the big cell in $G$. Since $N$ is a  formal neighbourhood of the origin $o$, it follows that $\phi(N) \subset G^{\circ}$. 

For $\underline{e = 0}$, in the following we will assume that we work with a $\phi$ such that $\phi(N) \subset G^\circ$. The set of such $\phi$ may be called the big cell of ${\tt U}_{E}$; we will denote it by the symbol $\mathfrak U$.

The condition $\phi(N) \subset G^\circ$ allows us to describe $\phi\colon N \ra G^\circ$  uniquely as a tuple of the following morphisms: 
$$
\left\{ \{\phi_r\colon N \lra U_r \}_{r \in \Phi},\; \phi_{t}\colon N \lra T \right\},$$
where $U_r \subset G$ is the root subgroup corresponding to $r \in \Phi$. Let us view $\phi_r$ and $\phi_{t}$ as elements of $G(B) \subset G(L)$.

We may write 
\begin{equation}
\phi_r(\omega)= \psi_r(\omega)\omega^{r(\Delta)} =  \tilde{\psi}_r(z) \omega^{r(\Delta)}.
\end{equation}
Since $\phi_r(\omega)$ has a zero of order at least $e$ in $\omega$,  
$$\frac{\phi_r(\omega)}{\omega^e}= \tilde{\psi}_r(z) \omega^{r(\Delta)-e},$$
and the right side defines a  regular function in $\omega$; \textit{i.e.} it extends to a regular function on $N$. So $\tilde{\psi_r}(z)$ can have a pole in $\omega$ of order at most $r(\Delta)-e$. Thus in $z$ it can have a pole of order at most
$$
\left[ \frac{r(\Delta) - e}{d}\right]; 
$$
\textit{i.e.} $[r(\frac{\Delta}{d}) - \frac{e}{d}] \stackrel{\eqref{deltaandtheta}}{=}[r(\theta_{\rho})- e']$, where $e' = \frac{e}{d}$. This forces the containment
\begin{equation}
\tilde{\psi_r}(z) \in U_r\left(z^{-\lfloor r(\theta_{\rho})-e' \rfloor} \cO\right),
\end{equation}
and hence $\tilde{\psi_r}$ belongs to $G_{\theta_{\rho},e'}$.
On the other hand, since $\phi_{t}(u) \in T$ for $u \in N$, it follows that $\phi_{t}(\gamma. u)= \rho(\gamma) \phi_{t}(u) \rho(\gamma)^{-1}= \phi_{t}(u)$. Hence $\phi_{t}\colon N \ra T$ is ${\sf\mu_{d}}$-invariant. Also, by \eqref{deltabusiness}, it follows that $\phi_{t} = \psi_{t}$. So it descends to a regular function on $D$, and this function is $\tilde{\psi}_{t}\colon D \ra T$.  Now $\phi_{t}(u) -1$ has a zero of order at least $e$ in $\omega$. So by the definition of $e'$, we see that $\tilde{\psi_{t}}-1$ has a zero in $z$ of order at least
$$
\left\lceil \frac{e}{d} \right\rceil= \left\lceil e' \right\rceil.
$$
Hence $\tilde{\psi}_{t}$ also belongs to $G_{\theta_{\rho},e'}$. 

Since $\psi$ and $\phi$ are related by a conjugation in $G(L)$ by $\Delta(\omega)$,  $\psi$ like $\phi$ can also be described by $\psi_r$ and $\psi_{t}$ uniquely. So the same holds for $\tilde{\psi}$. Hence since $\tilde{\psi}_r$ and $\tilde{\psi}_{t}$ belong to $G_{\theta_{\rho},e'}$, therefore so does $\tilde{\psi} \in G_{\theta_{\rho},e'}$.

For $\underline{e \geq 1}$,\footnote{For $\underline{e = 0}$, we get an isomorphism 
\begin{equation} \label{bigcelliso}
j\colon \mathfrak U \lra \mathfrak B
\end{equation}
of the big cell $\mathfrak U$ of ${\tt U}_{E}$ with that of $G_{\theta_{\rho}}$, which we will denote by $\mathfrak B$.} we get the following isomorphism of groups in the spirit of \cite[Theorem 2.3.1]{base}: 
\begin{equation}
{\tt U}^{(e)}_{E} \simeq  G_{\theta_{\rho},e'}.
\end{equation}

Combining these with \eqref{gammaphi} and \eqref{2conds}, we deduce that for $\underline{e \geq 1}$, 
\begin{equation*}\pushQED{\qed}
p_{*}^{\sf\mu_{d}}\left(\gG_{E}^{(e)}\right)(A) =
 \gG_{E}^{(e)}(B)^{\sf\mu_{d}} = {\tt U}^{(e)}_{E} \simeq G_{\theta_{\rho},e'} \simeq \gG_{\theta_\rho,e'}(A). \qedhere \popQED
	\end{equation*}
\renewcommand{\qed}{}   
\end{proof}

\begin{rem}   Note that the existence and defining properties for {\tt BT}-group schemes which realize Moy--Prasad groups were constructed in \cite{yu}. They arise by enlarging the definition of concave functions and constructing {\tt BT} group schemes associated to them. We observe that the above realization of these group schemes as {\sl invariant direct images} can be interpreted as giving a  different construction of these group schemes, namely as invariant direct images of congruence subgroup schemes from ramified covers. We have done this in the simple case field of characteristic zero, and this can be extended to ``good'' characteristics for  the group $G$. \end{rem}

\subsection{The {\tt n}-Moy--Prasad-group scheme}
The proof of Theorem~\ref{n-paragrpsch} does not quite apply as it is. Unlike the scenario of Theorem~\ref{n-paragrpsch}, on the Kawamata cover we no longer have a constant group scheme with fibre $G$ but have  {\em congruence subgroup schemes}. 

\begin{thm}
We have an {\tt n-Moy--Prasad}-group scheme $\gG_{(\bvt, \bf e')}$ which is {\em smooth and affine} on $\spec(\cO)$ which is given by the data of $(\bvt, {\bf e}') := \big((\theta_{1}, e'_{1}), \ldots, (\theta_{n}, e'_{n})\big)$. It satisfies the properties in Theorem~\ref{multipargrpsch}.
\end{thm}

\begin{proof} For the given data $(\bvt, {\bf e}') := \big((\theta_{1}, e'_{1}), \ldots, (\theta_{n}, e'_{n})\big)$, under the assumptions of Section~\ref{tamp}, setting $\bvt = (\theta_{1}, \ldots, \theta_{n})$, we have a Kawamata cover $q\colon Z \to {\bf A}$ which recovers the {\tt n}-parahoric group scheme $\gG_{\bvt}$ on ${\bf A}$ as the invariant push-forward $q^{\Gamma}_{*}(\mathcal H)$, where $\mathcal H$ is a group scheme on $Z$ with fibres $G$. Recall that $q$ is ramified with varying ramification indices $d_{i}$. In fact, by our tameness hypothesis, the ramification indices $d_i$ are such that $e_i:=d_i e_i'$ are positive integers which are coprime to the residue characteristic. Without loss of generality, we may suppose that the $e_i$ are {\em increasing with $i$}.

Set $e_0=0$.  
Now for $j$ going from $1$ to $n$, we use the main result in \cite{mayeux} to construct an iterated dilatation of $\mathcal H$ along each component $q^{*}(\cup_{i \geq j}D_{i})_{\red}$ of the reduced divisor in $Z$ thickened by $e_j-e_{j-1}$.  We get a new smooth, affine group scheme $\tilde{\mathcal H}$ on $Z$. By the universal property of dilatation, it follows that $q^{\Gamma}_{*}(\tilde{\mathcal H})$ is the required {\tt n-Moy--Prasad}-group scheme on ${\bf A}$. It is smooth and affine since $\tilde{\mathcal H}$ is so and because  the invariant direct image functor preserves these properties.
\end{proof}

\section{Schematization in the mixed-characteristic case}\label{mixedstuff}

Let $(\cO, \mathfrak m)$ be a complete discrete valuation ring with residue field $k$ of characteristic $p$ and uniformizer $z$. We want to extend all of the results of the previous sections to the mixed-characteristics case. This will be possible under the assumptions of Section~\ref{charassum}.

\subsection{Extension of \texorpdfstring{\cite[Theorem 2.3.1 and Proposition 5.1.2]{base}}{[BS15, Theorem 2.3.1 and Proposition 5.1.2]}} \label{bstomixed}
Let $\theta \in \cA_T$ be a rational point satisfying the assumptions of Section~\ref{charassum}. As in Section~\ref{charassum}, we assume that $d$ is coprime to $\charr(k)$ such that $d \theta \in Y(T)$. 
Consider the Eisenstein polynomial $
f(x)=x^d - z \in A[x]$.
Let $B:=\cO[x]/f(x)$. Thus $z$ is totally ramified in $B$, and the residue field $k_{B}$ has $d$ distinct roots of unity. Let $\varpi$ be a uniformizer in $B$ such that $\varpi^d=z$. Let $\zeta$ be a primitive $\supth{d}$ root of unity in $B$.

Let $K$ be its quotient field of $\cO$ and $L$  the  quotient field of $B$. We view $N := \spec(B)$ as a Galois cover over $D := \spec(\cO)$ with $\Gal(N/D) = {\sf\mu_{d}}$. Let $p\colon N \to D$ be the quotient morphism. Let $E$ be the trivial principal $G$-bundle $N \times G$ together with the twisted $\mu_d$-action as in (\ref{twistedactionontrivbundle}). This makes $E$ into a $(\mu_d,G)$-bundle on $N$.  Let $\gG_{E}$ be its adjoint group scheme (see~\eqref{adjointgroupscheme}). Let $\gG_{\theta_\rho}$ be the Bruhat--Tits group scheme over $\cO$ associated to it as in Theorem~\ref{mpasunit}. Analogously to \eqref{parahoricwithgen}, it will be convenient to write 
\begin{equation} \label{Pthetagen}
{\cP}_{\theta} = \left\langle T(\cO),\, U_r\left(z^{m_r(\theta)} \cO\right), r \in \Phi \right\rangle
\end{equation}
 in both the equal- and mixed-characteristic cases. This abuse of notation in the mixed-characteristic case should not cause any confusion because of the structure theorem for complete discrete valuation rings. 

Following the discussion in Section~\ref{schematizationmp} under the tameness assumptions, let us prove the required ge\-ne\-ra\-li\-zation  of \cite[Theorem 2.3.1]{base}, which is Theorem~\ref{mpasunit} for $e'=0$.

\begin{thm} \label{mixexcharbs}
We have an isomorphism of group schemes 
\begin{equation}\label{mixedcharweilrest}
p_{*}^{\sf\mu_{d}}\left(\gG_{E}^{}\right) \simeq \gG_{\theta_\rho}. 
\end{equation}
\end{thm}

\begin{proof} Let us state at the outset that by \cite[Section~1, p.~13]{bruhattits}, the references we will cite are valid for an infinite field $K$ containing a subring $A$ such that $K$ is its quotient field. 

The $\cO$-group scheme $\gG_{\theta_\rho}$ is connected and smooth because it is a parahoric group scheme. Weil restriction of scalars followed by taking invariants preserves connectedness. Further, Weil restrictions of smooth group schemes are smooth. Moreover, under the tameness assumptions, it is easily checked  that the fixed-point subscheme is also smooth. Thus, the $\cO$-group scheme $p_{*}^{\sf\mu_{d}}(\gG_{E}^{})$ is connected and smooth. As in~\eqref{gammaphi}, the  restriction of $p_{*}^{\sf\mu_{d}}(\gG_{E}^{})$ to $\spec(K)$ consists of $\phi\colon \spec(L) \rightarrow G$ satisfying the 
${\sf\mu_{d}}$-equivariance; \textit{i.e.} 
\begin{equation} \phi(\gamma u)=\rho(\gamma) \phi(u) \rho(\gamma)^{-1}, \quad u \in \spec(L), \gamma \in {\sf\mu_{d}}.
\end{equation}

Conjugate $\phi$ by $\Delta$ (see~\eqref{deltaandtheta}), as is also done in~\eqref{deltabusiness}; the conjugate $\Delta^{-1} \phi \Delta$ descends to a rational function $\tilde{\psi}\colon \spec(K) \rightarrow G$ as in~\eqref{descends}. This defines on $\spec(K)$ a morphism
\begin{equation}
j_K\colon p_{*}^{\sf\mu_{d}}\left(\gG_{E}^{}\right)|_K \lra  \gG_{\theta_\rho}|_K. 
\end{equation}
By evaluating on sections, it may be verified that $j_K$ is an isomorphism. 

Recall that in the proof of Theorem~\ref{mpasunit}, we defined the big cell $\mathfrak U$ of the unit group ${\tt U}_{E}$ (see~\eqref{ugp}) to be those $\phi\colon \spec(B) \rightarrow G$ whose image lies in the big cell $G^\circ$ of $G$. This defines the big cell of $p_{*}^{\sf\mu_{d}}(\gG_{E}^{})$. We will denote it as $\mathfrak U$ as well. Taking $\mathfrak U$ to be the open neighbourhood of the identity of $p_{*}^{\sf\mu_{d}}(\gG_{E}^{})$, we exhibited in~\eqref{bigcelliso} the isomorphism $j\colon \mathfrak U \rightarrow \mathfrak B$ onto the big cell $\mathfrak B$ of $ \gG_{\theta_\rho}$. 

Now by \cite[Section~1.2.13]{bruhattits}, $j_K$ extends uniquely to an $\cO$-morphism 
$$
j\colon p_{*}^{\sf\mu_{d}}\left(\gG_{E}^{}\right) \lra \gG_{\theta_\rho}
$$
of group schemes. Since $j$ restrict to an isomorphism from $\mathfrak U$ to $\mathfrak B$ between big cells, by \cite[Section~1.2.14]{bruhattits}, it follows that $j$ is an isomorphism of $p_{*}^{\sf\mu_{d}}(\gG_{E}^{})$ onto an open subgroup scheme of $\gG_{\theta_\rho}$. Thus $j$ is an isomorphism because $\gG_{\theta_\rho}$ is connected.
\end{proof}

Theorem~\ref{mixexcharbs} has two consequences. Firstly, \cite[Proposition 5.1.2]{base}, recalled in \eqref{frombalses}, generalizes following \cite[Sections~4 and 5]{base}.
Secondly,  Theorem~\ref{mpasunit} extends without any extra conditions imposed by $e'$. This is so because in the proof of Theorem~\ref{mpasunit}, we only need to have distinct roots of unity corresponding to the point $\theta$ in the apartment. 

\begin{rem} We observe that for the results of Sections~\ref{schematization1},~\ref{schematization2} and~\ref{schematization3}, the assumption of algebraic closedness of the residue field $k$ is not essential. The existence of a Kawamata cover (see Section~\ref{kawa}) is needed only for affine spaces, and this can be achieved by hand under the assumptions of Section~\ref{charassum}. We however always assume that $k$ is {\em perfect} so as to be able to apply \cite{bruhattits}. One needs to have extensions $k_{B}$ as above with the desired roots of unity for the above extensions of \cite[Theorem 2.3.1 and Proposition 5.1.2]{base}. This is possible under the assumptions of Section~\ref{charassum}.
\end{rem}

\subsection{The extension of Theorem~\ref{n-paragrpsch}}\label{extension1}
In this subsection, we  denote the uniformizer $z$ of $\cO$ by $x_{0}$. Let $\mathbf{A}_{\cO}:=\mathbb{A}^{n}_{\cO}=\spec(\cO[x_{1},\ldots,x_{n}])$.
For $0 \leq i \leq n$, let $H_i$ be the coordinate hyperplanes defined by the vanishing of $x_{i}$, and let $\zeta_{i}$ be their generic points. 

\begin{enumerate}[leftmargin=*]
\item \label{a} If we let $(X,D) = ({\mathbf{A}_{\cO}}, H)$, then the analogue of Proposition~\ref{onliestr} holds without any change for an arbitrary {\tt (n+1)-concave}-function ${\bf f}$. {\em This statement holds for any characteristic of the residue field}. In the following we denote the Lie algebra bundle by $\mathcal{R}$.
\item Observe that since $\cO$ is assumed to be complete, the formal power series ring ${\cO} \llbracket x_{1},\ldots,x_{n} \rrbracket$ is the completion of $\cO[x_{1},\ldots,x_{n}]$ at $\mathfrak m[x_{1},\ldots,x_{n}]$.  We begin with a natural extension of the notion of an {\tt (n+1)-bounded group} associated to an {\tt (n+1)-concave function} ${\bf f}$ on the root system $\Phi$ of $G$ by defining
\begin{equation}\label{evenmoregen1}
{\cP}_{\bf f} := \left\langle T\left({\cO} \llbracket x_{1},\ldots,x_{n} \rrbracket\right),\, U_r\left( \prod_{0 \leq i \leq n} x_{i}^{f_{i}(r)} {\cO} \llbracket x_{1},\ldots,x_{n} \rrbracket\right), r \in \Phi \right\rangle.
\end{equation}

\item Suppose we are given an ordered $(n+1)$-tuple $(\Sigma_{0},\ldots, \Sigma_{n})$ of facets of the apartment $\cA_T$. Then by Section~\ref{charassum}, we may choose a point $\bvt:=(\theta_{0},\ldots, \theta_{n})$  in $(\Sigma_{0},\ldots, \Sigma_{n})$ such that if $d_{i}$ is the minimal integer such that $d_{i} \theta_{i}$ lies in $Y(T)$, then each $ d_{i}$ is coprime to $\charr(k)$ and $k$ contains primitive $\supth{d_{i}}$ roots of unity. More generally, suppose we are given a point $\bvt:=(\theta_{0},\ldots, \theta_{n}) \in \cA_T^{n}$ satisfying the assumptions in Section~\ref{charassum}. Let $\tilde{\cO}$ be a complete discrete valuation ring containing $\mathcal{O}$ where $x_{0}$ ramifies to order a multiple of $d_{0}$ and that contains the primitive $\supth{d_{i}}$ roots of unity for $0 \leq i \leq n$.

\item Let $X' \subset \mathbf{A}_{\cO}$ denote the open subset of height $1$ prime ideals as in the paragraph after \eqref{invdi1}. Taking ${\bf f}$ to be the concave function determined by $\theta$ in item~\eqref{a}, let $\cR'$ denote the restriction of $\cR$ to $X'$.

\item We are now in the setting of Theorem~\ref{Artin-Weil-Kawamata} with the only difference that $\mathbf{A}_{\cO}$ replaces $\mathbb{A}^{n+1}_{k}$.

\item Let $\bf{B}_{\tilde{\cO}}:= \spec(\tilde{\cO}[y_{1},\ldots,y_{n}])$, where $y_{i}^{d_i}=x_{i}$. Let $q\colon \mathbf{B}_{\tilde{\cO}} \rightarrow \mathbf{A}_{\cO}$ be the natural projection map with Galois group $\Gamma$. This plays the role of an explicit Kawamata covering.
We set $\mathbf{A}_{0}:=\mathbf{A}_{\cO} \setminus H$,  $\mathbf{B}_{0}:=q^{-1}(\mathbf{A}_{0})$, $Z' :=q^{-1}(X')$ and $\tilde{X}_{i}:= q^{-1}(X_{i})$. Let $q_{i}\colon\tilde{X}_{i} \rightarrow X_{i}$ be the restriction of $q$. Let $\Gamma_{i}$ denote the Galois group of $q_{i}$. We are now in the setting of Section~\ref{restriction}. 

\item By applying Section~\ref{bstomixed} to the uniformizer $x_{0}$ and the generalization of \cite[Proposition 5.1.2]{base} for the rest of the height $1$ primes given by the variables $x_{j}$, we get a smooth, affine group scheme $\cH'$ on $Z'$ such that the Weil restrictions gives $q^{\Gamma}_{*}(\cH') = \gG'$ on $X'$.

\item Let $V' = \Lie(\cH')$ be the Lie algebra bundle on $Z'$ as in the discussion after \eqref{gunja}. Then, $V'$ is a Lie algebra bundle with fibre type $\mathfrak g$. Furthermore, the Cartan decomposition on $\cR'$ lifts to give an equivariant Cartan decomposition of the Lie algebra bundle $V'$ on $Z'$. Hence its reflexive closure $V$ also gets an equivariant structure and is also locally free. 

\item As in the proof of Theorem~\ref{Artin-Weil-Kawamata}, now that we have the Lie algebra bundle $V$ with fibre $\mathfrak g$, we get a group scheme $\cH$ on $\mathbf{B}_{\tilde{\cO}} $ and, by taking invariant direct images,
the group scheme $\gG := q^{\Gamma}_{*}(\cH)$ extending $\gG'$ to the whole of $\mathbf{A}_{\cO}$. This proves the extension of Theorem~\ref{n-paragrpsch}  in the setting of mixed characteristics with the tameness assumptions.
\item We also note that as in Proposition~\ref{bigcell1}, the big cell structure also descends to give one for $\gG$ on the whole of $\mathbf{A}_{\cO}$.
\item Finally, we note that the arguments in Theorem~\ref{n-paragrpsch} for the uniqueness of the group scheme carry over to the mixed-characteristic situation as well.
\end{enumerate}

\subsection{The main theorem for concave functions}
By Section~\ref{extension1}\eqref{a}, 
we see that there is a Lie algebra bundle on $(\mathbf{A}_{\cO},H)$ associated to an {\tt (n+1)-concave} function on $\Phi$. The results in Section~\ref{schematization2} now work under the tameness assumptions, see Section~\ref{tameness}, or more generally under the assumptions on $\charr(k)$ of Section~\ref{charassum}. We therefore get the extension of Theorem~\ref{multipargrpsch} to $(\mathbf{A}_{\cO},H)$ (see Remark~\ref{2.2.10p}). Note that this is for the groups $G = \SL(n), G_{2}, F_{4}, E_{6}$.  

For the rest of the groups, by the arguments in Section~\ref{schematization3}, we need a faithful representation $G \hra \GL(V)$ to start with, and so if we choose a faithful representation of minimal dimension $m(G)$, then for these groups, we need to choose $p$ to be coprime to $m(G)$ and bigger than the Coxeter number ${\tt h}_{G}$ of $G$ (see Section~\ref{schematization3}). The numbers $m(G)$ are known and can be obtained from the dimension formulae of the fundamental representations. They are as follows: $m(E_{6}) = 27$, $m(E_{7}) = 56$ and for $B_{n}$, $C_{n}$ and $D_{n}$, $m(G)$ is $2n+1$, $2n$ and $2n$, respectively. In summary, we have completed the proof of Theorem~\ref{multipargrpschmixed}.

\subsection{The example of a  {\tt 2BT}-group scheme in the work of Pappas and Zhu \texorpdfstring{\cite{pz}}{[PZ13]}}\label{pappas}
As we mentioned in the introduction, an example of a  {\tt 2BT}-group scheme comes up in \cite[Theorem~4.1]{pz}.  

There are broadly two parts to \cite[Theorem 4.1]{pz}, the first being the existence of the {\tt 2BT}-group scheme on $\mathbb A^{1}_{\cO}$ and second a proof that the resulting group scheme is {\em affine}. Although  \cite[Theorem 4.1]{pz} states the case of a parahoric group scheme along one axis, the proof considers the more general case of the $2$-concave function ${\bf f} = (0, f)$, where $f$ is a concave function on the root system $\Phi$. The proof of the affineness follows ideas from \cite{yu}. We are unable to comprehend it   when $f$ is given by a non-hyperspecial point of the apartment and $G$ is split.

We first briefly describe the setting in \cite[Section 4]{pz} so as to place their result in our context.  We will then indicate how this result can be derived as a special case of our results. We note however that the setting in \cite{pz} is for the general case of a quasi-split group $G$, while we consider only the split case. 

In the notation of~Section~\ref{extension1} above, the setting in \cite{pz} is for the case $n = 1$, \textit{i.e.} when $x_{1} = u$. The axes we consider correspond to setting $x_0$ and $x_1$ as zero (here by $x_0=0$ we mean going modulo the uniformizer), but Pappas--Zhu consider the axes given by setting $u$ as the uniformizer $\varpi$ and  by setting $\varpi=0$. The result \cite[Theorem 4.1]{pz} is stated for the case when the assignment along the $x_0$-axis is the split group (see \cite[Section~4.2.2(a)]{pz}),  while the assignment in the diagonal $(\varpi=u)$-direction is given by a parahoric group associated to a point $\theta$ of the affine apartment. Thus, in our setting this case corresponds to the $2$-concave function which is ${\bf f} = (0, f)$, where $f$ is a concave function $f_{\theta}$ on the root system $\Phi$  along the $u$-direction and the $0$-concave function gives the split group scheme along the $\varpi$-direction. Observe that this gives the concave function $f=0+f$ in the $(u=\varpi)$-direction.

The key point is that only {\em a single} concave function $f$ is involved. By  \cite[Section~4.5.4]{bruhattits}, one obtains a  schematic root datum coming from $f$. We now work in the setting of \cite[Section~3.9.4]{bruhattits}. Once a faithful representation $G \hookrightarrow \GL(V)$ is fixed, the schematic root datum arising out of the choice of $f$ entails the choice of a lattice $M \subset V_{K}$. The arguments in \cite[Section~3.9.4]{bruhattits} furnishes the group scheme $\mathcal G_{\theta_{0},f}$ (see  \cite[p.~175, second paragraph]{pz}). Thus, the {\tt 2BT}-group scheme over $\mathbb A^{1}_{\cO}$ which comes up in \cite{pz} is obtained as a direct consequence of \cite[Section~3.9.4]{bruhattits}. However, the method of construction of \cite[Section~3.9.4 and Proposition~2.2.10]{bruhattits} gives {\em a priori} only a {\em quasi-affine group scheme}. So the burden of \cite[Section 4]{pz} lies in proving the {\em affineness} of the group scheme $\mathcal G_{\theta_{0},f}$.  We emphasize that all this is for a single concave function.

We now compare this with the results of the present paper. Firstly, \cite[Theorem 4.1]{pz} as stated is for the $2$-concave function $(0,f_{\theta})$, and this is a special case of Theorem~\ref{n-paragrpsch} as described in Section~\ref{extension1}. Thus we derive directly both the {\em existence and affineness} of the required group scheme  and in the more general context of an $(n+1)$-concave function $\bf f$ over $\mathbb A^{n}_{\cO}$ which comes by prescribing $n+1$ points in the affine apartment. We note however that the methods in the present paper force certain mild tameness assumptions. 

\begin{part}{Some applications in the \texorpdfstring{$\boldsymbol{\charr(k)=0}$}{char(k)=0} equicharacteristic~case}
\end{part}

\section{{\tt BT}-group schemes on wonderful embeddings in \texorpdfstring{$\boldsymbol{\charr(k)=0}$}{char(k)=0}}\label{app1}

\subsection{The wonderful compactification \texorpdfstring{$\bf X$}{X}} \label{earlyloghomo}
Let ${\bf X} := \ol{G_{\ad}}$ be the wonderful compactification of $G_{\ad}$. 
Let $\{ D_{\alpha} \mid \alpha \in S \}$ denote the irreducible smooth {\em divisors} of ${\bf X}$.  Let $D:= \cup_{\alpha \in S} D_\alpha$. Then ${\bf X} \setminus G_{\ad} = D$.   The pair $({\bf X},D)$ is the primary example of a  $(G_{\ad} \times G_{\ad})$-homogenous pair; see  \cite{brion}.
Let $P_{I}$ be the standard parabolic subgroup defined by subsets $I \subset S$, the notation being such that the Levi subgroup $L_{I,\ad}$ containing $T_{\ad}$ has root system with basis $S \setminus I$. 
Recall that the $(G_{\ad} \times G_{\ad})$-orbits in ${\bf X}$ are indexed by subsets $I \subset S$ and have the following description:
$$Z_{I} = \left(G_{\ad} \times G_{\ad} 
\right) \times _{P_{I} \times P^-_{I}} L_{I,\ad}.$$
Then by \cite[Proposition A1]{brion1}, each $Z_{I}$ contains a unique base point $z_{I}$ such that $(B \times B^-)\cdot z_{I}$ is dense in $Z_{I}$ and there is a $1$-PS $\lambda$ of $T$ satisfying  $P_{I} = P(\lambda)$ and $\lim_{t \to 0} \lambda(t) = z_{I}$. The closures of these $\lambda$ define curves $C_{I} \subset {\bf X}$ which meet the strata $Z_{I}$ transversally at $z_{I}$. In particular, if $I = \{\alpha\}$ is a singleton, then the divisor $D_{\alpha}$ is the orbit closure $\bar{Z}_{I}$ and the $1$-PS can be taken to be the fundamental co-weight $\omega^{\vee}_\alpha$. The closure of the $1$-PS $\omega^{\vee}_\alpha\colon\mathbb G_{m} \to G_{\ad}$ defines the curve $C_{\alpha} \subset {\bf X}$ transversal to the divisor $D_{\alpha}$ at the point $z_{\alpha}$. For a non-empty $I \subset S$, the $\lambda$ defining 
$C_I$ can be taken to be $\sum_{\alpha \in I} \omega^{\vee}_{\alpha}$.

\subsection{The \texorpdfstring{$\ell{\tt BT}$}{lBT}-group scheme \texorpdfstring{${\mathfrak G}_{\bf X}^{\varpi}$}{G\textunderscore X\textasciicircum varpi} on $\bf X$} 
  
Before stating  the main result of this section,  we make a few remarks which might help the reader. Let ${\bf Y_{0}} \subset {\bf Y}$ the affine toric varieties associated to the Weyl chamber and the fan of Weyl chambers. The basic underlying principle in these constructions
is that the combinatorial data encoded in the triple consisting of the Weyl chamber, the fan of Weyl chambers and the Tits building is geometrically replicated by the inclusion ${\bf Y_{0}} \subset {\bf Y} \subset {\bf X}$. The  ``wonderful'' Bruhat--Tits group scheme which arises on ${\bf X}$ has its local Weyl-chamber model on the affine toric variety ${\bf Y}_{0}$. Indeed, in this case, the Kawamata cover is even explicit; see Remark~\ref{torickawamata}. In particular, the theorem below can be executed for ${\bf Y}_{0}$, but this will give the group scheme associated to the data coming from the Weyl chamber alone.
  
We now state a consequence of our main theorem in the context of the wonderful compactification of $G_{\ad}$ which is a primary example of a smooth variety with normal crossing divisors.

\begin{thm}\label{btoverx}
Let $\bvt_{} = (\theta_{\alpha} \mid \alpha \in S)$, where the $\theta_{\alpha} \in \mathcal{A}$ are arbitrary points in the apartment of\, $T$ satisfying conditions of Section~\ref{charassum}.  There exists an affine ``wonderful'' Bruhat--Tits group scheme ${\mathfrak G}_{\bf X}^{\varpi}$ on ${\bf X}$ satisfying the following classifying properties:
\begin{enumerate} 
\item At the height $1$ primes associated to the canonical boundary divisors associated to the maximal parabolic subgroups $P_{\alpha} \subset G$, $\alpha \in S$, the group scheme ${\mathfrak G}_{\bf X}^{\varpi}$ is isomorphic to the parahoric group scheme $\gG_{\theta_{\alpha}}$.

\item For $\emptyset \neq I \subset S$, the restriction of ${\mathfrak G}_{\bf X}^{\varpi}$ to  the formal neighbourhood $U_{I}$ of $z_{I}$ in the coordinate curves $C_{I}$, see Section~\ref{earlyloghomo}, is isomorphic to the Bruhat--Tits  group scheme  associated to the concave function ${\cm_{{\bvt}_{I}}}\colon \Phi \rightarrow \mathbb{Z}$ given by $r \ms \sum_{\alpha \in I} m_r(\theta_{\alpha})$.
\item Let $\Phi_{\theta}:=\{r \in \Phi | r(\theta) \in \mathbb{Z} \}$. 
The closed fibre of $\gG_{\cm_{{\bvt}_{I}}}$ modulo its unipotent radical has root system ge\-ne\-ra\-ted by the intersection of the $\Phi_{\theta_{\alpha}}$ for $\alpha \in I$.
\end{enumerate}  
\end{thm}
 
\begin{proof} By Proposition~\ref{onliestr}, we have a Lie algebra bundle $\mathcal{R}$ on ${\bf X}$ prescribed by the first property. Thus by Theorem~\ref{Artin-Weil-Kawamata}, we get a group scheme ${\mathfrak G}_{\bf X}^{\varpi}$ on ${\bf X}$ which by construction satisfies the first property. The context of the second property is that of formal neighbourhoods, which is also that of Section~\ref{atthehyper}. As remarked before the theorem, the group scheme constructed on ${\bf X}$ has a local Weyl-chamber model on the affine toric variety ${\bf Y}_{0}$. Thus we may restrict to ${\bf Y}_{0}$. Hence the second property follows from Theorem~\ref{n-paragrpsch}. Recall, see \cite[Corollaire~4.6.12]{bruhattits}, that the reductive quotient of the closed fibre of $\gG_{f}$ has root system given by $\{r \in \Phi\mid \lceil f \rceil (-r)=- \lceil f \rceil (r) \}$. For the case of the concave function $m_r(\theta)$, this set is $\Phi_\theta$, and for $\cm_{{\bvt}_{I}}$, it will be the intersection of the $\Phi_{\theta_{\alpha}}$ for $\alpha \in I$.
\end{proof}

\begin{Cor}
Let $\theta_{\alpha}$ be the non-zero vertices of the Weyl alcove \eqref{alcovevertices}. Let $L$ be the LCM of the coefficients of the highest root and $\ell$ the semisimple rank of\, $G$.
Then setting $\bvt_{}$ as $(\theta_{\alpha} \mid \alpha \in
S)$, $(\theta_{\alpha}/(\ell+1) \mid \alpha \in S)$ and
$(L\theta_{\alpha} \mid \alpha \in S)$ gives a group scheme whose
restriction to every coordinate curve $C_{I}$ has closed fibre with
root system generated by the highest root $\alpha_{0}$ and $S
\setminus I$, or simply generated by $S
\setminus I$, or equal to that of\, $G$,
respectively.
\end{Cor}

\begin{proof} The structure of the set $\{r \in \Phi\mid \lceil f \rceil (-r)=- \lceil f \rceil (r) \}$ and the formula \eqref{alcovevertices} show that the root system of the reductive quotient of the closed fibre equals the set of roots whose simple root coefficients, coming from the subset $I$ of $S$, equal an integral multiple of the corresponding term of the highest root. In the first case, this is the set of roots generated by the highest root $\alpha_{0}$ and $S \setminus I$. In the second case, it is the set of roots generated by $S \setminus I$. In the third case, all roots in $\Phi$ satisfy this property. 
\end{proof}

\subsection{The wonderful embedding \texorpdfstring{${\bf X}^{\aff}$}{X\textasciicircum aff}}\label{loopcase}
We continue to use the notation as in previous sections.  Let $G^{\aff}$ denote the Kac--Moody group associated to the affine Dynkin diagram of $G$. Recall that $G^{\aff}$ is given by a  central extension of $L^{\ltimes}G$ by $\GG_m$. Analogously to  the wonderful compactification of $G_{\ad}$, Solis in \cite{solis} has constructed a wonderful embedding ${\bf X}^{\aff}$ for $G^{\aff}_{\ad}:=G^{\aff}/Z(G^{\aff})=\GG_m \ltimes \LG/Z(G)$.  It is an ind-scheme containing $G^{\aff}_{\ad}$ as a dense open ind-scheme and carrying an equivariant action of $L^{\ltimes} G \times L^{\ltimes} G$. 

Let $T_{\ad} := T/Z(G)$ and $T^{\ltimes}_{\ad}:= \GG_m \times T_{\ad} \subset G^{\aff}_{\ad}$, where $\GG_m$ is the rotational torus. In ${\bf X}^{\aff}$, the closure ${\bf Y}^{\aff}:=\ol{T^{\ltimes}_{\ad}}$ gives  a  torus embedding. It is  covered by the  affine Weyl group $W^{\aff}$-translates  of the affine torus embedding ${\bf Y}_{0}^{\aff} := \ol{T^{\ltimes}_{\ad,0}} \simeq  \mathbb{A}^{\ell+1}$ given by the negative Weyl alcove.

\subsection{On the torus embedding \texorpdfstring{${\bf Y}_{0}^{\aff}$}{Y\textunderscore 0 \textasciicircum aff}} 
Recall that $Z:= {\bf Y}_{0}^{\aff} \setminus T^{\ltimes}_{\ad}$ is a union $\cup_{\alpha \in \mathbb S} H_{\alpha}$ of $\ell+1$  standard coordinate hyperplanes meeting at normal crossings.   For $\alpha \in \mathbb S$, let  $\zeta_{\alpha}$ denote the generic point of the divisor $H_{\alpha}$. Let  
\begin{equation} \label{Aalpha}
A_{\alpha}= \mathcal O_{{\bf Y}_{0}^{\aff}, \zeta_{\alpha}}
\end{equation} 
be the DVRs obtained by localizing at the height $1$ primes given by the $\zeta_{\alpha}$. Let $K_{\alpha}$ be the quotient field of $A_{\alpha}$. Let $Y_{\alpha}:= \spec(A_{\alpha}) $. Note that we can identify the open subset $\spec(K_{\alpha})$ with  ${T^{\ltimes}_{\ad}} \cap Y_{\alpha}$. Let 
\begin{equation}\label{yprime} Y':= {T^{\ltimes}_{\ad}} \cup_{\alpha } Y_{\alpha}.
\end{equation}
The complement ${\bf Y}_{0}^{\aff} \setminus Y'$ can again be realized as a colimit of open subsets of ${\bf Y}_{0}^{\aff}$ whose   codimension is at least $2$ in ${\bf Y}_{0}^{\aff}$.

\subsection{Construction of a finite-dimensional Lie algebra bundle $J$ on \texorpdfstring{${\bf Y}_{0}^{\aff}$}{Y\textunderscore 0 \textasciicircum aff} together with parabolic structures} 
This construction is exactly analogous to the construction of $\mathcal{R}$ on ${\bf Y}$ in Section~\ref{constructionR}. For $\alpha \in \mathbb{S}$, let $\theta_{\alpha} \in \mathcal{A}$ be arbitrary points in the apartment satisfying the conditions of Section~\ref{charassum}. We let $T^{\ltimes}_{\ad}$, ${\bf Y}_{0}^{\aff}$, $(1,\omega^{\vee}_{\alpha})$ and $\mathbb{S}$ play the roles of ${\bf A}_{0}$, ${\bf A}$, $\lambda_i$ and $\{1,\ldots,n\}$. More precisely, we set 
\begin{equation}
\eta_{\alpha}:=\left((0,\ldots,1,\ldots,0),\theta_{\alpha} \right), \quad \alpha \in \mathbb{S}
\end{equation}
with the unique $1$ in the $\supth{\alpha}$ coordinate. Consider the loop rotation action given by \eqref{looprotationhd} and \eqref{hdconjugation}. We may prove the following theorem exactly like Theorem~\ref{gpschLiestab}.

\begin{thm} \label{liealgbunonyaff}
Let $\bvt:=(\theta_{\alpha} \mid \alpha \in \mathbb{S})$, where the $\theta_{\alpha}$ are arbitrary points in the apartment $\mathcal{A}$ of\, $T$ satisfying the conditions of Section~\ref{charassum}. Let $U$ be a formal neighbourhood of the origin in ${\bf Y}_{0}^{\aff}$. There exists a canonical Lie algebra bundle $J$ on ${\bf Y}_{0}^{\aff}$ which extends the trivial bundle with fibre $ \mathfrak g$ on ${T^{\ltimes}_{\ad}}  \subset {\bf Y}^{\aff}_{0}$, and with notation as in \eqref{Liealghd}, we have the identification of functors from the category of $k$-algebras to $k$-Lie algebras
\begin{equation}
L^+ \left(\mathfrak{P}_{\bvt}\right)= L^+\left( \mathcal R\mid_{U}\right).
\end{equation}
\end{thm}

We may prove the following corollary like Corollary~\ref{restrictiontocurves}.

\begin{Cor}
  Let $\lambda= \sum_{\alpha \in \mathbb{S}} k_{\alpha} (1,\theta_{\alpha})$ be a non-zero dominant $1$-PS of\, $T_{\ad}^{\ltimes}$ where not all $k_{\alpha}$ are zero. 
Let $\mathcal{D}$ be the formal neighbourhood of origin in ${\bf Y}_{0}^{\aff}$ associated to the curve defined by $\lambda$. Let $f\colon \Phi \rightarrow \mathbb{Z}$ be the concave function defined by the assignment $r \ms \sum_{\alpha \in \mathbb{S}} k_{\alpha} m_r(\theta_{\alpha})$. Then for any $k$-algebra $R$, we have
\begin{equation}
  L^+\left( \mathcal R\mid_{\mathcal{D}}\right)(R)=\left\langle \mathfrak{t}(R \llbracket t \rrbracket),\, \mathfrak{u}_r\left(t^{f(r)} R \llbracket t \rrbracket\right), r \in \Phi \right\rangle.
\end{equation}
\end{Cor}

\begin{Cor}\label{canparstr}
  The Lie algebra bundle $J_{{\bf Y}^{\aff}}$ gets canonical parabolic structures \textup{(}see the appendix\,\textup{)} at the generic points $\xi_{\alpha}$ of the $W^{\aff}$-translates of the divisors $H_{\alpha} \subset {\bf Y}_{0}^{\aff} $, $\alpha \in \mathbb S$. 
\end{Cor} 

\begin{proof} We prescribe a ramification index $d_\alpha$ on the divisor $H_{\alpha}$ such that $d_{\alpha} \theta_{\alpha}$ belongs to $Y(T)$. (Thus in the case when the $d_{\alpha}$ are the alcove vertices, see \eqref{alcovevertices}, the $d_\alpha$ are given by \eqref{dalpha}). Then the identification \eqref{Liestrendow} of the Lie algebra structures of $J_{{\bf Y}^{\aff}}$ and the parahoric Lie algebra structures on the localizations of the generic points of  $H_{\alpha}$ allows us to endow parabolic structures at the generic points of the divisors.\end{proof} 

\subsection{The parahoric group scheme on the torus embedding \texorpdfstring{${\bf Y}^{\aff}$}{Y\textasciicircum aff}}

\begin{thm} \label{mtY}
Let $\bvt_{} = (\theta_{\alpha} \mid \alpha \in S)$, where the $\theta_{\alpha} \in \mathcal{A}$ are arbitrary points in the apartment of\, $T$ satisfying conditions of Section~\ref{charassum}. There exist an affine  and smooth ``wonderful'' Bruhat--Tits group scheme ${\mathfrak G}_{{\bf Y}^{\aff}}^{\varpi}$ on ${\bf Y}$ together with a canonical isomorphism $\Lie({\mathfrak G}_{{\bf Y}^{\aff}}^{\varpi}) \simeq J$. It further satisfies the following classifying property: 

For any point $h \in {\bf Y}^{\aff} \setminus T^{\ltimes}_{ad}$, let  $ \mathbb I \subset \mathbb S$ be a subset such that $h \in \cap_{\alpha \in \mathbb I} H_{\alpha}$. Let $C_{\mathbb I} \subset {\bf Y}^{\aff}$ be a smooth curve with generic point in $T^{\ltimes}_{\ad}$ and closed point $h$. Let $U_{h} \subset C_{\mathbb I}$ be a formal neighbourhood of $h$. Then, the restriction ${\mathfrak G}_{{\bf Y}^{\aff}}^{\varpi}|_{U_{h}}$ is isomorphic to the Bruhat--Tits group scheme  associated to the concave function ${\text{\cursive m}_{{\bvt}_{I}}}\colon \Phi \rightarrow \mathbb{Z}$ given by $r \ms \sum_{\alpha \in I} m_r(\theta_{\alpha})$. 
\end{thm}

\begin{proof}
Recall that ${{\bf Y}^{\aff}}$ is covered by affine spaces $Y_{w} \simeq \mathbb{A}^{\ell+1}$ parametrized by the affine Weyl group $W^{\aff}$. Each $Y_{w}$ is a translate of ${{\bf Y}_{0}^{\aff}}$.  The translates of the divisors $H_{\alpha}$ meet each $Y_{w}$ in the standard hyperplanes on $\mathbb{A}^{\ell+1}$, and thus, we can prescribe the same ramification data at the hyperplanes on each of the $Y_{w}$. On the other hand, although we have simple normal crossing singularities, we do not have an analogue of the Kawamata covering lemma for schemes such as ${{\bf Y}^{\aff}}$. The lemma is known only in the setting of quasi-projective schemes. So to construct the group scheme, we employ a different approach using the naturality of the constructions for gluing.

We observe firstly that the formalism of Kawamata coverings applies in the setting of the affine spaces $Y_{w} \subset {{\bf Y}^{\aff}}$. Let $p_w\colon Z_{w} \ra Y_{w}$ be the associated Kawamata cover (see Section~\ref{kawa}) with Galois group $\Gamma_w$. In Corollary~\ref{canparstr} we observed that the Lie algebra bundle $J$ has a canonical parabolic structure. Letting $Y_{w}$ play the role of ${\bf Y_{0}}$ and  using all arguments in the proof of  Theorem~\ref{btoverx}, we obtain $\cH_{w} \ra Z_{w}$, which is a $\Gamma_w$-group scheme with  fibres isomorphic to   $G$ whose invariant direct image is a group scheme $\gG_w$ such that $\Lie(\gG_w) =J|_{Y_{w}}$. The induced parabolic structure on $J|_{Y_{w}}$ is the restriction of the one on $J$.  Indeed, by Corollary~\ref{canparstr}, these parabolic structures are essentially given at the local rings at the generic points $\xi_{\alpha}$ of the divisors $H_{\alpha}$, and hence 
these parabolic structures on $J$ agree on the intersections $Y_{uv}:=Y_{u} \cap  Y_{v}$. 

Let $Z_{uv}:= p_u^{-1}(Y_{uv})$. Let $\tilde{Z}_{uv}$ be the normalization  of a component of  $Z_{uv} \times_{Y_{uv}} Z_{vu}$. Then $\tilde{Z}_{uv}$   serves as Kawamata cover (see Section~\ref{kawa}) of $Y_{uv}$ (see \cite[Corollary 2.6]{vieweg}).  We consider the morphisms $\tilde{Z}_{uv} \to Z_{u}$ (resp.\ $\tilde{Z}_{uv} \to Z_{v}$) and let $\cH_{u,\tilde{Z}}$ (resp.\ $\cH_{v,\tilde{Z}}$) denote the pull-backs of $\cH_{u}$ (resp.\ $\cH_{v}$)  to $\tilde{Z}_{uv}$. 

Let $\Gamma$ denote the Galois group for $\tilde{Z}_{uv} \ra Y_{uv}$. Then by Lemma~\ref{Weilrestriction}, the invariant direct images of the equivariant Lie algebra bundles $\Lie(\cH_{u,\tilde{Z}})$ and $\Lie(\cH_{v,\tilde{Z}})$ coincide with the Lie algebra structure on   $J$ restricted to the ${Y_{uv}}$ and also as isomorphic parabolic bundles. Therefore, we have a natural isomorphism of equivariant Lie algebra  bundles
\begin{equation}
\Lie\left(\cH_{u,\tilde{Z}}\right) \simeq \Lie\left(\cH_{v,\tilde{Z}}\right).
\end{equation}
As in the proof of Theorem~\ref{btoverx}, this gives a canonical identification of the equivariant group schemes $\cH_{u,\tilde{Z}}$ and $\cH_{v,\tilde{Z}}$ on $\tilde{Z}_{uv}$. Since the invariant direct image of both the group schemes  $\cH_{u,\tilde{Z}}$ and $\cH_{v,\tilde{Z}}$ are the restrictions 
$\gG_{u,{Y_{uv}}}$ and $\gG_{v,{Y_{uv}}}$, it follows that on $Y_{uv}=Y_{u} \cap Y_{v}$  we get a canonical identification of group schemes
\beqa\label{inducedfromlie}
\gG_{u,{Y_{uv}}} \simeq \gG_{v,{Y_{uv}}}.
\eeqa 
These identifications are canonically induced from the gluing data of the Lie algebra bundle $J$ for the cover $Y_{w}$. Therefore, the cocycle conditions are clearly satisfied, and the identifications \eqref{inducedfromlie} glue to give  the group scheme ${\mathfrak G}_{{\bf Y}^{\aff}}^{\varpi}$ 
 on ${\bf Y}^{\aff}$. The verification of the classifying property follows exactly as in the proof of Theorem~\ref{btoverx}.
\end{proof}

\subsection{The \texorpdfstring{$(\ell+1){\tt BT}$}{(l+1)BT}-group scheme on \texorpdfstring{${\bf X}^{\aff}$}{X\textasciicircum aff}}

Let  ${\bf X}^{\aff}$ be as in Section~\ref{loopcase}. The situation in this subsection is somewhat distinct from the previously discussed cases. Unlike the scheme ${\bf X}$, the space ${\bf X}^{\aff}$ is an {\em ind-scheme}, and so we give the details of the construction of the group scheme.

We begin with a generality. Let $\mathbb{X}$ be an ind-scheme. By an open subscheme $i\colon \mathbb{U} \hookrightarrow \mathbb{X}$, we mean an ind-scheme such that for any $f\colon \spec(A) \rightarrow \mathbb{X}$, the natural morphism $\mathbb{U} \times_{\mathbb{X}} \spec(A) \rightarrow \spec(A)$ is an open immersion. For a sheaf $\mathbb{F}$ on $\mathbb{U}$, by $i_{*}(\mathbb{F})$ we mean the sheaf associated to the pre-sheaf on the ``big site'' of $\mathbb{X}$, whose sections on $f\colon \spec(A) \rightarrow \mathbb{X}$ are given by $\mathbb{F}(\mathbb{U} \times_{\mathbb{X}} \spec(A))$.

 The ind-scheme ${\bf X}^{\aff}$ has a certain open subset ${\bf X}_{0}$ whose precise definition is somewhat technical; see \cite[Section~5.1, p.~705]{solis}. Let us mention the properties relevant for us.

Recall that ${\bf X}^{\aff} = (G_{\ad}^{\aff} \times G_{\ad}^{\aff})~{\bf X}_{0}$ and in fact ${\bf X}^{\aff} = (G_{\ad}^{\aff} \times G_{\ad}^{\aff}) ~{{\bf Y}^{\aff}}$. Further, the torus embedding ${\bf Y}^{\aff}$ is covered by ${\bf Y}^{\aff}_{w} \simeq \mathbb A^{\ell +1}$, which are $W^{\aff}$-translates of ${\bf Y}^{\aff}_{0}$, where ${\bf Y}^{\aff}_{0} = {{\bf Y}^{\aff}} \cap {\bf X}_{0}$. We remark that, analogously to the case of ${\bf Y}_{0} \subset {\bf Y} \subset {\bf X}$, just as the toric variety ${\bf Y}_{0} $ was associated to the negative Weyl chamber, the toric variety  ${\bf Y}^{\aff}_{0}$ is associated to the negative Weyl alcove.

So let us also denote by  $0$  the neutral element of $W^{\aff}$.  Let $U^{\pm} \subset B^{\pm}$ be the unipotent subgroups. Let $\cU^{\pm}:= \ev^{-1}(U^{\pm})$, where $\ev\colon G(\cO) \ra G(k)$ is the evaluation map. Further, by \cite[Proposition 5.3]{solis}, 
\begin{equation} {\bf X}_{0}= \cU \times {\bf Y}^{\aff}_{0} \times \cU^-.
\end{equation}
Let ${\bf X}_{w} := \cU \times {\bf Y}^{\aff}_{w} \times \cU^-$. These cover $\cU \times {{\bf Y}^{\aff}} \times \cU^-$. For  $g \in G_{\ad}^{\aff} \times G_{\ad}^{\aff}$, let 
\begin{equation} \label{compatibilities}
  {\bf X}_g:=g {\bf X}_0, \quad {\bf X}_{g,w}:=g {\bf X}_{w}, \quad {\bf Y}^{\aff}_{g,w}:=g {\bf Y}^{\aff}_{w}.
\end{equation}
Note that we have the projection ${\bf X}_{g,w} \to {\bf Y}^{\aff}_{g,w}$, which is a $(\cU \times \cU^-)$-bundle.

 By \cite[Theorem 5.1]{solis}, the ind-scheme ${\bf X}^{\aff}$ has divisors $D_\alpha$ for $\alpha \in \mathbb S$ such that the complement of their union is ${\bf X}^{\aff} \setminus G_{\ad}^{\aff}$. The next proposition shows the existence of a finite-dimensional Lie algebra bundle on ${\bf X}^{\aff}$ which is analogous to the bundle $\cR$ over ${\bf X}$. 

 \begin{prop}\label{isotliealginfty1}
 There is a finite-dimensional Lie algebra bundle ${\bf R}$ on ${\bf X}^{\aff}$ which extends the trivial Lie algebra bundle $G_{\ad}^{\aff} \times \mathfrak g$ on the open dense subset $G_{\ad}^{\aff} \subset {\bf X}^{\aff}$ and whose restriction to ${\bf Y}^{\aff}$ is $J$. \end{prop}

\begin{proof} Since the projection ${\bf X}_{g,w} \to {\bf Y}^{\aff}_{g,w}$ is a $(\cU \times \cU^-)$-bundle, the transition functions of~$J$ and its restrictions to tubular neighbourhoods of its divisors may be used to construct a locally free sheaf $J'$ on an open subset ${\bf X'}$ containing the union of $G^{\aff}$ and the height $1$ prime ideals of ${\bf X}^{\aff}$ using the transition functions of $J$.  Let ${\bf R}$ denote its push-forward to ${\bf X}^{\aff}$. To check that the push-forward is locally free, without loss of generality we may consider its restriction to ${\bf X}_0$. But on ${\bf X}_0$ the push-forward of $J'$ restricts to the pull-back of a Lie algebra bundle $J$ on ${\bf Y}^{\aff}$ constructed in Theorem~\ref{liealgbunonyaff}, which completes the argument. The rest of the properties follow immediately.
\end{proof}

\begin{thm}\label{gpshsolis}
Let $\bvt_{} = (\theta_{\alpha} \mid \alpha \in S)$, where the $\theta_{\alpha} \in \mathcal{A}$ are arbitrary points in the apartment of\, $T$ satisfying the conditions of Section~\ref{charassum}. There exist an affine and smooth ``wonderful'' Bruhat--Tits group scheme ${\mathfrak G}_{{\bf X}^{\aff}}^{\varpi}$ on ${\bf X}^{\aff}$ together with a canonical isomorphism $\Lie({\mathfrak G}_{{\bf X}^{\aff}}^{\varpi}) \simeq {\bf R}$. It further satisfies the following classifying property: 

For any point $h \in {\bf X}^{\aff} \setminus G_{\ad}^{\aff}$, let  $ \mathbb I \subset \mathbb S$ be defined by the condition $h \in \cap_{\alpha \in \mathbb{I}} D_\alpha$. Let $C_{\mathbb I} \subset {\bf X}$ be a smooth curve with generic point in $G_{\ad}^{\aff}$ with closed point $h$. Let $U_{h} \subset C_{\mathbb I}$ be a formal  neighbourhood of $h$. Then, the restriction ${\mathfrak G}_{{\bf X}^{\aff}}^{\varpi}|_{U_{h}}$ is isomorphic to the Bruhat--Tits group scheme  associated to the concave function ${\text{\cursive m}_{{\bvt}_{I}}}\colon \Phi \rightarrow \mathbb{Z}$ given by $r \ms \sum_{\alpha \in I} m_r(\theta_{\alpha})$. 
\end{thm}

\begin{proof} We begin by observing  that ${\bf X}_{g,w}$ and ${\bf R} \ra {\bf X}^{\aff}$ play the roles of ${\bf Y}_{g,w}^{\aff}$ and $J \ra {\bf Y}^{\aff}$ in the proof of Theorem~\ref{mtY}.  Therefore, the group scheme $\gG_{g,w}$ glue together, and we obtain the global group scheme ${\mathfrak G}_{{\bf X}^{\aff}}^{\varpi}$. The verification of the classifying property follows exactly as in the proof of Theorem~\ref{btoverx}.
\end{proof}

\section{{\tt 2BT}-group schemes and degenerations of torsors}\label{mckaybtetc}

In this section we assume $\charr(k)=0$. Our aim is to revisit   certain smooth, affine group schemes  on the  minimal resolution of normal surface singularities. We work with the complete local rings and in this discussion stick to the $A_{n}$-type singularities alone and satisfy ourselves with a few remarks on the other types in a brief remark. The senior author used these constructions in \cite{balaproc}, where $\tt 2 BT$-group schemes first occur. The purpose was to construct degenerations of the moduli spaces of principal $G$-bundles on smooth projective curves when the curve is made to degenerate to an irreducible nodal curve. 

The picture and notation we use are as in \cite{balaproc}. The novel feature is that the constructions were made  using the geometric McKay correspondence  of G.~Gonz\'alez-Sprinberg and J.-L.~Verdier (see \cite{gonverd}), the role of which is intriguing.

Let $N_{d} = \spec(\frac{k\llbracket t \rrbracket \llbracket x,y \rrbracket}{(x\cdot y - t^{d})})$. We recall that $N_{d}$ is a {\em normal} surface with an isolated singularity of type $\text{A}_{d}$. By the generality of $\text{A}_{d}$-type singularities, one can realize $N_{d}$ as a quotient $\sigma\colon D \to N_{d}$ of  $D:= \spec(\frac{k\llbracket t \rrbracket\llbracket u,v \rrbracket}{(u\cdot v - t)})$ by the cyclic group ${\sf\mu_{d}} = \langle \gamma \rangle$, where $x = u^{d}$, $y = v^{d}$, $\zeta$ is a primitive $\supth{d}$-root of unity and ${\sf\mu_{d}}$ acts on $D$ as 
follows:
\beqa\label{gammaaction}
\gamma\cdot(u,v) = \left(\zeta\cdot u, \zeta^{d-1}\cdot v\right).
\eeqa

We consider the following basic diagram for all $d > 0$ (see \cite{gonverd}):
\beqa\label{keydiag1}
\xymatrix{
{D^{(d)}} \ar[r]^{f} \ar[d]_{q} &
 {\sf N}^{(d)}  \ar[d]_{p_{d}} \\
0 \in D \ar[r]^{\sigma} &  ~~{N_{d}} \owns c\rlap{,} \\
}
\eeqa
where  $p_{d}\colon{\sf N}^{(d)} \to N_{d}$ is the {\em minimal resolution of singularities} of $N_{d}$ obtained by successively blowing up the singularity, with the exceptional divisor $E^{(d)} = p_{d}^{-1}(c)$ having $d-1$ rational components, and 
\beqa
D^{(d)} := \left(D \times _{N_{d}} {\sf N}^{(d)} \right)_{\red}.
\eeqa

The {\em closed fibre} $F^{(d)}$ of the canonical morphism ${\sf N}^{(d)} \to \spec(k\llbracket t \rrbracket)$ looks like
\beqa\label{closedfibre}
F^{(d)} = E^{(d)} \cup E(1) \cup E(2),
\eeqa
where $E(1)$ and $E(2)$ are the inverse images of the branches corresponding to $x = 0$ and $y = 0$ in $N_{d}$.

Thus, $F^{(d)} \subset {\sf N}^{(d)}$ is a normal crossing divisor with $d+1$ components. 
By \cite[Proposition 2.4]{gonverd},  the morphism  
\beqa\label{platif}
f\colon D^{(d)} \lra {\sf N}^{(d)}
\eeqa
is {\em finite and flat, the minimal platificateur in the sense of Grothendieck}; see \cite[Corollary~7, p.~448]{gonverd}. Since ${\sf N}^{(d)}$ is smooth, this implies that $f$ is ramified at the generic point of each of the $d-1$ rational components of the exceptional divisor $E^{(d)} = p_{d}^{-1}(c) \subset F^{(d)}$.

Let
\beqa\label{basictorsor}
\eT_{D} \simeq D \times^{\rho} G
\eeqa 
be the trivial $({\sf\mu_{d}},G)$-torsor  on $D$ (see \eqref{keydiag1}) with a ${\sf\mu_{d}}$-structure given by a homomorphism $\rho\colon{\sf\mu_{d}} \to G$.  This gives a homomorphism $\rho\colon{\sf\mu_{d}} \to T$ into the maximal torus $T$ of $G$. We fix once and for all an isomorphism $T \simeq \mathbb{G}_{m}^{\ell}$. Thus $\rho$ determines an ordered pair of integers modulo $d$ called the {\em type} $\tau = (a_1, a_2, \ldots, a_{\ell})$. More precisely, we have a ${\sf\mu_{d}}$-action on $D \times G$,  given by 
\beqa\label{localaction0}
\gamma\cdot (u,v,g) = \left(\zeta\cdot u, \zeta^{-1}\cdot v, \rho(\gamma)\cdot g\right). 
\eeqa
Let $z_{1}$ and $z_{2}$ be the two points in $D^{(d)}$ above the origin $0 \in D$ where the normalization of the curve $u\cdot v = 0$ meets the fibre $q^{-1}(0)$. We observe that  the action of ${\sf\mu_{d}}$ is {\em balanced}  at these two marked points; \textit{i.e.} the action of a generator $\zeta$ on the tangent spaces to each branch are inverses to each other.  For the corresponding dual action in the neighbourhood $D'_{0}$ (the component with local coordinate $v$),  the action is by $\zeta^{-1}$. If we begin with a representation $\rho\colon{\sf\mu_{d}} \to G$ of local type $\tau$ at a point in a branch, then the corresponding local type for the dual action at the point in the second branch is denoted by $\bar{\tau}$.

Consider the  adjoint group scheme $\eT_{D}(G):= \eT_{D} \times ^{G,\Ad} G$ on $D$, where $G$ acts on itself by inner conjugation.  We define the equivariant group scheme
\beqa\label{egtau}
E(G, \tau):= q^{*}\left(\eT_{D}(G)\right)
\eeqa
on $D^{(d)}$ of local type $\tau$ in the sense that it comes with a ${\sf\mu_{d}}$-action via a representation $\rho\colon{\sf\mu_{d}} \to G$. 
Since the morphism $f\colon D^{(d)} \to {\sf N}^{(d)}$ is also {\em finite and flat},  we can take the Weil restriction of scalars 
\beqa\label{localnsgrpscheme1}
f_{*}\left(E(G, \tau)_{D^{(d)}}\right):= {\Res}_{{D^{(d)}}/{\sf N}^{(d)}} \left(E(G, \tau\right)_{D^{(d)}}),
\eeqa
and since $E(G, \tau)_{D^{(d)}} \to D^{(d)}$ is a smooth (affine) group scheme,  the basic properties of Weil restriction of scalars (\textit{cf.} \cite[Lemma 2.2]{edix}) show that $f_{*}(E(G, \tau))$ is a smooth group scheme on ${\sf N}^{(d)}$ together with a ${\sf\mu_{d}}$-action. By taking  invariants under the action of ${\sf\mu_{d}}$ and noting that we are over characteristic zero, by \cite[Proposition 3.4]{edix}, we obtain the smooth (affine) group scheme on ${\sf N}^{(d)}$ obtained by taking {\em invariant direct images}:
\beqa\label{nsgrpscheme1.5}
\qh^{G}_{\tau,{\sf N}^{(d)}} = (f^{\sf\mu_{d}}_{*}) ((E(G, \tau)_{D^{(d)}})) 
\eeqa 
(see \cite[Definition 4.1.3]{base}). 
The  {\tt 2BT}-group scheme of type $\tau$ with generic fibre $G$ of singularity type $\text{A}_{d}$  associated to $\theta_{\tau}$ is defined to be the affine group scheme $\qh^{G}_{\tau,{\sf N}^{(d)}}$ from \eqref{nsgrpscheme1.5} on the regular surface ${\sf N}^{(d)}$. This process defines  a distinguished collection of {\tt 2BT}-group scheme $\big\{{\qh^{G}_{\tau,{\sf N}^{(d)}}}\big\}_{\tau}$ indexed by the type $\tau$.

\subsection{The McKay correspondence revisited}\label{mckaystory} We recall the geometric interpretation of the McKay correspondence given by Gonz\'alez-Sprinberg and Verdier (see \cite{gonverd}).  Let $\Irr^{o}({\sf\mu_{d}}) \subset \Irr({\sf\mu_{d}})$ be the non-trivial irreducible representations of ${\sf\mu_{d}}$, and let $\Irr(E^{(d)})$ denote the set of irreducible rational components of  the exceptional divisor  of the minimal resolution $p_{d}\colon{\sf N}^{(d)} \to N_{d}$ from \eqref{keydiag1}.  Let $\psi$  be a non-trivial character of ${\sf\mu_{d}} = \langle \gamma \rangle$. Then $\psi$ corresponds to $\zeta \mapsto \zeta^{s}$, where $\zeta$ is the primitive $\supth{d}$-root of unity chosen above and $1 \leq s \leq d-1$. Let $L_{\psi}$ be the equivariant line bundle  on $D$ where ${\sf\mu_{d}}$ acts on $D \times k$ as $\gamma\cdot (u,v, a) = (\zeta\cdot u, \zeta^{d -1}\cdot v, \zeta^{s}\cdot a)$, $a \in k$.
A ${\sf\mu_{d}}$-invariant section \text{\cursive h} of this line bundle is given by the relation $\{\gamma\cdot \text{\cursive h}\}(u,v) = \text{\cursive h} (\gamma\cdot (u,v))  = \zeta^{s} \text{\cursive h}(u,v)$, and hence the ${\sf\mu_{d}}$-invariant sections are generated by $u^{s}$ and $v^{d-s}$. 
Let  $\mathcal L_{\psi} := f^{\sf\mu_{d}}_{*}(q^{*}(L_{\psi}))$ be the induced line bundle on $\sf N^{(d)}$. This is a line bundle since $f$ is finite and flat.

{\em Mckay correspondence following Gonz\'alez-Sprinberg and Verdier}. There is a bijection $
\Irr^{o}({\sf\mu_{d}}) \to \Irr(E^{(d)})$, $\psi \mapsto E_{\psi}$,  such that for any $E_{j} \in \Irr(E^{(d)})$, we have
\beqa
c_{1}(\mathcal L_{\psi})\cdot E_{j} =
\begin{cases}
0&\quad\text{if } E_{j} \neq E_{\psi}, \\
1&\quad{\text if }E_{j} = E_{\psi}.
\end{cases} \eeqa
The above statement implies that the first Chern class $c_{1}(\mathcal L_{\psi})$ can be represented by an {\sl effective} divisor $\delta_{\psi} \subset \sf N^{(d)}$ which meets $F^{(d)}$ transversally at a unique point   which lies in $E_{\psi}$. 

We interpret this on the side of the surface  $D^{(d)}$. Consider the reduced fibre $\tilde{E} :=  q^{-1}(0)_{\red}$. The group ${\sf\mu_{d}}$-fixes the divisor $q^{-1}(0)$ and hence its reduced subscheme $\tilde{E}$. There is a smooth curve $\delta'_{\psi} \in D^{(d)}$ which meets the divisor $q^{-1}(0) \subset D^{(d)}$ at a unique component $\tilde{E_{\psi}}$ of $\tilde{E}$. Locally at this point of $\tilde{E_{\psi}}$, the transversal curve  $\delta'_{\psi}$ is given by the invariant section $\text{\cursive h}$ of  $q^{*}(\mathcal L_{\psi})$. The action of ${\sf\mu_{d}}$ on  $\text{\cursive h}$ shows that the group ${\sf\mu_{d}}$ acts on the local uniformizer $z_{\psi}$ of the component $\tilde{E_{\psi}}$ by $\gamma\cdot z_{\psi} \mapsto \zeta^{s}z_{\psi}$.

Let us explain how to a representation $\rho\colon{\sf\mu_{d}} \to T$ of type $\tau$, we associate a rational point $\theta_{\tau}$ of $\cA_T$ in the fundamental domain of $Y(T)$ as in \cite{base}, and not just in the alcove $\mathbf{a}_0$. Let $\rho\colon{\sf\mu_{d}} \to T$ map the generator $\gamma$ of $\sf\mu_{d}$ to $(\zeta_{d}^{a_1}, \ldots, \zeta_{d}^{a_{\ell}})$, where $\zeta_{d}$ is a primitive $\supth{d}$ root of unity and the $a_i$ are integers uniquely determined modulo $d$. We say that $\rho$ is  of  {\em type} $\tau = (a_1, a_2, \ldots, a_{\ell})$.  This determines a rational point in $\frac{1}{d} Y(T) \mod Y(T)$ and therefore a rational point $\theta_{\tau}$ in the fundamental domain of $Y(T)$ in $\cA_T$. Let $F^{(d)} = E^{(d)} \cup E(1) \cup E(2)$ be as in \eqref{closedfibre}, and let the two nodal end points be ${\bf z}_{i}$, $i = 1,2$, and the nodes where the rational components meet be ${\bf y}_{s}$, $s = 1, \ldots, d-2$ for $d \geq 3$.

Consider the canonical morphism ${\sf N}^{(d)} \to \spec(k\llbracket t \rrbracket)$. Let ${\sf N}^{(d)}_{t} \subset {\sf N}^{(d)}$ be the open subset defined by the non-vanishing of $t$. For $s = 1, \ldots, d-2$, let $C_{s} \subset {\sf N}^{(d)}$ be a smooth curve with generic point in ${\sf N}^{(d)}_{t}$ and containing ${\bf y}_{s}$ as a closed point. Let $U_{s}$ be the formal neighbourhood of ${\bf y}_{s}$ in $C_{s}$. Similarly, let $C_{{\bf z}_{i}}$ for $i = 1,2$ denote smooth curves with generic point in ${\sf N}^{(d)}_{t}$ and containing ${\bf z}_{i}$ as a closed point. Let $U_{{\bf z}_{i}}$ be the formal neighbourhood of ${\bf z}_{i}$ in $C_{{\bf z}_{i}}$.

\bth\label{nsdescrip}
The group scheme $\qh^{G}_{\tau,{\sf N}^{(d)}}$ is affine. It has the following description at the generic points of the rational curves, at $U_{{\bf z}_{i}}$ and at $U_{s}$: 
\begin{enumerate} 
\item Let $\psi$  denote the character which takes $\zeta$ to $\zeta^{s}$. We read the tuple $(s\cdot a_1, s\cdot a_2, \ldots, s\cdot a_{\ell})$ modulo $d$. Let $\tau_{s} := (s\cdot a_1, s\cdot a_2, \ldots, s\cdot a_{\ell})$. At the generic point of the rational component ${E_{\psi}}$, the local type of the restriction of $\qh^{G}_{\tau,{\sf N}^{(d)}}$ is given by the point $\theta_{\tau_{s}}$ in the fundamental domain of\, $Y(T)$. 
 
\item The restrictions of  $\qh^{G}_{\tau,{\sf N}^{(d)}}$ to $U_{{\bf z}_{i}}$ for $i=1,2$ are isomorphic to the  parahoric group schemes $\mathfrak G_{\theta_{\tau}}$ and $\mathfrak G_{\theta_{\bar\tau}}$ corresponding to $\tau $ and $\bar\tau$, respectively.
\item  The restriction of  $\qh^{G}_{\tau,{\sf N}^{(d)}}$ to $U_{s}$ is isomorphic to the Bruhat--Tits group scheme $\mathfrak G_{f_{s}}$, where $f_{s}\colon\Phi \to \mathbb R$ is the concave function given by 
  \begin{equation}
    f_{s}(r) := m_r\left(\theta_{\tau_{s}}\right) + m_r\left(\theta_{\tau_{s+1}}\right).
\end{equation}
\end{enumerate}  
\eeth

\begin{proof} This is an immediate consequence of Theorems~\ref{n-paragrpsch} and~\ref{Artin-Weil-Kawamata}.
\end{proof} 

\begin{rem} \label{balajiproctype3}
Let us assume that the residue field $k$ contains primitive $\supth{8}$ (resp.\ $\suprd{3}$) roots of unity. Let $\{ \alpha, \beta \}$ denote the short and long roots of $B_2$ (resp.\ $G_2$). Let us fix an isomorphism of $T$ with $\mathbb{G}_m^2$. Since the coroot lattice equals $Y(T)$, let us take the coroot homomorphisms $\alpha^{\vee}, \beta^{\vee}\colon \mathbb{G}_m \rightarrow T$ as the first and second coordinate maps. Let us take $\rho$ of type $\tau_{1}$ such that $\theta_{\tau_1}$ equals $\frac{\theta_{\alpha}}{2}$ for $G$ of type $B_{2}$ (resp.\ $\theta_{\alpha}$ for $G$ of type $G_{2}$). Then $\tau_{2}$ gives $\theta_{\alpha}$ (resp.\ $2 \theta_{\alpha}$). By Example~\ref{B2type3} (resp.~\ref{G2type3}), we see that for $s=1$, the function $f_{s}$  is of type III. This shows that the limiting objects of \cite{balaproc} could become torsors under BT-group schemes of the most general type.
\end{rem}

\renewcommand\thesection{\Alph{section}}
\setcounter{section}{0}

\section*{Appendix on parabolic and equivariant bundles}\label{parabstuff}
\addcontentsline{toc}{section}{Appendix on parabolic and equivariant bundles}
\refstepcounter{section}

\setcounter{subsection}{0}

In this section we recall and summarize some results on parabolic bundles and equivariant bundles on Kawamata covers. These play a central role in the constructions of the Bruhat--Tits group schemes made above. Consider a pair $(X,D)$, where $X$ is a smooth quasi-projective variety and $D = \sum_{j = 0}^{\ell} D_{j}$ is a {\em reduced normal crossing divisor}  with non-singular components $D_{j}$  intersecting each other transversely. The basic  examples we have in mind  are discrete valuation rings with their closed points, the wonderful compactification ${\bf A}$ with its boundary divisors or affine toric varieties.

Let $E$ be a locally free sheaf on $X$.  Let $n_{j}$, $j = 0 \ldots, \ell$, be positive integers attached to the components $D_{j}$. Let $\xi$ be a generic point of $D$. 

Let $E_{\xi} := E \otimes_{\mathcal O_{X}} \mathcal O_{X,\xi}$ and $\bar{E}_{\xi} := E_{\xi}/\mathfrak m_{\xi} E_{\xi}$, and let $\mathfrak m_{\xi}$ be the maximal ideal of $\mathcal O_{X,\xi}$.

\begin{defi}\label{sestype} A (generic) parabolic structure on $E$ consists of the following data:
\begin{itemize}
\item a flag $\bar{E}_{\xi} = F^{1}\bar{E}_{\xi} \supset \cdots \supset F^{r{_{j}}}\bar{E}_{\xi}$ at the generic point $\xi$ of each of the components $D_{j}$ of $D$; 
\item weights $d_{s}/n_{j}$, with $0 \leq d_{s} < n_{j}$, attached to $F^{s}\bar{E}_{\xi} $  such  that  $d_1 < \cdots < d_{r_j}$.
\end{itemize}
\end{defi}
By saturating the flag datum on each of the divisors, we get 
for each component $D_{j}$ a filtration 
\beqa
E_{D_{j}} = F^{1}_{j} \supset \cdots \supset F^{r{_{j}}}_{j}
\eeqa
of subsheaves on $D_{j}$. Define the coherent subsheaf $\mathcal F^{s}_{j}$, where $0 \leq j \leq \ell$ and $1 \leq s \leq r_{j}$, of $E$ by
\beqa
0 \lra \mathcal F^{s}_{j} \lra E \lra E_{D_{j}}/F^{s}_{j} \lra 0; 
\eeqa
the last map is by restriction to the divisors. 

We wish to emphasize that this definition, where the quasi-parabolic structure is given only at the generic points of the divisors, is somewhat different from what is done in \cite{biswas}.

In \cite{ms}, working over curves, two essential things are shown:
\begin{enumerate}
\item Given a parabolic bundle $E$ as above,  there exist a ramified covering of $p\colon(Y,\tilde{D}) \to (X,D)$  with suitable ramification data and Galois group $\Gamma$ together with a $\Gamma$-equivariant vector bundle $V$ on $Y$ such that the invariant direct image sheaf $p^{\Gamma}_{*}(V)$ equals $E$ and furthermore $V$ also recovers the parabolic structure, namely the  filtrations $\mathcal F^{s}_{j}$ and weights, on $E$.
\item Conversely, if we begin with an equivariant bundle $V$ on $Y$, then the invariant direct image $p^{\Gamma}_{*}(V)$ gives a vector bundle $E$ on $X$ with parabolic structures at the generic points of the height $1$ prime ideals.
\end{enumerate}

In this appendix we work with the higher-dimensional and generic variant we have defined above and establish analogous results under suitable conditions. This is the key fact that is used in the present paper.

For the sake of completeness, we give a self-contained {\em ad hoc} argument for this construction which is more in the spirit of the present note. The data given in Definition~\ref{sestype} is as in \cite{ms},   which deals with points on curves. {\em Note that in our setting we have rational weights}. Under these conditions, we have a natural  functor 
\beqa\label{bisses}
p^{\Gamma}_{*}\colon \{\Gamma\text{-equivariant~bundles~on}~Y\} \lra \{\text{parabolic~bundles~on}~X\}.
\eeqa
Under suitable conditions relevant to this paper, this functor is in fact a  surjection. See \cite{biswas} for other conditions where we can obtain a surjection. 

As the notation suggests, this is achieved by taking invariant direct images. Since $p\colon Y \to X$ is finite and flat, if $V$ is locally free on $Y$, then so is $p^{\Gamma}_{*}(V)$. An equivariant bundle $V$ on $Y$ is defined in a  formal neighbourhood of a generic point $\zeta$ of a component by a representation of the isotropy group $\Gamma_{\zeta}$ and locally (in the formal sense), the action of $\Gamma_{\zeta}$ is the product action. This is called the local type in \cite{pibundles} or \cite{ms}. By appealing to the $1$-dimensional case, we obtain a canonical (generic) parabolic structure on $E$, \textit{i.e.} parabolic structures  at the generic points $\xi$ of the components $D_{j}$ in $X$. 

Conversely,  given a (generic) parabolic bundle $E$ on $(X,D)$, to get the $\Gamma$-equivariant bundle $V$ on $Y$ such that $p^{\Gamma}_{*}(V) = E$, we heuristically proceed as follows. If we know the existence of such a $V$, then we can consider the inclusion $p^*(E) \subset V$. Taking its dual (and since $V/p^*(E)$ is torsion, taking duals is an inclusion), we get 
\beqa
V^* \longhookrightarrow  (p^*(E))^* = p^*\left(E^*\right).
\eeqa
By the $1$-dimensional case (obtained by restricting to the height $1$ primes at the generic points), we see that the quotient $T_{\zeta} := p^*(E^*)_{\zeta}/V_{\zeta}^*$ is a torsion $\mathcal O_{Y,\zeta}$-module and $T_{\zeta}$ is completely determined by the parabolic structure on $E$. 

The (generic) parabolic structure on $E$ therefore determines canonical quotients
\beqa\label{quotients}
p^*\left(E^*\right)_{\zeta} \lra T_{\zeta}
\eeqa
for each generic point $\zeta$ of components in $Y$ above the components $D_{j}$ in $X$.

The discussion above suggest how one would construct such a $V$;  we begin with these quotients \eqref{quotients}. Then we observe that there is a maximal coherent subsheaf $V' \subset p^*(E^*)$ such that 
\beqa
p^*\left(E^*\right)_{\zeta}/V_{\zeta}' = T_{\zeta}.
\eeqa
Away from $p^{-1}(D)$ in $Y$, the inclusion $V' \hra p^*(E^*)$ is an isomorphism. Dualizing again, we get an inclusion $p^*(E) \hra (V')^{*}$ which is an isomorphism away from $p^{-1}(D)$. Set $W =  (V')^{*}$. Since $W$ is the dual of a coherent sheaf, it is {\em reflexive}. 

In situations as in Theorem~\ref{Artin-Weil-Kawamata}, we find examples of vector bundles with Lie algebra structures which satisfy the surjectivity of \eqref{bisses}.

\subsection {Kawamata coverings}\label{kawa}
Suppose we are given positive integers $n_{0}, \ldots, n_{\ell}$. Let $X$ be a smooth quasi-projective variety over a perfect field $k$ whose characteristic $\charr(k)$ is coprime to $n_{0}, \ldots, n_{\ell}$. Let $D$ be a simple or reduced normal crossing divisor with
 decomposition $D = \sum_{i=0}^{\ell} D_i$ into its smooth components intersecting transversally.

By a ``Kawamata covering'' of $X$, we mean the existence   
of a connected smooth quasi-projective
variety $Z$ over $k$ and a Galois covering morphism
\beqa\label{kawamatacm}
\kappa\colon Z \lra  X 
\eeqa
such that the reduced divisor $\kappa^{*}{D}:= \,({\kappa}^{*}D)_{\red}$
is a normal crossing divisor on $Z$ and, furthermore,
${\kappa}^{*}D_{i}= n_{i}.({\kappa}^{*}D_{i})_{\red}$. Let $\Gamma$ denote the Galois group
for $\kappa$. 

The ``covering lemma'' of Y. Kawamata
(see \cite[Lemma 2.5, p.~56]{vieweg}) says that
  there is such a covering under the assumption of $k$ being algebraically closed.
  
We however note that if the base is reasonably simple, such as an affine space $\mathbb A^{n}_k$, then such a covering will exist with just the assumption of $k$ being perfect.  

The isotropy group of any point $z \in Z$, for the
action of $\Gamma$ on $Z$, will be denoted by ${\Gamma}_{z}$.  It is easy to see that the stabilizers at generic points of the irreducible components of $(\kappa^{*}D_i)_{\red}$ are cyclic of order $n_{i}$.

\brem (In positive characteristics) The Kawamata covering lemma is seen to hold under {\em tameness assumptions}; \textit{i.e.} if the characteristic $p$ is coprime to the $n_{i}$, then it holds; see \cite[Lemma 2.5, p.~56]{vieweg}.
\erem

\subsection{The group scheme situation}\label{grstuff} Let $(X,D)$ be as above. Let $\xi$ be the generic point of a component $D_{j}$ of $D$, and let $A := \mathcal O_{X,\xi}$ and $K$ be the quotient field of $A$. {\em We always assume that these group schemes are generically split}. Let us begin by stating some results from \cite{base} in the simplest situation.  Then we will state the more general case.

Let $\mathfrak G_{\bvt}$ be a Bruhat--Tits group scheme on $\spec(A)$ associated to a vertex $\theta_{\alpha}$ of the Weyl alcove $\mathbf{a}_0$ (see Section~\ref{liedata}). Let $B = \mathcal O_{Y, \zeta}$, where $\zeta$ is the generic point of a component of $Y$ above $D_{j}$, and let $L$ be the quotient field of $B$. We assume that the local ramification data for the Kawamata covering has  numbers $d_{\alpha}$; see \eqref{dalpha}.  Let $\Gamma_{\zeta}$ be the stabilizer of $\Gamma$ at $\zeta \in Y$. The results of \cite[Proposition 5.1.2 and Remark 2.3.3]{base} show that there exists an equivariant group scheme $\mathcal H_{B}$ on $\spec(B)$ with fibre isomorphic to the simply connected group $G$ and such that $\Res_{B/A}(\mathcal H_{B})^{\Gamma_{\zeta}} \simeq \mathfrak G_{\bvt}$. 

Although in \cite{base} the above is stated for points in the alcove $\mathbf{a}_0$, the proofs reveal that actually they hold more generally. Under assumptions on the residue field (see Sections~\ref{tameness} and~\ref{charassum}), the proofs of these results generalize in a straightforward way for an arbitrary complete DVR (see  Section~\ref{bstomixed}). We may further assume (see Remark~\ref{uptoyt}) that $\theta$ is a rational point of $\mathcal{A}_T$ in the fundamental domain of $Y(T)$ and not just a point in the alcove $\mathbf{a}_0$ (see Section~\ref{liedata}).

Suppose that we have a group scheme $\mathfrak G_{X'}$ on an open $X' \subset X$ which includes all the height $1$ primes coming from the divisors $D_{j}$, with the following properties:
\begin{itemize}
\item Away from the divisor $D \subset X$, $\mathfrak G$ is the constant group scheme with fibre $G$.
\item The restrictions $\mathfrak G\mid_{\spec(A)}$ at the generic points $\xi$ are isomorphic to the Bruhat--Tits  group scheme $\mathfrak G_{\theta}$ for varying $\xi$ and $\theta$ varying in the fundamental domain of $Y(T)$ in $\mathcal{A}_T$.
\end{itemize}  
In other words, $\mathfrak G_{X'}$ is obtained by a gluing of the constant group schemes with group schemes $\mathfrak G_{\theta}$ along $\spec(K)$ by an automorphism of the constant group scheme $G_{K}$.

Now consider the inverse image of the constant group scheme $p^{*}(G_{X - D}) \simeq G \times p^*(X-D)$. Then using the gluing on $X'$, we can glue the constant group scheme $p^{*}(G_{X - D})$ with the local group schemes $\mathcal H_{B}$ for each generic point $\zeta$ to obtain a group scheme $\mathcal H_{Y'}$ on $Y' = p^{-1}(X')$ such that
\beqa
\Res_{Y'/X'}\left(\mathcal H_{Y'}\right)^{\Gamma} \simeq \mathfrak G_{X'}.
\eeqa

\newcommand{\etalchar}[1]{$^{#1}$}

\end{document}